\documentclass[reqno]{amsart}
\usepackage{fullpage, verbatim, xy, amssymb}
\usepackage{amsmath}
\usepackage{amssymb}
\usepackage{amsfonts}
\usepackage{graphicx}

\xyoption{all}
\newtheorem{theorem}{Theorem}[section]

\newtheorem{corollary}[theorem]{Corollary}

\newtheorem{lemma}[theorem]{Lemma}

\newtheorem{proposition}[theorem]{Proposition}
\newtheorem{remark}[theorem]{Remark}

\input{tcilatex}
\newcommand{\ev}{\operatorname{ev}}
\begin{document}
\title{Poincar\'e duality isomorphisms in tensor categories}
\author{Marc Masdeu}
\email{m.masdeu@warwick.ac.uk}
\address{Department of Mathematics, University of Warwick, Coventry, United Kingdom}
\author{Marco Adamo Seveso}
\email{seveso.marco@gmail.com}
\address{Dipartimento di Matematica "Federigo Enriques", Universit\`a di Milano, Via Cesare Saldini 50, 20133 Milano, Italia}
\begin{abstract}
If for a vector space  $V$ of dimension $g$ over a characteristic zero field we denote by $\wedge^iV$ its alternating powers, and by $V^\vee$ its linear dual, then there are natural Poincar\'e isomorphisms:
\[
\wedge^i V^\vee \cong \wedge^{g-i} V.
\]
We describe an analogous result for objects in rigid pseudo-abelian $\mathbb{Q}$-linear ACU tensor categories.
\end{abstract}
\maketitle
\tableofcontents

\section{Introduction}

Let $V$ be a vector space of finite dimension $g$ over a characteristic zero
field, let $\mathbb{I}$ be the field of scalars, viewed as a vector space,
and consider the alternating algebra $\wedge ^{\cdot }V$. Then the internal
multiplication morphism defined by the formula%
\begin{equation*}
\iota _{1}\left( x\right) \left( \omega _{1}\wedge ...\wedge \omega
_{j}\right) :=\tsum\nolimits_{k=1}^{j}\left( -1\right) ^{j}\left\langle
x,\omega _{k}\right\rangle \omega _{1}\wedge ...\wedge \widehat{\omega }%
_{k}\wedge ...\wedge \omega _{j}
\end{equation*}%
gives a map%
\begin{equation*}
\iota _{1}:V\rightarrow Hom\left( \wedge ^{\cdot }V^{\vee },\wedge ^{\cdot
-1}V^{\vee }\right)
\end{equation*}%
valued in the space of degree $-1$ anti-derivation. Since $\iota _{1}\left(
x\right) ^{2}=0$, by the universal property of the alternating algebra the
morphism $\iota _{1}$ extends to a morphism of algebras%
\begin{equation*}
\iota :\wedge ^{\cdot }V\rightarrow Hom\left( \wedge ^{\cdot }V^{\vee
},\wedge ^{\cdot }V^{\vee }\right) ^{op}\text{,}
\end{equation*}%
where $\left( \cdot \right) ^{op}$ means the opposite algebra, such that $%
\iota \left( x\right) :\wedge ^{j}V^{\vee }\rightarrow \wedge ^{j-i}V^{\vee
} $ if $x\in \wedge ^{i}V$ and $j\geq i$ (and it is zero otherwise). In
order to match with the notations employed in the paper, it will be
convenient to define, for every $j\geq i$:%
\begin{equation*}
\iota _{i,j}\left( x\right) :=\frac{\left( j-i\right) !}{j!}\iota \left(
x\right) _{\mid \wedge ^{j}V^{\vee }}:\wedge ^{j}V^{\vee }\rightarrow \wedge
^{j-i}V^{\vee }\text{, if }x\in \wedge ^{i}V\text{.}
\end{equation*}%
This gives morphisms $\iota _{i,j}:\wedge ^{i}V\rightarrow Hom\left( \wedge
^{j}V^{\vee },\wedge ^{j-i}V^{\vee }\right) $ with the following property.
If we identify $\wedge ^{\cdot }V^{\vee }\simeq \left( \wedge ^{\cdot
}V\right) ^{\vee }$ by means of%
\begin{equation}
\ev_{V,a}^{i}:\wedge ^{i}V^{\vee }\otimes \wedge ^{i}V\rightarrow \mathbb{I}
\label{Intro F0}
\end{equation}%
obtained by the natural inclusions $\wedge ^{i}\left( \cdot \right)
\hookrightarrow \otimes ^{i}\left( \cdot \right) $ followed by the perfect
pairing
\[
\ev_{V}^{i}\left( \omega _{1}\otimes ...\otimes \omega
_{i},x_{1}\otimes ...x_{i}\right) :=\tprod\nolimits_{k=1}^{i}\left\langle
\omega _{k},x_{k}\right\rangle,
\]
then%
\begin{equation}
\ev_{V,a}^{j}\left( \omega _{j},x_{i}\wedge x_{j-i}\right)
=\ev_{V,a}^{j-i}\left( \iota _{i,j}\left( x_{i}\right) \left( \omega
_{j}\right) ,x_{j-i}\right) \text{ for }x_{i}\in \wedge ^{i}V\text{, }%
x_{j-i}\in \wedge ^{j-i}V^{\vee }\text{ and }\omega _{j}\in \wedge
^{j}V^{\vee }\text{,}  \label{Intro F1}
\end{equation}%
meaning that $\iota _{i,j}\left( x_{i}\right) :\wedge ^{j}V^{\vee
}\rightarrow \wedge ^{j-i}V^{\vee }$ is dual to the multiplication map $%
x_{i}\wedge \cdot :\wedge ^{j-i}V\rightarrow \wedge ^{j}V$.

These internal multiplications morphisms allow for the definition of the Poincar\'e morphism%
\begin{equation*}
D^{i,g}:\wedge ^{i}V\overset{\iota _{i,g}}{\rightarrow }Hom\left( \wedge
^{g}V^{\vee },\wedge ^{g-i}V^{\vee }\right) \simeq \wedge ^{g-i}V^{\vee }
\end{equation*}%
and using reflexivity after dualizing yields%
\begin{equation*}
D_{i,g}:\wedge ^{i}V^{\vee }\overset{\iota _{i,g}}{\rightarrow }Hom\left(
\wedge ^{g}V,\wedge ^{g-i}V\right) \simeq \wedge ^{g-i}V\text{.}
\end{equation*}%
As it it well known one has%
\begin{equation}
D_{g-i,g}\circ D^{i,g}=\left( -1\right) ^{i\left( g-i\right) }\binom{g}{g-i}%
^{-1}\text{ and }D^{i,g}\circ D_{g-i,g}=\left( -1\right) ^{i\left(
g-i\right) }\binom{g}{i}^{-1}  \label{Intro F2}
\end{equation}

If the category of finite dimensional vector spaces is replaced by a
more general neutral tannakian category, the fibre functor allows to extend
this result to this category due to $\left( \text{\ref{Intro F2}}\right) $
and the existence of a faithful exact linear functor valued in the category
of vector spaces, once the appropriate definition of the Poincar\'{e}
morphism is given in such a way that it is preserved by tensor functors. The
aim of this paper is to generalize this result to rigid pseudo-abelian and $%
\mathbb{Q}$-linear $ACU$ tensor categories, with the aim of applications to
Chow motives, and prove the analogue statement for the symmetric algebras $%
\vee ^{\cdot }V$.

Suppose indeed that $V$ is a supervector space of odd degree. Then the same
formalism applies, replacing the alternating algebra with the symmetric
algebra: the reason is that, by definition, the commutativity constraint $%
\tau _{V,W}:V\otimes W\rightarrow W\otimes V$ in the category of supervector
spaces is given by $\tau _{V,W}\left( x\otimes y\right) =-\left( y\otimes
x\right) $ if $V$ and $W$ have odd degree and, hence, the symmetrizer
operates as an anti-symmetrizer on the underlying vector spaces.

\bigskip

The viewpoint taken in this paper is to use $\left( \text{%
\ref{Intro F1}}\right) $ as the defining property of the internal
multiplication morphisms. Suppose that $\mathcal{C}$ is a rigid
pseudo-abelian and $\mathbb{Q}$-linear $ACU$ tensor category with identity
object $\mathbb{I}$\ and that we are given $V\in \mathcal{C}$ of rank $r\in
End\left( \mathbb{I}\right) $. If $A_{\cdot }$ denotes one of the
alternating or symmetric algebras, the data of the multiplication morphisms $%
\varphi _{i,j}:A_{i}\otimes A_{j}\rightarrow A_{i+j}$ is equivalent to that
of the associated morphisms $f_{i,j}:A_{i}\rightarrow \hom \left(
A_{j},A_{i+j}\right) $. When $j\geq i$, we may consider the composite%
\begin{equation*}
\iota _{i,j}:A_{i}\overset{f_{i,j-i}}{\rightarrow }\hom \left(
A_{j-i},A_{j}\right) \overset{d}{\rightarrow }\hom \left( A_{j}^{\vee
},A_{j-i}^{\vee }\right)
\end{equation*}%
where $d:\hom \left( X,Y\right) \rightarrow \hom \left( Y^{\vee },X^{\vee
}\right) $ is the internal duality morphism as defined in \S 2. Next we
define%
\begin{equation*}
D^{i,j}:A_{i}\overset{\iota _{i,j}}{\rightarrow }\hom \left( A_{j}^{\vee
},A_{j-i}^{\vee }\right) \overset{\alpha ^{-1}}{\rightarrow }A_{j-i}^{\vee
}\otimes A_{j}^{\vee \vee }\text{,}
\end{equation*}%
where $\alpha :\hom \left( X,Y\right) \rightarrow Y\otimes X^{\vee }$ is the
canonical morphism. Working dually and employing the reflexivity one also
gets%
\begin{equation*}
D_{i,j}:A_{i}^{\vee }\rightarrow A_{j-i}\otimes A_{j}^{\vee }\text{.}
\end{equation*}

We say that $V$ has \emph{alternating} (resp. \emph{symmetric}) rank $g\in 
\mathbb{N}_{\geq 1}$ if $L:=\wedge ^{g}V$ (resp. $L:=\vee ^{g}V$) is
invertible and if $\binom{r+i-g}{i}$\ (resp. $\binom{r+g-1}{i}$) is
invertible in $End\left( \mathbb{I}\right) $ for every $0\leq i\leq g$.
Here, for an integer $k\geq 1$,%
\begin{equation*}
\binom{T}{k}:=\frac{1}{k!}T\left( T-1\right) ...\left( T-k+1\right) \in 
\mathbb{Q}\left[ T\right] \text{ and }\binom{T}{0}=1\text{.}
\end{equation*}%
Then we compute, for every $i\leq g$, the compositions%
\begin{eqnarray*}
&&A_{i}\overset{D^{i,g}}{\rightarrow }A_{g-i}^{\vee }\otimes L\overset{%
D_{g-i,g}\otimes 1_{L}}{\rightarrow }A_{i}^{\vee }\otimes L^{-1}\otimes
L\simeq A_{i}^{\vee }\text{,} \\
&&A_{g-i}^{\vee }\overset{D_{g-i,g}}{\rightarrow }A_{i}\otimes L^{-1}\overset%
{D^{i,g}\otimes 1_{L^{-1}}}{\rightarrow }A_{g-i}^{\vee }\otimes L\otimes
L^{-1}\simeq A_{g-i}^{\vee }
\end{eqnarray*}%
and we prove in Theorem \ref{Alternating algebras T} $\left( 3\right) $\
(resp. Theorem \ref{Symmetric algebras T} $\left( 3\right) $)\ that, when $%
A_{\cdot }=\wedge ^{\cdot }V$ (resp. $A_{\cdot }=\vee ^{\cdot }V$), they are
equal to%
\begin{eqnarray}
&&\left( -1\right) ^{i\left( g-i\right) }\binom{g}{g-i}^{-1}\binom{r-i}{g-i}%
\text{ (resp. }\binom{g}{g-i}^{-1}\binom{r+g-1}{g-i}\text{),}  \notag \\
&&\left( -1\right) ^{i\left( g-i\right) }\binom{g}{i}^{-1}\binom{r+i-g}{i}%
\text{ (resp. }\binom{g}{i}^{-1}\binom{r+g-1}{i}\text{).}  \label{Intro F3}
\end{eqnarray}%
In particular, the multiplication maps $\varphi _{i,g-i}:A_{i}\otimes
A_{g-i}\rightarrow A_{g}$ are perfect pairings for every $0\leq i\leq g$
(see Corollaries \ref{Alternating algebras CT} and \ref{Symmetric algebras
CT}). We remark that the same constants obtained in $\left( \text{\ref{Intro
F2}}\right) $ and, more generally, for odd degree supervector spaces,
matches those in $\left( \text{\ref{Intro F3}}\right) $ when $r=g$ in the
alternating case and, respectively, $r=-g$ in the symmetric case. We say in
this case that $V$ has strong alternating or symmetric rank in these cases.

\bigskip

Some remarks are in order about the range of applicability of our results.
First of all we note that, in general, the alternating or the symmetric rank
may be not uniquely determined. Suppose, however, that we know that there is
a field $K$ such that $r\in K\subset End\left( \mathbb{I}\right) $ admitting
an embedding $\iota :K\hookrightarrow \mathbb{R}$. Then it follows from the
formulas $\mathrm{rank}\left( \wedge ^{k}V\right) =\binom{r}{k}$ and $%
\mathrm{rank}\left( \vee ^{k}V\right) =\binom{r+k-1}{k}$ (see \cite[7.2.4
Proposition]{AKh} or \cite[$\left( 7.1.2\right) $]{De}) that we have $r\in
\left\{ -1,g\right\} $ (resp. $r\in \left\{ -g,1\right\} $) when $V$ has
alternating (resp. symmetric)\ rank $g$. In particular, when $r>0$ (resp. $%
r<0$) with respect to the ordering induced by $\iota $, we deduce that $r=g$
(resp. $r=-g$), so that $g$ is a uniquely determined and $V$ has strong
alternating (resp. symmetric)\ rank $g=r$ (resp. $g=-r$)

We recall that $V$ is Kimura positive (resp. negative) when $\wedge ^{N+1}V=0
$ (resp. $\vee ^{N+1}V=0$) for $N\geq 0$ large enough. In this case, the
formula $\mathrm{rank}\left( \wedge ^{k}V\right) =\binom{r}{k}$ (resp. $%
\mathrm{rank}\left( \vee ^{k}V\right) =\binom{r+k-1}{k}$) implies that $r\in 
\mathbb{Z}_{\geq 0}$ (resp. $r\in \mathbb{Z}_{\leq 0}$)\ and the smallest
integer $N$ such that $\wedge ^{N+1}V=0$ (resp. $\vee ^{N+1}V=0$) is $r$
(resp. $-r$). Furthermore, it is known that in this case, when $End\left( 
\mathbb{I}\right) $ does not have non-trivial idempotents, then $\wedge ^{r}V
$ (resp. $\vee ^{-r}V$) is invertible (see \cite[11.2 Lemma]{Kh}): in other
words $V$ has strong alternating (resp. symmetric) rank $g=r$ (resp. $g=-r$).

In particular, our results applies to the motives $V=h^{1}\left( X\right) $
attached to abelian schemes $X=A$ (see \cite{DM} and \cite{Ku}) or a smooth
complete curve $X=C$ over a field (see \cite{Ki1}), which are known to be
Kimura negative, while products of an even number of such motives are Kimura
positive (see \cite{Ki1} for applications of this notion to the product of
two curves). In the subsequent paper \cite{MS} we will apply these results
in order to get a motive whose realizations affords two copies of odd weight
modular forms on indefinite quaternion algebras. When the quaternion algebra
is split, the construction due to Scholl refines and gives a motive whose
realizations affords modular forms of both even or odd weight (see \cite{Sc}%
). Working over an indefinite division quaternion algebra and employing
ideas which goes back to \cite{JL}, a motive of even weight modular forms
has been constructed in \cite{IS} as the kernel of an appropriate Laplace
operators. The results of this paper will be used in \cite{MS} in order to
show the existence of kernels of Dirac operators which are square-roots of
these Laplace operators; the idea of constructing canonical models for the
various incarnations of two copies of odd weight modular forms from square
roots of the Laplace operators is due, once again, to Jordan and Livn\'{e}%
. However, even for these realizations, it is not possible to canonically
split them in a single copy: this is possible only including a splitting
field for the quaternion algebra in the coefficients, but the resulting
splitting depends on the choice of an identification of the base changed
algebra with the split quaternion algebra.

Finally, we remark that the perfectness of the multiplication maps gives a
Poincar\'{e} duality%
\begin{equation}
A_{i}\simeq \hom \left( A_{g-i},A_{g}\right) \simeq A_{g}\otimes
A_{g-i}^{\vee }\text{.}  \label{Intro F4}
\end{equation}%
Indeed, when $V=h^{1}\left( A\right) $ for an abelian scheme $A$ of
dimension $d$, we have that $h^{2d}\left( A\right) \simeq \mathbb{I}\left(
-d\right) $ is invertible and then it is known that%
\begin{equation*}
\vee ^{i}h^{1}\left( A\right) \simeq h^{i}\left( A\right) \simeq
h^{2d-i}\left( A\right) ^{\vee }\left( -d\right) \simeq h^{2d}\left(
A\right) \otimes h^{2d-i}\left( A\right) ^{\vee }\simeq \vee
^{2d}h^{1}\left( A\right) \otimes \vee ^{2d-i}h^{1}\left( A\right) ^{\vee }%
\text{,}
\end{equation*}%
where the canonical identifications $h^{k}\left( A\right) \simeq \vee
^{k}h^{1}\left( A\right) $ are proved in \cite[Remarks $\left( 3.1.2\right) $
$\left( i\right) $]{Ku}, while $h^{i}\left( A\right) \simeq h^{2d-i}\left(
A\right) ^{\vee }\left( -d\right) $ is proved in \cite{DM} (see also \cite[%
Remarks $\left( 3.1.2\right) $ $\left( i\right) $]{Ku}). This gives a
refinement of the motivic Poincar\'{e} duality which states that, for a
smooth projective scheme $X$ of relative dimension $d$, we have%
\begin{equation}
h\left( X\right) \simeq h\left( X\right) ^{\vee }\left( -d\right) \text{.}
\label{Intro F5}
\end{equation}%
Applying $\left( \text{\ref{Intro F4}}\right) $ to the motive $h^{1}\left(
C\right) $ of a smooth complete curve over a field of genus $e$, which is
Kimura negative of Kimura rank $2e$ with $\vee ^{2e}h^{1}\left( A\right)
\simeq \mathbb{I}\left( -e\right) $\ (by \cite[Theorem 4.2 and Remark 4.5]%
{Ki1}), one gets%
\begin{equation*}
\vee ^{i}h^{1}\left( C\right) \simeq \vee ^{2e}h^{1}\left( C\right) \otimes
\vee ^{2e-i}h^{1}\left( C\right) ^{\vee }\simeq \vee ^{2e-i}h^{1}\left(
C\right) ^{\vee }\left( -d\right)
\end{equation*}%
which however, in this case, is not a refinement of $\left( \text{\ref{Intro
F5}}\right) $. We also mention the fact that it is conjectured in \cite[%
Conjecture 7.1]{Ki1} that Chow motives should be Kimura finite, i.e. they
should be a direct sum of a Kimura positive and a Kimura negative motive.

\bigskip

The paper is organized as follows. In \S 2 we develop a general formalism
of internal multiplication morphisms attached to a pairing $\varphi
:S\otimes X\rightarrow Y$ to be applied to the multiplication morphisms in
some algebra object. In \S 3 we prove the prototype of our Poincar\'{e}
isomorphism, which only depends on the data of $\varphi _{S,X}:S\otimes
X\rightarrow Y$, $\varphi _{X,S}:X\otimes S\rightarrow Y$, $\varphi
_{S^{\vee },X^{\vee }}:S^{\vee }\otimes X^{\vee }\rightarrow Y^{\vee }$, $%
\varphi _{X^{\vee },S^{\vee }}:X^{\vee }\otimes S^{\vee }\rightarrow Y^{\vee
}$ subject to an appropriate commutativity constraint and one involving how
the internal multiplications are related with respect to the Casimir
elements: no associativity constraint is needed for these results. In \S 4
we apply the above results to the case of algebra objects and include
results from \S 2 in order to get what is the effect of the associativity
constraint on internal multiplication morphisms (see Proposition \ref{S1
Algebras P1} and Corollary \ref{S1 Algebras C1}). We also make explicit the
identifications $\wedge ^{\cdot }V^{\vee }\simeq \left( \wedge ^{\cdot
}V\right) ^{\vee }$ and $\vee ^{\cdot }V^{\vee }\simeq \left( \vee ^{\cdot
}V\right) ^{\vee }$ by choosing an appropriate evaluation map as in $\left( 
\text{\ref{Intro F0}}\right) $, as useful for the subsequent computations.
In \S 5 we prove the results for the alternating algebras and in \S 6 we
state the results in the symmetric case, the proof being entirely analogous.
The key property relating the internal multiplication morphisms with the
Casimir elements which is needed to apply the formal Poincar\'{e}
isomorphism of \S 3 is proved in \S 5.1 and the proof requires, besides the
two properties of Corollary \ref{S1 Algebras C1}, the anti-derivation (resp.
derivation) property in case $A_{\cdot }=\wedge ^{\cdot }V$ (resp. $A_{\cdot
}=\vee ^{\cdot }V$) which is verified in \S 5. We also prove various
compatibilities of these Poincar\'{e} morphisms in Theorem \ref{Alternating
algebras T} and Proposition \ref{Alternating algebras P2} in the alternating
case, while the corresponding results in the symmetric case are given in
Theorem \ref{Symmetric algebras T} and Proposition \ref{Symmetric algebras
P2}. These further results will be crucial for the applications given in 
\cite{MS}.

\section{Linear algebra in tensor categories}

In the first part of this paper we let $\mathcal{C}$ be an $ACU$ additive $%
\otimes $-biadditive category with unit object $\left( \mathbb{I},l,r\right) 
$ and internal homs. We will usually not write the associativity or unitary
object constraints explicitly, while the commutativity constraint will be
usually denoted by $\tau _{X,Y}:X\otimes Y\rightarrow Y\otimes X$ or by
labeling the positions which are switched, such as $\tau _{1,2}\otimes
1_{Z}=\tau _{X,Y}\otimes 1_{Z}:X\otimes Y\otimes Z\rightarrow Y\otimes
X\otimes Z$ or $\tau _{1,2\otimes 3}=\tau _{X,Y\otimes Z}:X\otimes Y\otimes
Z\rightarrow Y\otimes Z\otimes X$.

\bigskip

To fix notations we recall that the existence of internal homs means that,
if $X,Y\in \mathcal{C}$ there is $\hom \left( X,Y\right) \in \mathcal{C}$
such that%
\begin{equation}
Hom\left( S,\hom \left( X,Y\right) \right) =Hom\left( S\otimes X,Y\right)
\label{S1 F1}
\end{equation}%
holds as contravariant functors on $\mathcal{C}$. Taking $S=\hom \left(
X,Y\right) $ and $1_{\hom \left( X,Y\right) }$ yields%
\begin{equation*}
\ev_{X,Y}:\hom \left( X,Y\right) \otimes X\rightarrow Y
\end{equation*}%
such that $f:S\rightarrow \hom \left( X,Y\right) $ uniquely corresponds to%
\begin{equation*}
\varphi _{f}:S\otimes X\overset{f\otimes 1_{X}}{\rightarrow }\hom \left(
X,Y\right) \otimes X\overset{\ev_{X,Y}}{\rightarrow }Y
\end{equation*}%
under the identification $\left( \text{\ref{S1 F1}}\right) $. The opposite
evaluation is the composite%
\begin{equation*}
\ev_{X,Y}^{\tau }:X\otimes \hom \left( X,Y\right) \overset{\tau _{X,\hom
\left( X,Y\right) }}{\rightarrow }\hom \left( X,Y\right) \otimes X\overset{%
\ev_{X,Y}}{\rightarrow }Y
\end{equation*}%
and $\left( \hom \left( X,Y\right) ,\ev_{X,Y}^{\tau }\right) $ represents $%
Hom\left( X\otimes S,Y\right) $. Then $\left( \hom \left( X,Y\right)
,\ev_{X,Y}\right) $, uniquely determined up to a unique isomorphism, is
called an internal hom pair for $\left( X,Y\right) $ and, when $Y=\mathbb{I}$%
, we write:
\[
\left( \hom \left( X,Y\right) ,\ev_{X,Y}\right) =\left( X^{\vee
}, \ev_{X}\right),\quad\left( \hom \left( X,Y\right) ,\ev_{X,Y}^{\tau }\right)
=\left( X^{\vee },\ev_{X}^{\tau }\right)
\]
and we call $\left( X^{\vee
},\ev_{X}\right) $ a dual pair for $X$.

We remark that $\hom \left( X,Y\right) $ is a bifunctor, contravariant in
the first variable and covariant in the second variable as follows. If $%
f:X_{2}\rightarrow X_{1}$ and $g:Y_{1}\rightarrow Y_{2}$ we define%
\begin{equation*}
\hom \left( f,g\right) :\hom \left( X_{1},Y_{1}\right) \rightarrow \hom
\left( X_{2},Y_{2}\right)
\end{equation*}%
as the unique morphism making the following diagram commutative:%
\begin{equation}
\xymatrix{ \hom\left(X_{1},Y_{1}\right)\otimes X_{2} \ar[r]^-{1_{\hom\left(X_{1},Y_{1}\right)}\otimes f} \ar[d]_{\hom\left(f,g\right)\otimes1_{X_{2}}} & \hom\left(X_{1},Y_{1}\right)\otimes X_{1} \ar[d]^{g\circ \ev_{X_{1},Y_{1}}} \\ \hom\left(X_{2},Y_{2}\right)\otimes X_{2} \ar[r]^-{\ev_{X_{2},Y_{2}}} & Y_{2}\text{.} }
\label{S1 FHom(f,g)}
\end{equation}%
Note that we have $Hom\left( 1_{S},\hom \left( f,g\right) \right) =Hom\left(
1_{S}\otimes f,g\right) $ via Yoneda's embedding and $\left( \text{\ref{S1
F1}}\right) $, from which the functoriality of $\hom $ follows.

It follows from this functorial description that $\hom $ is biadditive. More
explicitly, suppose that we have given biproduct decompositions $%
X=X^{+}\oplus X^{-}$ and $Y=Y^{+}\oplus Y^{-}$ which are given by injective
morphisms $i_{X}^{\pm }:X^{\pm }\rightarrow X$, $i_{Y}^{\pm }:Y^{\pm
}\rightarrow Y$, surjective morphisms $p_{X}^{\pm }:X\rightarrow X^{\pm }$, $%
p_{Y}^{\pm }:Y\rightarrow Y^{\pm }$ and associated idempotents $e_{X}^{\pm
}:X\rightarrow X$, $e_{Y}^{\pm }:Y\rightarrow Y$. The functorial description
yields%
\begin{equation*}
\hom \left( X,Y\right) =\hom \left( X^{+},Y^{+}\right) \oplus \hom \left(
X^{+},Y^{-}\right) \oplus \hom \left( X^{-},Y^{+}\right) \oplus \hom \left(
X^{-},Y^{-}\right)
\end{equation*}%
associated to the decomposition of $Hom\left( S\otimes X,Y\right) $. For $%
\varepsilon ,\eta \in \left\{ \pm \right\} $, writing $i_{\hom \left(
X^{\varepsilon },Y^{\eta }\right) }:\hom \left( X^{\varepsilon },Y^{\eta
}\right) \rightarrow \hom \left( X,Y\right) $, $p_{\hom \left(
X^{\varepsilon },Y^{\eta }\right) }:\hom \left( X,Y\right) \rightarrow \hom
\left( X^{\varepsilon },Y^{\eta }\right) $ and $e_{\hom \left(
X^{\varepsilon },Y^{\eta }\right) }:\hom \left( X,Y\right) \rightarrow \hom
\left( X,Y\right) $ for the injective and surjective morphisms and the
idempotents arising from the decomposition of $Hom\left( S\otimes X,Y\right) 
$ and Yoneda's lemma, one checks%
\begin{equation}
i_{\hom \left( X^{\varepsilon },Y^{\eta }\right) }=\hom \left(
p_{X}^{\varepsilon },i_{Y}^{\eta }\right) \text{, }p_{\hom \left(
X^{\varepsilon },Y^{\eta }\right) }=\hom \left( i_{X}^{\varepsilon
},p_{Y}^{\eta }\right) \text{ and }e_{\hom \left( X^{\varepsilon },Y^{\eta
}\right) }=\hom \left( e_{X}^{\varepsilon },e_{Y}^{\eta }\right)
\label{S1 FDec morphisms}
\end{equation}%
as well as%
\begin{equation}
\ev_{X^{\varepsilon },Y^{\eta }}=p_{Y}^{\eta }\circ \ev_{X,Y}\circ \left(
i_{\hom \left( X^{\varepsilon },Y^{\eta }\right) }\otimes i_{X}^{\varepsilon
}\right) =p_{Y}^{\eta }\circ \ev_{X,Y}\circ \left( \hom \left(
p_{X}^{\varepsilon },i_{Y}^{\eta }\right) \otimes i_{X}^{\varepsilon
}\right) \text{.}  \label{S1 FDec evaluation}
\end{equation}

In particular, taking $f:X=X_{2}\rightarrow X_{1}=Y$ and $g=1_{\mathbb{I}}$
yields%
\begin{equation*}
f^{\vee }:=Y^{\vee }\rightarrow X^{\vee }
\end{equation*}%
and $X\leadsto X^{\vee }$ is a contravariant biadditive functor.

\bigskip

We proceed to define standard canonical morphisms. For a totally ordered
finite set $I$, a family $\left( X_{i},Y_{i}\right) _{i\in I}$ of objects $%
X_{i},Y_{i}\in \mathcal{C}$ and a morphism $\varphi :\otimes _{i\in
I}Y_{i}\rightarrow Y$, we may consider%
\[
\ev_{X_{i},Y_{i}}^{\varphi ,I}:(\otimes _{i\in I}\hom \left(
X_{i},Y_{i}\right) \otimes \left( \otimes _{i\in I}X_{i}\right) \overset{%
\tau _{X_{i},Y_{i}}^{I}}{\rightarrow }\otimes _{i\in I}\left( \hom \left(
X_{i},Y_{i}\right) \otimes X_{i}\right) \overset{\otimes _{i\in
I}\ev_{X_{i},Y_{i}}}{\rightarrow }\otimes _{i\in I}Y_{i}\overset{\varphi }{%
\rightarrow }Y\text{,}
\]%
where $\tau _{X_{i},Y_{i}}^{I}$ is obtained by appropriately switching the
components\footnote{%
In symbols,%
\[
\tau _{X_{i},Y_{i}}^{I}\left( \left( \otimes _{i\in I}f_{i}\right) \otimes
\left( \otimes _{i\in I}x_{i}\right) \right) :=\otimes _{i\in I}\left(
f_{i}\otimes x_{i}\right) \text{.}
\]%
}. Then we may define%
\[
\epsilon _{X_{i},Y_{i}}^{\psi ,\varphi ,I}:\otimes _{i\in I}\hom \left(
X_{i},Y_{i}\right) \rightarrow \hom \left( X,Y\right) 
\]%
as the unique morphism such that $\ev_{X,Y}\circ \left( \epsilon
_{X_{i},Y_{i}}^{\psi ,\varphi ,I}\otimes 1_{X}\right)
=\ev_{X_{i},Y_{i}}^{\psi ,\varphi ,I}$. When $I=\left\{ 1,...,i\right\} $, $%
X_{i}=X$ for every $i$, $Y_{i}=\mathbb{I}$ for every $i$ and $\varphi
:\otimes _{i\in I}\mathbb{I}\overset{\sim }{\rightarrow }\mathbb{I}$ is the
canonical morphism we write $\tau _{X}^{i}:=\tau _{X_{i},Y_{i}}^{I}$, $%
\ev_{X}^{i}:=\ev_{X_{i},Y_{i}}^{\varphi ,I}$ and $\epsilon
_{X_{i},Y_{i}}^{\varphi ,I}:=\epsilon _{X}^{i}$.

\bigskip

The morphisms%
\begin{equation*}
i_{X}:X\rightarrow X^{\vee \vee }\text{ and }\alpha _{X,Y}:Y\otimes X^{\vee
}\rightarrow \hom \left( X,Y\right)
\end{equation*}%
are defined, respectively, as the unique morphisms making the following
diagrams commutative:%
\begin{equation*}
\xymatrix{ X\otimes X^{\vee} \ar@/^{0.75pc}/[dr]^-{\ev_{X}^{\tau}} \ar[d]|{i_{X}\otimes1_{X^{\vee}}} & & Y\otimes X^{\vee}\otimes X \ar@/^{0.75pc}/[dr]^-{1_{Y}\otimes \ev_{X}} \ar[d]|{\alpha_{X,Y}\otimes1_{X}} & \\ X^{\vee\vee}\otimes X^{\vee} \ar[r]^-{\ev_{X^{\vee}}} & \mathbb{I} & \hom\left(X,Y\right)\otimes X \ar[r]^-{\ev_{X,Y}} & Y\text{.}}
\end{equation*}

\bigskip

We may consider the morphism%
\begin{equation*}
\hom \left( Y,Z\right) \otimes \hom \left( X,Y\right) \otimes X\overset{%
1_{\hom \left( Y,Z\right) }\otimes \ev_{X,Y}}{\rightarrow }\hom \left(
Y,Z\right) \otimes Y\overset{\ev_{Y,Z}}{\rightarrow }Z
\end{equation*}%
and define the internal composition law%
\begin{equation*}
c=c_{X,Y,Z}=\left( \cdot \right) \circ \left( \cdot \right) :\hom \left(
Y,Z\right) \otimes \hom \left( X,Y\right) \rightarrow \hom \left( X,Z\right)
\end{equation*}%
as the unique morphism such that%
\begin{equation*}
\ev_{X,Z}\circ \left( c_{X,Y,Z}\otimes 1_{X}\right) =\ev_{Y,Z}\circ \left(
1_{\hom \left( Y,Z\right) }\otimes \ev_{X,Y}\right) \text{.}
\end{equation*}%
The opposite internal composition law is defined as the composite%
\begin{equation*}
c_{X,Y,Z}^{\tau }:\hom \left( X,Y\right) \otimes \hom \left( Y,Z\right) 
\overset{\tau _{\hom \left( X,Y\right) ,\hom \left( Y,Z\right) }}{%
\rightarrow }\hom \left( Y,Z\right) \otimes \hom \left( X,Y\right) \overset{%
c_{X,Y,Z}}{\rightarrow }\hom \left( X,Z\right)
\end{equation*}

The following result is easily established.

\begin{lemma}
\label{S1 LComp}Suppose that we have given%
\begin{equation*}
f:S\rightarrow \hom \left( X,Y\right) \text{ and }g:T\rightarrow \hom \left(
Y,Z\right)
\end{equation*}%
which correspond, under $\left( \text{\ref{S1 F1}}\right) $, to morphisms%
\begin{equation*}
\varphi _{f}:S\otimes X\rightarrow Y\text{ and }\varphi _{g}:T\otimes
Y\rightarrow Z\text{.}
\end{equation*}

Then%
\begin{equation*}
c_{X,Y,Z}\circ \left( g\otimes f\right) :T\otimes S\overset{g\otimes f}{%
\rightarrow }\hom \left( Y,Z\right) \otimes \hom \left( X,Y\right) \overset{%
c_{X,Y,Z}}{\rightarrow }\mathbf{Hom}\left( X,Z\right)
\end{equation*}%
corresponds, under $\left( \text{\ref{S1 F1}}\right) $, to the morphism%
\begin{equation*}
\varphi _{g}\circ \left( 1_{T}\otimes \varphi _{f}\right) :T\otimes S\otimes
X\rightarrow Z\text{.}
\end{equation*}
\end{lemma}

\bigskip

In addition to the "external" duality morphism $Hom\left( X,Y\right)
\rightarrow Hom\left( Y^{\vee },X^{\vee }\right) $, the category $\mathcal{C}
$ is endowed with an internal duality morphism%
\begin{equation*}
d_{X,Y}:\hom \left( X,Y\right) \rightarrow \hom \left( Y^{\vee },X^{\vee
}\right) \text{,}
\end{equation*}%
which is by definition the unique morphism making the following diagram
commutative:%
\begin{equation*}
\xymatrix{ \hom\left(X,Y\right)\otimes Y^{\vee} \ar[r]^-{\tau_{\hom\left(X,Y\right),Y^{\vee}}} \ar[d]_{d_{X,Y}\otimes1_{Y^{\vee}}} & Y^{\vee}\otimes\hom\left(X,Y\right) \ar[d]^{c_{X,Y,\mathbb{I}}} \\ \hom\left(Y^{\vee},X^{\vee}\right)\otimes Y^{\vee} \ar[r]^-{\ev_{Y^{\vee},X^{\vee}}} & X^{\vee}\text{.} }
\end{equation*}%
It enjoys a number of expected properties, namely it makes commutative the
following diagrams.

\begin{itemize}
\item It is the unique morphism making the following diagram commutative%
\footnote{%
In symbols, setting $f^{\vee }:=d_{X,Y}\left( f\right) $ for $f\in \hom
\left( X,Y\right) $,%
\begin{equation*}
\left\langle f\left( x\right) ,y^{\vee }\right\rangle =\left\langle
x,f^{\vee }\left( y^{\vee }\right) \right\rangle \text{ for }x\in X\text{
and }y^{\vee }\in Y^{\vee }\text{.}
\end{equation*}%
}:%
\begin{equation}
\xymatrix{ \hom\left(X,Y\right)\otimes X\otimes Y^{\vee} \ar[rr]^{\left(1_{X}\otimes d_{X,Y}\otimes1_{Y^{\vee}}\right)\circ\left(\tau_{\hom\left(X,Y\right),X}\otimes1_{Y^{\vee}}\right)} \ar[d]_{\ev_{X,Y}\otimes1_{Y^{\vee}}} & & X\otimes\hom\left(Y^{\vee},X^{\vee}\right)\otimes Y^{\vee} \ar[d]^{1_{X}\otimes \ev_{Y^{\vee},X^{\vee}}} \\ Y\otimes Y^{\vee} \ar[r]^{\ev_{Y}^{\tau}} & \mathbb{I} & X\otimes X^{\vee} \ar[l]_{\ev_{X}^{\tau}}}
\label{S1 D Internal duality 1}
\end{equation}

\item The following diagram is commutative\footnote{%
In symbols, for $f\in \hom \left( X,Y\right) $ and $g\in \hom \left(
Y,Z\right) $,%
\begin{equation*}
\left( g\circ f\right) ^{\vee }=g^{\vee }\circ ^{opp}f^{\vee }=f^{\vee
}\circ g^{\vee }\text{.}
\end{equation*}%
}:%
\begin{equation}
\xymatrix{ \hom\left(Y,Z\right)\otimes\hom\left(X,Y\right) \ar[r]^-{d_{Y,Z}\otimes d_{X,Y}} \ar[d]_{c_{X,Y,Z}} & \hom\left(Z^{\vee},Y^{\vee}\right)\otimes\hom\left(Y^{\vee},X^{\vee}\right) \ar[d]^{c_{Z^{\vee},Y^{\vee},X^{\vee}}^{\tau}} \\ \hom\left(X,Z\right) \ar[r]^-{d_{X,Z}} & \hom\left(Z^{\vee},X^{\vee}\right)\text{.} }
\label{S1 D Internal duality 2}
\end{equation}

\item If we have given $f:X_{2}\rightarrow X_{1}$ and $g:Y_{1}\rightarrow
Y_{2}$ the following diagram is commutative:%
\begin{equation}
\xymatrix{ \hom\left(X_{1},Y_{1}\right) \ar[r]^-{d_{X_{1},Y_{1}}} \ar[d]_{\hom\left(f,g\right)} & \hom\left(Y_{1}^{\vee},X_{1}^{\vee}\right) \ar[d]^{\hom\left(g^{\vee},f^{\vee}\right)} \\ \hom\left(X_{2},Y_{2}\right) \ar[r]^-{d_{X_{2},Y_{2}}} & \hom\left(Y_{2}^{\vee},X_{2}^{\vee}\right)\text{.} }
\label{S1 D Internal duality 3}
\end{equation}

\item The following further diagrams are commutative\footnote{%
In symbols, the second commutative diagram tells that $f^{\vee \vee }=f$ up
to the identification $X^{\vee \vee }=X$ and $Y^{\vee \vee }=Y$ whenever $X$
and $Y$ are reflexive.}:%
\begin{equation}
\xymatrix{ Y\otimes X^{\vee} \ar[r]^-{\alpha_{X,Y}} \ar[d]|{\left(1_{X^{\vee}}\otimes i_{Y}\right)\circ\tau_{Y,X^{\vee}}} & \hom\left(X,Y\right) \ar[d]|{d_{X,Y}} & \hom\left(X,Y\right) \ar@/^{1pc}/[rd]^-{\hom\left(1_{X},i_{Y}\right)} \ar[d]|{d_{Y^{\vee},X^{\vee}}\circ d_{X,Y}} & \\ X^{\vee}\otimes Y^{\vee\vee} \ar[r]^-{\alpha_{Y^{\vee},X^{\vee}}} & \hom\left(Y^{\vee},X^{\vee}\right) & \hom\left(X^{\vee\vee},Y^{\vee\vee}\right) \ar[r]^-{\hom\left(i_{X},1_{Y^{\vee\vee}}\right)} & \hom\left(X,Y^{\vee\vee}\right)}
\label{S1 D Internal duality 4}
\end{equation}
\end{itemize}

\bigskip

We recall that $\mathcal{C}$ is rigid whenever the morphisms $\epsilon
_{X_{i},Y_{i}}^{I}$ and $i_{X}$ are isomorphisms and is said to be
pseudo-abelian when idempotents have kernels (and then also cokernels).

\bigskip 

We will employ the following notation: a label $\left( \otimes \right) $
(resp. $\left( \tau \right) $) placed in the middle of a diagram will mean
that the diagram is commutative by functoriality of $\otimes $ (resp. the $%
\tau $ constraint).

\subsection{Abstract internal multiplication}

Suppose that we have given a morphism%
\begin{equation*}
f:S\rightarrow \hom \left( X,Y\right) \text{ corresponding to }\varphi
_{f}:S\otimes X\overset{f\otimes 1_{X}}{\rightarrow }\hom \left( X,Y\right)
\otimes X\overset{\ev_{X,Y}}{\rightarrow }Y\text{.}
\end{equation*}

Then we define the corresponding "internal multiplication" morphism as the
composite:%
\begin{equation*}
\iota _{f}:S\overset{f}{\rightarrow }\hom \left( X,Y\right) \overset{d_{X,Y}}%
{\rightarrow }\hom \left( Y^{\vee },X^{\vee }\right) \text{ corresponding to 
}\varphi _{\iota _{f}}:S\otimes Y^{\vee }\overset{\iota _{f}\otimes
1_{Y^{\vee }}}{\rightarrow }\hom \left( Y^{\vee },X^{\vee }\right) \otimes
Y^{\vee }\overset{\ev_{Y^{\vee },X^{\vee }}}{\rightarrow }X^{\vee }\text{.}
\end{equation*}

One checks that $\left( \text{\ref{S1 D Internal duality 1}}\right) $\
implies that the following diagram is commutative:%
\begin{equation}
\xymatrix@C78pt{ S\otimes X\otimes Y^{\vee} \ar[r]^-{\left(1_{X}\otimes\varphi_{\iota_{f}}\right)\circ\left(\tau_{S,X}\otimes1_{Y^{\vee}}\right)} \ar[d]_{\varphi_{f}\otimes1_{Y^{\vee}}} & X\otimes X^{\vee} \ar[d]^{\ev_{X}^{\tau}} \\ Y\otimes Y^{\vee} \ar[r]^-{\ev_{Y}^{\tau}} & \mathbb{I}\text{.} }
\label{S1 AIM D1}
\end{equation}

\begin{remark}
\label{S1 AIM R1}The morphism $\varphi _{\iota _{f}}$, and hence $\iota _{f}$%
, is characterized by the property of making $\left( \text{\ref{S1 AIM D1}}%
\right) $ commutative.
\end{remark}

\bigskip

Suppose now that we have also given:%
\begin{eqnarray*}
&&g:T\rightarrow \hom \left( Y,Z\right) \text{ corresponding to }\varphi
_{g}:T\otimes Y\rightarrow Z\text{,} \\
&&h:U\rightarrow \hom \left( X,Z\right) \text{ corresponding to }\varphi
_{h}:U\otimes X\rightarrow Z\text{,} \\
&&k:T\rightarrow \hom \left( S,U\right) \text{ corresponding to }\varphi
_{k}:T\otimes S\rightarrow U\text{.}
\end{eqnarray*}

As an application of Lemma \ref{S1 LComp}, we have the equivalence:%
\begin{equation}
\xymatrix{ T\otimes S\otimes X \ar@{}[dr]|{\circlearrowright} \ar[r]^-{1_{T}\otimes\varphi_{f}} \ar[d]|{\varphi_{k}\otimes1_{X}} & T\otimes Y \ar@{}[dr]|{\Leftrightarrow} \ar[d]|{\varphi_{g}} & T\otimes S \ar@{}[dr]|{\circlearrowright} \ar[r]^-{g\otimes f} \ar[d]|{\varphi_{k}} & \hom\left(Y,Z\right)\otimes\hom\left(X,Y\right) \ar[d]|{c_{X,Y,Z}} \\ U\otimes X \ar[r]^-{\varphi_{h}} & Z & U \ar[r]^-{h} & \hom\left(X,Z\right)}
\label{S1 AIM D2}
\end{equation}

We also have the associated internal multiplication morphisms:%
\begin{eqnarray*}
\iota _{g} &:&T\overset{g}{\rightarrow }\hom \left( Y,Z\right) \overset{%
d_{Y,Z}}{\rightarrow }\hom \left( Z^{\vee },Y^{\vee }\right) \text{
corresponding to }\varphi _{\iota _{g}}:T\otimes Z^{\vee }\rightarrow
Y^{\vee }\text{,} \\
\iota _{h} &:&U\overset{h}{\rightarrow }\hom \left( X,Z\right) \overset{%
d_{X,Z}}{\rightarrow }\hom \left( Z^{\vee },X^{\vee }\right) \text{
corresponding to }\varphi _{\iota _{h}}:U\otimes Z^{\vee }\rightarrow
X^{\vee }\text{.}
\end{eqnarray*}

Consider the morphism:%
\begin{equation*}
\varphi _{k}^{\tau }:S\otimes T\overset{\tau _{S,T}}{\rightarrow }T\otimes S%
\overset{\varphi _{k}}{\rightarrow }U\text{.}
\end{equation*}

The equivalence $\left( \text{\ref{S1 AIM D2}}\right) $, applied with $%
\left( f,g,h,\varphi _{k}\right) $ replaced by $\left( \iota _{g},\iota
_{h},\iota _{f},\varphi _{k}^{\tau }\right) $, easily translates into the
equivalence:%
\begin{equation}
\xymatrix{ S\otimes T\otimes Z^{\vee} \ar@{}[dr]|{\circlearrowright} \ar[r]^-{1_{S}\otimes\varphi_{\iota_{g}}} \ar[d]|{\varphi_{k}^{\tau}\otimes1_{Z^{\vee}}} & S\otimes Y^{\vee} \ar@{}[dr]|{\Leftrightarrow} \ar[d]|{\varphi_{\iota_{f}}} & T\otimes S \ar@{}[dr]|{\circlearrowright} \ar[r]^-{\iota_{g}\otimes\iota_{f}} \ar[d]|{\varphi_{k}} & \hom\left(Z^{\vee},Y^{\vee}\right)\otimes\hom\left(Y^{\vee},X^{\vee}\right) \ar[d]|{c_{Z^{\vee},Y^{\vee},X^{\vee}}^{\tau}} \\ U\otimes Z^{\vee} \ar[r]^-{\varphi_{\iota_{h}}} & X^{\vee} & U \ar[r]^-{\iota_{h}} & \hom\left(Z^{\vee},X^{\vee}\right)\text{.}}
\label{S1 AIM D3}
\end{equation}

Finally we remark that the second square of the following diagram is
commutative by $\left( \text{\ref{S1 D Internal duality 2}}\right) $:%
\begin{equation*}
\xymatrix{ T\otimes S \ar[r]^-{g\otimes f} \ar[d]_{\varphi_{k}} & \hom\left(Y,Z\right)\otimes\hom\left(X,Y\right) \ar[r]^-{d_{Y,Z}\otimes d_{X,Y}} \ar[d]|{c_{X,Y,Z}} & \hom\left(Z^{\vee},Y^{\vee}\right)\otimes\hom\left(Y^{\vee},X^{\vee}\right) \ar[d]^{c_{Z^{\vee},Y^{\vee},X^{\vee}}^{\tau}} \\ U \ar[r]^-{h} & \hom\left(X,Z\right) \ar[r]^-{d_{X,Z}} & \hom\left(Z^{\vee},X^{\vee}\right)\text{.}}
\end{equation*}%
It follows that we have the implication%
\begin{equation}
\left( \text{\ref{S1 AIM D2}}\right) \text{ commutative }\Rightarrow \text{ }%
\left( \text{\ref{S1 AIM D3}}\right) \text{ commutative\footnote{%
Which is indeed an equivalence when $X$, $Y$, and $Z$ are reflexive, by the
second commutative diagram in $\left( \text{\ref{S1 D Internal duality 4}}%
\right) $.} }  \label{S1 AIM Implication1}
\end{equation}

\bigskip

We now turn to the consideration of how the formation of internal
multiplication behaves with respect to biproduct decompositions. To this end
we assume that, for all the objects $W$ considered above, we have given a
biproduct decomposition $W=W^{+}\oplus W^{-}$ obtained by means of injective
morphisms $i_{W}^{\pm }:W^{\pm }\rightarrow W$, surjective morphisms $%
p_{W}^{\pm }:W\rightarrow W^{\pm }$ and associated idempotents $e_{W}^{\pm
}:W\rightarrow W$.

For $\left( \varepsilon _{1},\eta _{1},\nu _{1}\right) ,\left( \varepsilon
_{2},\eta _{2},\nu _{2}\right) \in \left\{ \pm \right\} \times \left\{ \pm
\right\} \times \left\{ \pm \right\} $, define the following morphisms:%
\begin{eqnarray*}
&&f^{\varepsilon _{1},\varepsilon _{2};\eta _{2}}:S^{\varepsilon _{1}}%
\overset{i_{S}^{\varepsilon _{1}}}{\rightarrow }S\overset{f}{\rightarrow }%
\hom \left( X,Y\right) \overset{p_{\hom \left( X^{\varepsilon _{2}},Y^{\eta
_{2}}\right) }}{\rightarrow }\hom \left( X^{\varepsilon _{2}},Y^{\eta
_{2}}\right) \text{,} \\
&&\iota _{f}^{\varepsilon _{1},\eta _{2};\varepsilon _{2}}:S^{\varepsilon
_{1}}\overset{i_{S}^{\varepsilon _{1}}}{\rightarrow }S\overset{\iota _{f}}{%
\rightarrow }\hom \left( Y^{\vee },X^{\vee }\right) \overset{p_{\hom \left(
Y^{\eta _{2}\vee },X^{\varepsilon _{2}\vee }\right) }}{\rightarrow }\hom
\left( Y^{\eta _{2}\vee },X^{\varepsilon _{2}\vee }\right) \text{,} \\
&&\varphi _{f}^{\varepsilon _{1},\varepsilon _{2};\eta _{2}}:S^{\varepsilon
_{1}}\otimes X^{\varepsilon _{2}}\overset{i_{S}^{\varepsilon _{1}}\otimes
i_{X}^{\varepsilon _{2}}}{\rightarrow }S\otimes X\overset{\varphi _{f}}{%
\rightarrow }Y\overset{p_{Y}^{\eta _{2}}}{\rightarrow }Y^{\eta _{2}}\text{,}
\\
&&\varphi _{\iota _{f}}^{\varepsilon _{1},\eta _{2};\varepsilon
_{2}}:S^{\varepsilon _{1}}\otimes Y^{\eta _{2}\vee }\overset{%
i_{S}^{\varepsilon _{1}}\otimes i_{Y^{\vee }}^{\eta _{2}}}{\rightarrow }%
S\otimes Y^{\vee }\overset{\varphi _{\iota _{f}}}{\rightarrow }X^{\vee }%
\overset{p_{X^{\vee }}^{\varepsilon _{2}}}{\rightarrow }X^{\varepsilon
_{2}\vee }
\end{eqnarray*}%
as well as the morphisms $g^{\eta _{1},\eta _{2};\nu _{2}}$, $h^{\nu
_{1},\varepsilon _{2};\nu _{2}}$, $k^{\eta _{1},\varepsilon _{1};\nu _{1}}$
and the other defined similarly as for $f$.

\begin{lemma}
\label{S1 AIM L1}Writing $\varphi _{f^{\varepsilon _{1},\varepsilon
_{2};\eta _{2}}}:S^{\varepsilon _{1}}\otimes X^{\varepsilon _{2}}\rightarrow
Y^{\eta _{2}}$ (resp. $\varphi _{\iota _{f}^{\varepsilon _{1},\eta
_{2};\varepsilon _{2}}}:S^{\varepsilon _{1}}\otimes Y^{\eta _{2}\vee
}\rightarrow X^{\varepsilon _{2}\vee }$)\ for the morphism corresponding to $%
f^{\varepsilon _{1},\varepsilon _{2};\eta _{2}}:S^{\varepsilon
_{1}}\rightarrow \hom \left( X^{\varepsilon _{2}},Y^{\eta _{2}}\right) $
(resp. $\iota _{f}^{\varepsilon _{1},\eta _{2};\varepsilon
_{2}}:S^{\varepsilon _{1}}\rightarrow \hom \left( Y^{\eta _{2}\vee
},X^{\varepsilon _{2}\vee }\right) $),\ we have $\varphi _{f^{\varepsilon
_{1},\varepsilon _{2};\eta _{2}}}=\varphi _{f}^{\varepsilon _{1},\varepsilon
_{2};\eta _{2}}$ and $\varphi _{\iota _{f}^{\varepsilon _{1},\eta
_{2};\varepsilon _{2}}}=\varphi _{\iota _{f}}^{\varepsilon _{1},\eta
_{2};\varepsilon _{2}}$\ as well as $\iota _{f^{\varepsilon _{1},\varepsilon
_{2};\eta _{2}}}=\iota _{f}^{\varepsilon _{1},\eta _{2};\varepsilon _{2}}$
or, equivalently, $\varphi _{\iota _{f^{\varepsilon _{1},\varepsilon
_{2};\eta _{2}}}}=\varphi _{\iota _{f}^{\varepsilon _{1},\eta
_{2};\varepsilon _{2}}}=\varphi _{\iota _{f}}^{\varepsilon _{1},\eta
_{2};\varepsilon _{2}}$.
\end{lemma}

\begin{proof}
The equality $\varphi _{f^{\varepsilon _{1},\varepsilon _{2};\eta
_{2}}}=\varphi _{f}^{\varepsilon _{1},\varepsilon _{2};\eta _{2}}$ is a
consequence of $p_{\hom \left( X^{\varepsilon _{2}},Y^{\eta _{2}}\right)
}=\hom \left( i_{X}^{\varepsilon _{2}},p_{Y}^{\eta _{2}}\right) $ given by $%
\left( \text{\ref{S1 FDec morphisms}}\right) $ and the characterizing
property $\left( \text{\ref{S1 FHom(f,g)}}\right) $ and $\varphi _{\iota
_{f}^{\varepsilon _{1},\eta _{2};\varepsilon _{2}}}=\varphi _{\iota
_{f}}^{\varepsilon _{1},\eta _{2};\varepsilon _{2}}$ is proved in the same
way. Next, consider the following diagram:%
\begin{equation*}
\xymatrix{S^{\varepsilon_{1}} \ar[r]^{i_{S}^{\varepsilon_{1}}} & S \ar[r]^{f} & \hom\left(X,Y\right) \ar[r]^{d_{X,Y}} \ar[d]|{p_{\hom\left(X^{\varepsilon_{2}},Y^{\eta_{2}}\right)}} & \hom\left(Y^{\vee},X^{\vee}\right) \ar[d]^{p_{\hom\left(Y^{\eta_{2}\vee},X^{\varepsilon_{2}\vee}\right)}} \\ & & \hom\left(X^{\varepsilon_{2}},Y^{\eta_{2}}\right) \ar[r]^{d_{X^{\varepsilon_{2}},Y^{\eta_{2}}}} & \hom\left(Y^{\eta_{2}\vee},X^{\varepsilon_{2}\vee}\right)\text{.}}
\end{equation*}%
The square is commutative because \[\hom \left( i_{X}^{\varepsilon
_{2}},p_{Y}^{\eta _{2}}\right) =p_{\hom \left( X^{\varepsilon _{2}},Y^{\eta
_{2}}\right) }\] and \[\hom \left( \left( p_{Y}^{\eta _{2}}\right) ^{\vee
},\left( i_{X}^{\varepsilon _{2}}\right) ^{\vee }\right) =\hom \left(
i_{Y^{\vee }}^{\eta _{2}},p_{X^{\vee }}^{\varepsilon _{2}}\right) =p_{\hom
\left( Y^{\eta _{2}\vee },X^{\varepsilon _{2}\vee }\right) },\] again by $%
\left( \text{\ref{S1 FDec morphisms}}\right) $ and because the duality is a
contravariant and additive functor, so that we may apply $\left( \text{\ref%
{S1 D Internal duality 3}}\right) $. But we have%
\begin{eqnarray*}
&&p_{\hom \left( Y^{\eta _{2}\vee },X^{\varepsilon _{2}\vee }\right) }\circ
d_{X,Y}\circ f\circ i_{S}^{\varepsilon _{1}}=p_{\hom \left( Y^{\eta _{2}\vee
},X^{\varepsilon _{2}\vee }\right) }\circ \iota _{f}\circ i_{S}^{\varepsilon
_{1}}=\iota _{f}^{\varepsilon _{1},\eta _{2};\varepsilon _{2}}\text{,} \\
&&d_{X^{\varepsilon _{2}},Y^{\eta _{2}}}\circ p_{\hom \left( X^{\varepsilon
_{2}},Y^{\eta _{2}}\right) }\circ f\circ i_{S}^{\varepsilon
_{1}}=d_{X^{\varepsilon _{2}},Y^{\eta _{2}}}\circ f^{\varepsilon
_{1},\varepsilon _{2};\eta _{2}}=\iota _{f^{\varepsilon _{1},\varepsilon
_{2};\eta _{2}}}\text{.}
\end{eqnarray*}
\end{proof}

\bigskip

It follows from Lemma \ref{S1 AIM L1} (also applied to $g$, $h$ and $k$)\
that we may apply the above considerations with $\left( f,g,h,\varphi
_{k}\right) $ replaced by $\left( f^{\varepsilon _{1},\varepsilon _{2};\eta
_{2}},g^{\eta _{1},\eta _{2};\nu _{2}},h^{\nu _{1},\varepsilon _{2};\nu
_{2}},\varphi _{k}^{\eta _{1},\varepsilon _{1};\nu _{1}}\right) $. Hence, we
deduce that%
\begin{equation}
\xymatrix@C110pt{ S^{\varepsilon_{1}}\otimes X^{\varepsilon_{2}}\otimes Y^{\eta_{2}\vee} \ar[r]^-{\left(1_{X^{\varepsilon_{2}}}\otimes\varphi_{\iota_{f}}^{\varepsilon_{1},\eta_{2};\varepsilon_{2}}\right)\circ\left(\tau_{S^{\varepsilon_{1}},X^{\varepsilon_{2}}}\otimes1_{Y^{\eta_{2}\vee}}\right)} \ar[d]_{\varphi_{f}^{\varepsilon_{1},\varepsilon_{2};\eta_{2}}\otimes1_{Y^{\eta_{2}\vee}}} & X^{\varepsilon_{2}}\otimes X^{\varepsilon_{2}\vee} \ar[d]^{\ev_{X^{\varepsilon_{2}}}^{\tau}} \\ Y^{\eta_{2}}\otimes Y^{\eta_{2}\vee} \ar[r]^-{\ev_{Y^{\eta_{2}}}^{\tau}} & \mathbb{I} }
\label{S1 AIM D4}
\end{equation}%
is commutative and that we have the implications:

\begin{equation}
\xymatrix{ T^{\eta_{1}}\otimes S^{\varepsilon_{1}}\otimes X^{\varepsilon_{2}} \ar@{}[dr]|{\circlearrowright} \ar[r]^-{1_{T^{\eta_{1}}}\otimes\varphi_{f}^{\varepsilon_{1},\varepsilon_{2};\eta_{2}}} \ar[d]|{\varphi_{k}^{\eta_{1},\varepsilon_{1};\nu_{1}}\otimes1_{X^{\varepsilon_{2}}}} & T^{\eta_{1}}\otimes Y^{\eta_{2}} \ar@{}[dr]|{\Leftrightarrow} \ar[d]|{\varphi_{g}^{\eta_{1},\eta_{2};\nu_{2}}} & T^{\eta_{1}}\otimes S^{\varepsilon_{1}} \ar@{}[dr]|{\circlearrowright} \ar[r]^-{g^{\eta_{1},\eta_{2};\nu_{2}}\otimes f^{\varepsilon_{1},\varepsilon_{2};\eta_{2}}} \ar[d]|{\varphi_{k}^{\eta_{1},\varepsilon_{1};\nu_{1}}} & \hom\left(Y^{\eta_{2}},Z^{\nu_{2}}\right)\otimes\hom\left(X^{\varepsilon_{2}},Y^{\eta_{2}}\right) \ar[d]|{c_{X^{\varepsilon_{2}},Y^{\eta_{2}},Z^{\nu_{2}}}} \\ U^{\nu_{1}}\otimes X^{\varepsilon_{2}} \ar[r]^-{\varphi_{h}^{\nu_{1},\varepsilon_{2};\nu_{2}}} & Z^{\nu_{2}} & U^{\nu_{1}} \ar[r]^-{h^{\nu_{1},\varepsilon_{2};\nu_{2}}} & \hom\left(X^{\varepsilon_{2}},Z^{\nu_{2}}\right)\text{,}}
\label{S1 AIM D5}
\end{equation}

\begin{equation}
\xymatrix{ S^{\varepsilon_{1}}\otimes T^{\eta_{1}}\otimes Z^{\nu_{2}\vee} \ar@{}[dr]|{\circlearrowright} \ar[r]^-{1_{S^{\varepsilon_{1}}}\otimes\varphi_{\iota_{g}}^{\eta_{1},\nu_{2};\eta_{2}}} \ar[d]|{\varphi_{k}^{\eta_{1},\varepsilon_{1};\nu_{1},\tau}\otimes1_{Z^{\nu_{2}\vee}}} & S^{\varepsilon_{1}}\otimes Y^{\eta_{2}\vee} \ar@{}[dr]|{\Leftrightarrow} \ar[d]|{\varphi_{\iota_{f}}^{\varepsilon_{1},\eta_{2};\varepsilon_{2}}} & T^{\eta_{1}}\otimes S^{\varepsilon_{1}} \ar@{}[dr]|{\circlearrowright} \ar[r]^-{\iota_{g}^{\eta_{1},\nu_{2};\eta_{2}}\otimes\iota_{f}^{\varepsilon_{1},\eta_{2};\varepsilon_{2}}} \ar[d]|{\varphi_{k}^{\eta_{1},\varepsilon_{1};\nu_{1}}} & \hom\left(Z^{\nu_{2}\vee},Y^{\eta_{2}\vee}\right)\otimes\hom\left(Y^{\eta_{2}\vee},X^{\varepsilon_{2}\vee}\right) \ar[d]|{c_{Z^{\nu_{2}\vee},Y^{\eta_{2}\vee},X^{\varepsilon_{2}\vee}}^{\tau}} \\ U^{\nu_{1}}\otimes Z^{\nu_{2}\vee} \ar[r]^-{\varphi_{\iota_{h}}^{\nu_{1},\nu_{2};\varepsilon_{2}}} & X^{\varepsilon_{2}\vee} & U^{\nu_{1}} \ar[r]^-{\iota_{h}^{\nu_{1},\nu_{2};\varepsilon_{2}}} & \hom\left(Z^{\nu_{2}\vee},X^{\varepsilon_{2}\vee}\right)\text{,}}
\label{S1 AIM D6}
\end{equation}

as well as 
\begin{equation}
\left( \text{\ref{S1 AIM D5}}\right) \text{ commutative }\Rightarrow \text{ }%
\left( \text{\ref{S1 AIM D6}}\right) \text{ commutative.}
\label{S1 AIM Implication2}
\end{equation}

The proof of the following Lemma is just a formal computation.

\begin{lemma}
\label{S1 AIM L2}Suppose that $e_{Z}^{\nu _{2}}\circ \varphi _{h}\circ
\left( e_{U}^{\nu _{1}}\otimes 1_{X}\right) =e_{Z}^{\nu _{2}}\circ \varphi
_{h}$ and $e_{Z}^{\nu _{2}}\circ \varphi _{g}\circ \left( 1_{T}\otimes
e_{Y}^{\eta _{2}}\right) =e_{Z}^{\nu _{2}}\circ \varphi _{g}$. Then we have
the implication $\left( \text{\ref{S1 AIM D2}}\right) $ commutative $\Rightarrow $ $%
\left( \text{\ref{S1 AIM D5}}\right) $ commutative$.$
\end{lemma}

The following result is now a combination of the commutativity of $\left( 
\text{\ref{S1 AIM D4}}\right) $, Lemma \ref{S1 AIM L2} and $\left( \text{\ref%
{S1 AIM Implication2}}\right) $.

\begin{proposition}
\label{S1 AIM P1}Suppose that $e_{Z}^{\nu _{2}}\circ \varphi _{h}\circ
\left( e_{U}^{\nu _{1}}\otimes 1_{X}\right) =e_{Z}^{\nu _{2}}\circ \varphi
_{h}$ and $e_{Z}^{\nu _{2}}\circ \varphi _{g}\circ \left( 1_{T}\otimes
e_{Y}^{\eta _{2}}\right) =e_{Z}^{\nu _{2}}\circ \varphi _{g}$ and that $%
\left( \text{\ref{S1 AIM D2}}\right) $ is commutative. Then $\left( \text{%
\ref{S1 AIM D4}}\right) $, $\left( \text{\ref{S1 AIM D5}}\right) $ and $%
\left( \text{\ref{S1 AIM D6}}\right) $ are commutative.
\end{proposition}

For future reference it will be convenient to introduce some more notation.
When $\alpha _{X,Y}$ is an isomorphism, we define%
\begin{equation*}
D_{f}:S\overset{f}{\rightarrow }\hom \left( X,Y\right) \overset{\alpha
_{X,Y}^{-1}}{\rightarrow }Y\otimes X^{\vee }\text{.}
\end{equation*}%
Suppose now that we have given $g:S\rightarrow \hom \left( X^{\vee },Y^{\vee
}\right) $, so that we have $\iota _{g}:S\rightarrow \hom \left( Y^{\vee
\vee },X^{\vee \vee }\right) $. When $X$ is reflexive, we define%
\begin{equation*}
\iota _{g}^{\ast }:S\overset{\iota _{g}}{\rightarrow }\hom \left( Y^{\vee
\vee },X^{\vee \vee }\right) \overset{\hom \left( i_{Y},i_{X}^{-1}\right) }{%
\rightarrow }\hom \left( Y,X\right) \text{.}
\end{equation*}%
Suppose that $X$ is reflexive and that both $\alpha _{X^{\vee },Y^{\vee }}$
and $\alpha _{Y,X}$ are isomorphisms. Then it follows from the first
commutative diagram of $\left( \text{\ref{S1 D Internal duality 4}}\right) $
that $d_{Y,X}$ is an isomorphism and then, from the second commutative
diagram of $\left( \text{\ref{S1 D Internal duality 4}}\right) $, we deduce
that $\hom \left( i_{Y},i_{X}^{-1}\right) \circ d_{X^{\vee },Y^{\vee
}}=d_{Y,X}^{-1}$, so that%
\begin{eqnarray*}
\iota _{\iota _{g}^{\ast }} &:&=d_{Y,X}\circ \iota _{g}^{\ast }=d_{Y,X}\circ
\hom \left( i_{Y},i_{X}^{-1}\right) \circ d_{X^{\vee },Y^{\vee }}\circ g \\
&=&d_{Y,X}\circ d_{Y,X}^{-1}\circ g=g\text{.}
\end{eqnarray*}%
It follows from $\left( \text{\ref{S1 AIM D1}}\right) $ that the first of
the subsequent equivalently commutative diagrams commutes:%
\begin{equation}
\xymatrix@C50pt{ S\otimes Y\otimes X^{\vee} \ar[r]^-{\left(1_{Y}\otimes\varphi_{g}\right)\circ\left(\tau_{S,Y}\otimes1_{X^{\vee}}\right)} \ar[d]|{\varphi_{\iota_{g}^{\ast}}\otimes1_{X^{\vee}}} & Y\otimes Y^{\vee} \ar@{}[dr]|{\Leftrightarrow} \ar[d]|{\ev_{Y}^{\tau}} & S\otimes X^{\vee}\otimes Y \ar[r]^-{\left(1_{X^{\vee}}\otimes\varphi_{\iota_{g}^{\ast}}\right)\circ\left(\tau_{S,X^{\vee}}\otimes1_{Y}\right)} \ar[d]|{\varphi_{g}\otimes1_{Y}} & X^{\vee}\otimes X \ar[d]|{\ev_{X}} \\ X\otimes X^{\vee} \ar[r]^-{\ev_{X}^{\tau}} & \mathbb{I} & Y^{\vee}\otimes Y \ar[r]^-{\ev_{Y}} & \mathbb{I}\text{.}}
\label{S1 AIM D1 reflexivity 1}
\end{equation}%
Here the equivalence is easily obtained by applying $1_{S}\otimes \tau
_{X^{\vee },Y}$ (resp. $1_{S}\otimes \tau _{Y,X^{\vee }}$)\ to the first
(resp. second) diagram to get the second (resp. first) diagram. Also, it
follows from the functorial description $Hom\left( 1_{S},\hom \left(
i_{Y},i_{X}^{-1}\right) \right) =Hom\left( 1_{S}\otimes
i_{Y},i_{X}^{-1}\right) $ (up to $\left( \text{\ref{S1 F1}}\right) $)\ of $%
\hom $ that the following diagram is commutative:%
\begin{equation}
\xymatrix{ S\otimes Y \ar[r]^-{\varphi_{\iota_{g}^{\ast}}} \ar[d]_{1_{S}\otimes i_{Y}} & X \ar[d]^{i_{X}} \\ S\otimes Y^{\vee\vee} \ar[r]^-{\varphi_{\iota_{g}}} & X^{\vee\vee}\text{.} }
\label{S1 AIM D1 reflexivity 2}
\end{equation}

\bigskip

Finally, in addition to $D_{\iota _{g}}:S\rightarrow X^{\vee \vee }\otimes
Y^{\vee \vee \vee }$, defined when $\alpha _{Y^{\vee \vee },X^{\vee \vee }}$
is an isomorphism, we may define, when $X$ is reflexive and $\alpha _{Y,X}$
is an isomorphism:%
\begin{equation*}
D_{\iota _{g}^{\ast }}:S\overset{\iota _{g}^{\ast }}{\rightarrow }\hom
\left( Y,X\right) \overset{\alpha _{Y,X}^{-1}}{\rightarrow }X\otimes Y^{\vee
}\text{.}
\end{equation*}%
The relationship between $D_{\iota _{g}}$ and $D_{\iota _{g}^{\ast }}$ can
be made explicit as follows. Consider the following diagram:%
\begin{equation}
\xymatrix{ S \ar[r]^-{g} & \hom\left(X^{\vee},Y^{\vee}\right) \ar[r]^-{d_{X^{\vee},Y^{\vee}}} & \hom\left(Y^{\vee\vee},X^{\vee\vee}\right) \ar[r]^-{\hom\left(i_{Y},i_{X}^{-1}\right)} & \hom\left(Y,X\right) \\ & Y^{\vee}\otimes X^{\vee\vee} \ar[r]^-{\left(1_{X^{\vee\vee}}\otimes i_{Y^{\vee}}\right)\circ\tau_{Y^{\vee},X^{\vee\vee}}} \ar[u]^{\alpha_{X^{\vee},Y^{\vee}}}  & X^{\vee\vee}\otimes Y^{\vee\vee\vee} \ar[r]^-{i_{X}^{-1}\otimes\left(i_{Y}\right)^{\vee}} \ar[u]|{\alpha_{Y^{\vee\vee},X^{\vee\vee}}} & X\otimes Y^{\vee}\text{.} \ar[u]_{\alpha_{Y,X}}}
\label{S1 AIM D7'}
\end{equation}%
The first square is commutative thank to the first diagram in $\left( \text{%
\ref{S1 D Internal duality 4}}\right) $, while the second square is
commutative by functoriality of $\alpha $. The subsequent lemma, whose proof
we leave to the reader, shows that, when $Y$ is reflexive, $\left(
i_{Y}\right) ^{\vee }=i_{Y^{\vee }}^{-1}$ and we find, in this case,%
\begin{equation}
D_{\iota _{g}}=\left( i_{X}\otimes i_{Y^{\vee }}\right) \circ D_{\iota
_{g}^{\ast }}  \label{S1 AIM D7}
\end{equation}

\begin{lemma}
\label{S1 Lemma reflexivity duality}We have the equality $\left(
i_{X}\right) ^{\vee }\circ i_{X^{\vee }}=1_{X^{\vee }}$. In particular, if $%
X $ is reflexive, then $X^{\vee }$ is reflexive and $i_{X^{\vee
}}^{-1}=\left( i_{X}\right) ^{\vee }$.
\end{lemma}

\bigskip

The following lemma will be useful later.

\begin{lemma}
\label{S1 AIM L3}Suppose that we have given $f:S\rightarrow \hom \left(
X,Y\right) $ and, respectively, $g:S\rightarrow \hom \left( X^{\vee
},Y^{\vee }\right) $, that $\alpha _{Y^{\vee },X^{\vee }}$ is an
isomorphism, so that $D_{\iota _{f}}$ is defined, and, respectively, that $%
\alpha _{Y^{\vee \vee },X^{\vee \vee }}$ is an isomorphism and that $X$ is reflexive and $%
\alpha _{Y,X}$ is an isomorphism, so that $D_{\iota _{g}^{\ast }}$ is
defined. Then the first and, respectively, the second of the following
diagrams is commutative:%
\begin{equation*}
\xymatrix{ S\otimes X \ar[r]^-{\varphi_{f}} \ar[d]|{D_{\iota_{f}}\otimes1_{X}} & Y \ar[d]|{i_{Y}} & S\otimes X^{\vee} \ar[d]|{D_{\iota_{g}^{\ast}}\otimes1_{X^{\vee}}} \ar@/^{0.75pc}/[dr]^-{\varphi_{g}} &  \\ X^{\vee}\otimes Y^{\vee\vee}\otimes X \ar[r]^-{\ev_{13,Y^{\vee\vee}}^{\phi}} & Y^{\vee\vee} & X\otimes Y^{\vee}\otimes X^{\vee} \ar[r]^-{\ev_{13,Y^{\vee}}^{\tau}} & Y^{\vee}}
\end{equation*}%
where $\ev_{13,Y^{\vee \vee }}^{\phi }:=\left( 1_{Y^{\vee \vee }}\otimes
\ev_{X}\right) \circ \left( \tau _{X^{\vee },Y^{\vee \vee }}\otimes
1_{X}\right) $ and $\ev_{13,Y^{\vee }}^{\tau }:=\left( 1_{Y^{\vee }}\otimes
\ev_{X}^{\tau }\right) \circ \left( \tau _{X,Y^{\vee }}\otimes 1_{X^{\vee
}}\right) $.
\end{lemma}

\begin{proof}
Consider the following diagram, where we set $t_{Y,X}:=\left( 1_{X^{\vee
}}\otimes i_{Y}\right) \circ \tau _{Y,X^{\vee }}$:%
\begin{equation*}
\xymatrix{ & Y \ar[r]^{i_{Y}} & Y^{\vee\vee} \\ \hom\left(X,Y\right)\otimes
X \ar@/^{0.75pc}/[ur]|{\ev_{X,Y}} \ar[d]|{d_{X,Y}\otimes1_{X}} & Y\otimes
X^{\vee}\otimes X \ar@{}[ul]|(0.4){(A)} \ar@{}[ur]|{(\otimes)}
\ar@{}[dr]|(0.4){(\tau)} \ar@{}[dl]|{(B)} \ar[u]|{1_{Y}\otimes \ev_{X}}
\ar[l]_(0.45){\alpha_{X,Y}\otimes1_{X}}
\ar[r]^(0.45){i_{Y}\otimes1_{X^{\vee}\otimes X}}
\ar[d]|{t_{Y,X}\otimes1_{X}} & Y^{\vee\vee}\otimes X^{\vee}\otimes X
\ar[u]|{1_{Y^{\vee\vee}}\otimes \ev_{X}}
\ar@/^{0.75pc}/[dl]|(0.3){\tau_{Y^{\vee\vee},X^{\vee}}\otimes1_{X}} \\
\hom\left(Y^{\vee},X^{\vee}\right)\otimes X & X^{\vee}\otimes
Y^{\vee\vee}\otimes X \ar[l]_(0.45){\alpha_{Y^{\vee},X^{\vee}}\otimes1_{X}}
& }
\end{equation*}%
The region $\left( A\right) $ is commutative by definition of $\alpha _{X,Y}$
and $\left( B\right) $ thanks to the first diagram in $\left( \text{\ref{S1
D Internal duality 4}}\right) $. Noticing that $\tau _{X^{\vee },Y^{\vee
\vee }}=\left( \tau _{Y^{\vee \vee },X^{\vee }}\right) ^{-1}$ we deduce%
\begin{eqnarray*}
i_{Y}\circ \varphi _{f} &=&i_{Y}\circ \ev_{X,Y}\circ \left( f\otimes
1_{X}\right) =\ev_{13,Y^{\vee \vee }}^{\phi }\circ \left( \alpha _{Y^{\vee
},X^{\vee }}^{-1}\otimes 1_{X}\right) \circ \left( d_{X,Y}\otimes
1_{X}\right) \circ \left( f\otimes 1_{X}\right) \\
&=&\ev_{13,Y^{\vee \vee }}^{\phi }\circ \left( D_{\iota _{f}}\otimes
1_{X}\right) \text{.}
\end{eqnarray*}

Next, we have $i_{Y^{\vee }}\circ \varphi _{g}=\ev_{13,Y^{\vee \vee \vee
}}^{\phi }\circ \left( D_{\iota _{g}}\otimes 1_{X^{\vee }}\right) $ by the
previous computation (because $\alpha _{Y^{\vee \vee },X^{\vee \vee }}$ is
an isomorphism). Applying $\left( i_{Y}\right) ^{\vee }$ we deduce, by Lemma %
\ref{S1 Lemma reflexivity duality},%
\begin{eqnarray*}
\varphi _{g} &=&\left( i_{Y}\right) ^{\vee }\circ \ev_{13,Y^{\vee \vee \vee
}}\circ \left( D_{\iota _{g}}\otimes 1_{X^{\vee }}\right) =\ev_{13,Y^{\vee
}}\circ \left( 1_{X^{\vee \vee }}\otimes \left( i_{Y}\right) ^{\vee }\otimes
1_{X^{\vee }}\right) \circ \left( D_{\iota _{g}}\otimes 1_{X^{\vee }}\right)
\\
&=&\ev_{13,Y^{\vee }}\circ \left( i_{X}\otimes 1_{Y^{\vee }}\otimes
1_{X^{\vee }}\right) \circ \left( i_{X}^{-1}\otimes \left( i_{Y}\right)
^{\vee }\otimes 1_{X^{\vee }}\right) \circ \left( D_{\iota _{g}}\otimes
1_{X^{\vee }}\right) \\
&=&\ev_{13,Y^{\vee }}^{\tau }\circ \left( i_{X}^{-1}\otimes \left(
i_{Y}\right) ^{\vee }\otimes 1_{X^{\vee }}\right) \circ \left( D_{\iota
_{g}}\otimes 1_{X^{\vee }}\right) \text{.}
\end{eqnarray*}%
But it follows from $\left( \text{\ref{S1 AIM D7'}}\right) $ that we have $%
D_{\iota _{g}^{\ast }}=\left( i_{X}^{-1}\otimes \left( i_{Y}\right) ^{\vee
}\right) \circ D_{\iota _{g}}$, proving that the second diagram is
commutative.
\end{proof}

\subsection{Some commutative diagram involving the Casimir element}

If we have given two objects $X$ and $Y$ and $W=W_{1}\otimes W_{2}\otimes
W_{3}\otimes W_{4}$, where $\left( W_{1},W_{2},W_{3},W_{4}\right) $ is a
permutation of the string $\left( X^{\vee },X,Y^{\vee },Y\right) $, we
define morphisms $\ev_{ij,kl}^{\alpha ,\beta }:W\rightarrow \mathbb{I}$,
where $i,j,k,l\in \left\{ 1,2,3,4\right\} $ are such that $i<j$, $k<l$ and $%
i<k$ and $\alpha ,\beta \in \left\{ \phi ,\tau \right\} $, as follows. We
let $ij$ be one of the two pairs for which $ev:W_{i}\otimes W_{j}\rightarrow 
\mathbb{I}$ is defined and we write a corresponding superscript $\alpha
=\phi $ if $W_{j}\in \left\{ X,Y\right\} $ (so that $ev=\ev_{X}$ or $\ev_{Y}$%
)\ or $\alpha =\tau $ if $W_{j}\in \left\{ X^{\vee },Y^{\vee }\right\} $ (so
that $ev=\ev_{X}^{\tau }$ or $ev=\ev_{Y}^{\tau }$); the same rule is applied
to the triple $\left( k,l,\beta \right) $. Then we define%
\begin{equation*}
\ev_{ij,kl}^{\alpha ,\beta }:W\overset{\tau _{\sigma }}{\rightarrow }X^{\vee
}\otimes X\otimes Y^{\vee }\otimes Y\overset{\ev_{X}\otimes \ev_{Y}}{%
\rightarrow }\mathbb{I}\otimes \mathbb{I}\overset{u_{\mathbb{I}}}{%
\rightarrow }\mathbb{I}\text{,}
\end{equation*}%
where $\tau _{\sigma }$ is the morphism obtained from any permutation $%
\sigma $ suitably reordering the factors. We have, for example,%
\begin{eqnarray*}
&&\ev_{12,34}^{\phi ,\phi }:X^{\vee }\otimes X\otimes X^{\vee }\otimes X%
\overset{\ev_{X}\otimes \ev_{X}}{\rightarrow }\mathbb{I}\otimes \mathbb{I}%
\overset{u_{\mathbb{I}}}{\rightarrow }\mathbb{I}\text{.} \\
&&\ev_{14,23}^{\tau ,\phi }:X\otimes X^{\vee }\otimes X\otimes X^{\vee }%
\overset{\tau _{1,2\otimes 3}\otimes 1_{X^{\vee }}}{\rightarrow }X^{\vee
}\otimes X\otimes X\otimes X^{\vee }\overset{\ev_{X}\otimes \ev_{X}^{\tau }}{%
\rightarrow }\mathbb{I}\otimes \mathbb{I}\overset{u_{\mathbb{I}}}{%
\rightarrow }\mathbb{I}\text{.}
\end{eqnarray*}

We say that an object $X$ admits a Casimir element if $X$ is a reflexive
object such that $\epsilon :X^{\vee \vee }\otimes X^{\vee }\rightarrow
\left( X^{\vee }\otimes X\right) ^{\vee }$ is an isomorphism. Then we define
the Casimir element:%
\begin{equation*}
C_{X}:\mathbb{I}\overset{\ev_{X}^{\vee }}{\rightarrow }\left( X^{\vee
}\otimes X\right) ^{\vee }\overset{\epsilon ^{-1}}{\rightarrow }X^{\vee \vee
}\otimes X^{\vee }\overset{i_{X}^{-1}\otimes 1_{X^{\vee }}}{\rightarrow }%
X\otimes X^{\vee }\text{.}
\end{equation*}%
We collect in the following lemma well known properties of the Casimir
element.

\begin{lemma}
\label{S1 Casimir L Properties}Suppose that $X$ has a Casimir element.

\begin{enumerate}
\item[$\left( 1\right) $] $C_{X}$ is the unique morphism making one of the
following diagrams commutative:%
\begin{eqnarray*}
&&\ev_{X}:X^{\vee }\otimes X\overset{1_{X^{\vee }}\otimes C_{X}\otimes 1_{X}}{%
\rightarrow }X^{\vee }\otimes X\otimes X^{\vee }\otimes X\overset{%
\ev_{12,34}^{\phi ,\phi }}{\rightarrow }\mathbb{I}\text{,} \\
&&\ev_{X}:X^{\vee }\otimes X\overset{C_{X}\otimes \tau _{X^{\vee },X}}{%
\rightarrow }X\otimes X^{\vee }\otimes X\otimes X^{\vee }\overset{%
\ev_{14,23}^{\tau ,\phi }}{\rightarrow }\mathbb{I}\text{,} \\
&&\ev_{X}:X^{\vee }\otimes X\overset{\tau _{X^{\vee },X}\otimes C_{X}}{%
\rightarrow }X\otimes X^{\vee }\otimes X\otimes X^{\vee }\overset{%
\ev_{14,23}^{\tau ,\phi }}{\rightarrow }\mathbb{I}\text{.}
\end{eqnarray*}

\item[$\left( 2\right) $] $C_{X}$ is the unique morphism making one of the
following diagrams commutative:%
\begin{eqnarray*}
&&1_{X}:X\overset{C_{X}\otimes 1_{X}}{\rightarrow }X\otimes X^{\vee }\otimes
X\overset{1_{X}\otimes \ev_{X}}{\rightarrow }X\text{,} \\
&&1_{X^{\vee }}:X^{\vee }\overset{1_{X^{\vee }}\otimes C_{X}}{\rightarrow }%
X^{\vee }\otimes X\otimes X^{\vee }\overset{\ev_{X}\otimes 1_{X^{\vee }}}{%
\rightarrow }X^{\vee }\text{.}
\end{eqnarray*}

\item[$\left( 3\right) $] If $X_{1}$, $X_{2}$ and $X_{1}\otimes X_{2}$ have
a Casimir element, then%
\[
C_{X_{1}\otimes X_{2}}=\left( 1_{X_{1}\otimes X_{2}}\otimes \epsilon \right)
\circ \left( 1_{X_{1}}\otimes \tau _{X_{1}^{\vee },X_{2}}\otimes
1_{X_{2}^{\vee }}\right) \circ \left( C_{X_{1}}\otimes C_{X_{2}}\right) 
\text{.}
\]

\item[$\left( 4\right) $] If $X=X^{+}\oplus X^{-}$ is a biproduct
decomposition inducing $X^{\vee }=X^{\vee +}\oplus X^{\vee -}$, then $X^{\pm
}$ both have a Casimir element and $C_{X^{\pm }}=\left( p_{X^{\pm }}\otimes
p_{X^{\vee \pm }}\right) \circ C_{X}$ for the associated surjective
morphisms $p_{X^{\pm }}:X\rightarrow X^{\pm }$ and $p_{X^{\vee \pm
}}:X^{\vee }\rightarrow X^{\vee \pm }$.

\item[$\left( 5\right) $] We have that $X^{\vee }$ has a Casimir element and 
$C_{X^{\vee }}=\left( 1_{X^{\vee }}\otimes i_{X}\right) \circ \tau
_{X,X^{\vee }}\circ C_{X}$.
\end{enumerate}
\end{lemma}

\bigskip

When $X$ has a Casimir element, an explicit inverse of the canonical map $%
f\mapsto \varphi _{f}$ can be given. This is the content of the subsequent
proposition.

\begin{proposition}
\label{S1 Casimir P DefD_f}Suppose that we have given $f:S\rightarrow \hom
\left( X,Y\right) $, which is associated to $\varphi _{f}:S\otimes
X\rightarrow Y$, and that $X$ has a Casimir element. Then the following
diagram is commutative%
\begin{equation*}
\xymatrix{ S \ar[r]^-{f} \ar[d]_{1_{S}\otimes C_{X}} & \hom\left(X,Y\right) \\ S\otimes X\otimes X^{\vee} \ar[r]^-{\varphi_{f}\otimes1_{X^{\vee}}} & Y\otimes X^{\vee}\text{.} \ar[u]_{\alpha_{X,Y}}}
\end{equation*}
\end{proposition}

\begin{proof}
Consider the following diagram%
\begin{equation*}
\xymatrix{ S\otimes X \ar[r]^-{1_{S}\otimes C_{X}\otimes1_{X}} \ar@/_{0.75pc}/[dr]_{1_{S\otimes X}} & S\otimes X\otimes X^{\vee}\otimes X \ar[r]^-{\varphi_{f}\otimes1_{X^{\vee}\otimes X}} \ar[d]|{1_{S\otimes X}\otimes \ev_{X}} & Y\otimes X^{\vee}\otimes X \ar[r]^-{\alpha_{X,Y}\otimes1_{X}} \ar[d]|{1_{Y}\otimes \ev_{X}} & \hom\left(X,Y\right)\otimes X \ar@/^{0.75pc}/[dl]^{\ev_{X,Y}} \\ & S\otimes X \ar[r]^{\varphi_{f}} & Y & }
\end{equation*}%
The first triangle is commutative because $1_{X}:X\overset{C_{X}\otimes 1_{X}%
}{\rightarrow }X\otimes X^{\vee }\otimes X\overset{1_{X}\otimes \ev_{X}}{%
\rightarrow }X$ by Lemma \ref{S1 Casimir L Properties}, the square is
commutative by functoriality of $\otimes $ and the second triangle by
definition of $\alpha _{X,Y}$. But the map $\varphi _{a}$ associated to $%
a:=\alpha _{X,Y}\circ \left( \varphi _{f}\otimes 1_{X^{\vee }}\right) \circ
\left( 1_{S}\otimes C_{X}\right) $ is obtained going from $S\otimes X$ to $%
\hom \left( X,Y\right) \otimes X$ in the upper row and then applying $%
\ev_{X,Y}$. The commutativity implies that this is the morphism $\varphi
_{f}\circ 1_{S\otimes X}=\varphi _{f}$.
\end{proof}

\bigskip

We are mainly concerned with the following consequence of Proposition \ref%
{S1 Casimir P DefD_f}: when $X$ has a Casimir element,%
\begin{equation}
D_{f}:S\overset{1_{S}\otimes C_{X}}{\rightarrow }S\otimes X\otimes X^{\vee }%
\overset{\varphi _{f}\otimes 1_{X^{\vee }}}{\rightarrow }Y\otimes X^{\vee }\text{.}
\label{S1 Casimir D DefD_f}
\end{equation}

\subsection{Behavior of the internal multiplications with respect to tensor
product constructions}

We suppose in this section that we have given $f_{i}:S_{i}\rightarrow \hom
\left( X_{i},Y_{i}\right) $ is associated to $\varphi _{i}=\varphi
_{f_{i}}:S_{i}\otimes X_{i}\rightarrow Y_{i}$ for $i=1,2$. Define the
following morphisms%
\begin{eqnarray*}
&&f_{1}\otimes _{\epsilon }f_{2}:S_{1}\otimes S_{2}\overset{f_{1}\otimes
f_{2}}{\rightarrow }\hom \left( X_{1},Y_{1}\right) \otimes \hom \left(
X_{2},Y_{2}\right) \overset{\epsilon }{\rightarrow }\hom \left( X_{1}\otimes
X_{2},Y_{1}\otimes Y_{2}\right) \text{,} \\
&&f_{1}\otimes _{\epsilon }^{\tau }f_{2}:S_{2}\otimes S_{1}\overset{\tau
_{S_{2},S_{1}}}{\rightarrow }S_{1}\otimes S_{2}\overset{f_{1}\otimes f_{2}}{%
\rightarrow }\hom \left( X_{1},Y_{1}\right) \otimes \hom \left(
X_{2},Y_{2}\right) \overset{\epsilon }{\rightarrow }\hom \left( X_{1}\otimes
X_{2},Y_{1}\otimes Y_{2}\right) \text{,} \\
&&\varphi _{1}\otimes _{\epsilon }\varphi _{2}:S_{1}\otimes S_{2}\otimes
X_{1}\otimes X_{2}\overset{1_{S_{1}\otimes \tau _{2,3}\otimes 1_{X_{2}}}}{%
\rightarrow }S_{1}\otimes X_{1}\otimes S_{2}\otimes X_{2}\overset{\varphi
_{1}\otimes \varphi _{2}}{\rightarrow }Y_{1}\otimes Y_{2}\text{,} \\
&&\varphi _{1}\otimes _{\epsilon }^{\tau }\varphi _{2}:S_{2}\otimes
S_{1}\otimes X_{1}\otimes X_{2}\overset{\tau _{1,2\otimes 3}\otimes 1_{X_{2}}%
}{\rightarrow }S_{1}\otimes X_{1}\otimes S_{2}\otimes X_{2}\overset{\varphi
_{1}\otimes \varphi _{2}}{\rightarrow }Y_{1}\otimes Y_{2}\text{.}
\end{eqnarray*}%
It is easy to see that one has the following result.

\begin{lemma}
\label{S1 Casimir Ltensor1}We have that $\varphi _{1}\otimes _{\epsilon
}\varphi _{2}=\varphi _{f_{1}\otimes _{\epsilon }f_{2}}$ (resp. $\varphi
_{1}\otimes _{\epsilon }^{\tau }\varphi _{2}=\varphi _{f_{1}\otimes
_{\epsilon }^{\tau }f_{2}}$) is the morphism associated to $f_{1}\otimes
_{\epsilon }f_{2}$ (resp. $f_{1}\otimes _{\epsilon }^{\tau }f_{2}$).
\end{lemma}

\bigskip

Next, we consider the associated internal multiplication morphisms $\iota
_{i}:=\iota _{f_{i}}$, for $i=1,2$. The following lemma is easily deduced
from the characterizing property $\left( \text{\ref{S1 AIM D1}}\right) $ of $%
\varphi _{\iota _{i}}$.

\begin{lemma}
\label{S1 Casimir Ltensor2}The following diagram is commutative%
\begin{equation*}
\xymatrix@C150pt{ S_{1}\otimes S_{2}\otimes X_{1}\otimes X_{2}\otimes Y_{1}^{\vee}\otimes Y_{2}^{\vee} \ar[r]^-{\left(1_{X_{1}\otimes X_{2}}\otimes\left(\varphi_{\iota_{1}}\otimes_{\epsilon}\varphi_{\iota_{2}}\right)\right)\circ\left(\tau_{S_{1}\otimes S_{2},X_{1}\otimes X_{2}}\otimes1_{Y_{1}^{\vee}\otimes Y_{2}^{\vee}}\right)} \ar[d]_{\left(\varphi_{1}\otimes_{\epsilon}\varphi_{2}\right)\otimes1_{Y_{1}^{\vee}\otimes Y_{2}^{\vee}}} & X_{1}\otimes X_{2}\otimes X_{1}^{\vee}\otimes X_{2}^{\vee} \ar[d]^{\ev_{13,24}^{\tau,\tau}} \\ Y_{1}\otimes Y_{2}\otimes Y_{1}^{\vee}\otimes Y_{2}^{\vee} \ar[r]^-{\ev_{13,24}^{\tau,\tau}} & \mathbb{I} }
\end{equation*}%
and the same with $\otimes _{\epsilon }$ replaced by $\otimes _{\epsilon
}^{\tau }$ and $S_{1}\otimes S_{2}$ by $S_{2}\otimes S_{1}$.
\end{lemma}

\bigskip

\begin{remark}
\label{S1 Casimir R}It is easy to deduce from Lemma \ref{S1 Casimir Ltensor2}
that $\iota _{f_{1}\otimes _{\epsilon }f_{2}}:S_{1}\otimes S_{2}\rightarrow
\hom \left( \left( Y_{1}\otimes Y_{2}\right) ^{\vee },\left( X_{1}\otimes
X_{2}\right) ^{\vee }\right) $ is associated to $\varphi _{\iota
_{f_{1}\otimes _{\epsilon }f_{2}}}$ making the following diagram commutative:%
\begin{equation*}
\xymatrix{ S_{1}\otimes S_{2}\otimes Y_{1}^{\vee}\otimes Y_{2}^{\vee} \ar[r]^-{\varphi_{\iota_{f_{1}}}\otimes_{\epsilon}\varphi_{\iota_{f_{2}}}} \ar[d]_{1_{S_{1}\otimes S_{2}}\otimes\epsilon} & X_{1}^{\vee}\otimes X_{2}^{\vee} \ar[d]^{\epsilon} \\ S_{1}\otimes S_{2}\otimes\left(Y_{1}\otimes Y_{2}\right)^{\vee} \ar[r]^-{\varphi_{\iota_{f_{1}\otimes_{\epsilon}f_{2}}}} & \left(X_{1}\otimes X_{2}\right)^{\vee}\text{.} }
\end{equation*}%
In particular, when the $\epsilon $ morphisms are isomorphism, we deduce
from Remark \ref{S1 AIM R1} that the commutativity of the diagram of Lemma %
\ref{S1 Casimir Ltensor2} is characterizing for $\varphi _{\iota
_{1}}\otimes _{\epsilon }\varphi _{\iota _{2}}$ (and similarly for $\varphi
_{\iota _{1}}\otimes _{\epsilon }^{\tau }\varphi _{\iota _{2}}$)
\end{remark}

\bigskip

From now on we specialize ourselves to the case where $f_{1}:S_{1}%
\rightarrow \hom \left( X,Y\right) $ and $f_{2}:S_{2}\rightarrow \hom \left(
X^{\vee },Y^{\vee }\right) $. We will assume, from now on, that $\alpha
_{X,Y}$ and $\alpha _{X^{\vee },Y^{\vee }}$ are isomorphisms, so that $%
D_{f_{1}}$ and $D_{f_{2}}$ are defined, and that $X$, $X^{\vee }$ and $Y$
have a Casimir element.

\begin{lemma}
\label{S1 Casimir Ltensor3}The following diagram is commutative:%
\begin{equation*}
\xymatrix{ S_{1}\otimes S_{2} \ar[r]^-{C_{Y}\otimes1_{S_{1}\otimes S_{2}}\otimes C_{X}} \ar[d]_{D_{f_{1}}\otimes D_{f_{2}}} & Y\otimes Y^{\vee}\otimes S_{1}\otimes S_{2}\otimes X\otimes X^{\vee} \ar@/^{0.75pc}/[dr]|-{1_{Y^{\vee}\otimes Y}\otimes\left(\varphi_{f_{1}}\otimes_{\epsilon}\varphi_{f_{2}}\right)} & \\ Y\otimes X^{\vee}\otimes Y^{\vee}\otimes X^{\vee\vee} \ar[r]^-{\ev_{24,Y\otimes Y^{\vee}}^{\tau}} & Y\otimes Y^{\vee} \ar[r]^-{C_{Y}\otimes1_{Y\otimes Y^{\vee}}} & Y\otimes Y^{\vee}\otimes Y\otimes Y^{\vee}\text{,} }
\end{equation*}%
where%
\begin{equation*}
\ev_{24,Y\otimes Y^{\vee }}^{\tau }:Y\otimes X^{\vee }\otimes Y^{\vee
}\otimes X^{\vee \vee }\overset{1_{Y}\otimes \tau _{X^{\vee },Y^{\vee
}}\otimes 1_{X^{\vee \vee }}}{\rightarrow }Y\otimes Y^{\vee }\otimes X^{\vee
}\otimes X^{\vee \vee }\overset{1_{Y\otimes Y^{\vee }}\otimes \ev_{X^{\vee
}}^{\tau }}{\rightarrow }Y\otimes Y^{\vee }\text{.}
\end{equation*}
\end{lemma}

\begin{proof}
Define%
\begin{eqnarray*}
&&D_{12}:S_{1}\otimes S_{2}\overset{1_{S_{1}\otimes S_{2}}\otimes
C_{X}\otimes C_{X^{\vee }}}{\rightarrow }S_{1}\otimes S_{2}\otimes X\otimes
X^{\vee }\otimes X^{\vee }\otimes X^{\vee \vee }\overset{1_{S_{1}\otimes
S_{2}\otimes X}\otimes \tau _{X^{\vee },X^{\vee }}\otimes 1_{X^{\vee \vee }}}%
{\rightarrow } \\
&&S_{1}\otimes S_{2}\otimes X\otimes X^{\vee }\otimes X^{\vee }\otimes
X^{\vee \vee }\overset{\left( \varphi _{f_{1}}\otimes _{\epsilon }\varphi
_{f_{2}}\right) \otimes 1_{X^{\vee }\otimes X^{\vee \vee }}}{\rightarrow }%
Y\otimes Y^{\vee }\otimes X^{\vee }\otimes X^{\vee \vee }
\end{eqnarray*}%
and consider the following diagram:%
\begin{equation*}
\xymatrix{Y\otimes Y^{\vee}\otimes S_{1}\otimes S_{2}\otimes X\otimes
X^{\vee} \ar[rr]^(0.5){1_{Y\otimes
Y^{\vee}}\otimes\left(\varphi_{f_{1}}\otimes_{\epsilon}\varphi_{f_{2}}%
\right)} & \ar@{}[d]|{(A)} & Y\otimes Y^{\vee}\otimes Y\otimes Y^{\vee} \\
S_{1}\otimes S_{2} \ar@{}[ddr]|{(C)} \ar[u]^{C_{Y}\otimes1_{S_{1}\otimes
S_{2}}\otimes C_{X}} \ar[r]^(0.4){1_{S_{1}\otimes S_{2}}\otimes C_{X}}
\ar[dr]|{D_{12}} \ar@/_{1pc}/[ddr]_{D_{f_{1}}\otimes D_{f_{2}}} &
S_{1}\otimes S_{2}\otimes X\otimes X^{\vee} \ar@{}[d]|(0.4){(B)}
\ar[r]^(0.6){\varphi_{f_{1}}\otimes_{\epsilon}\varphi_{f_{2}}} & Y\otimes
Y^{\vee} \ar@{}[ddl]|{(D)} \ar[u]_{C_{Y}\otimes1_{Y\otimes Y^{\vee}}}\\ &
Y\otimes Y^{\vee}\otimes X^{\vee}\otimes X^{\vee\vee} \ar[ur]|{1_{Y\otimes
Y^{\vee}}\otimes \ev_{X^{\vee}}^{\tau}} & \\ & Y\otimes X^{\vee}\otimes
Y^{\vee}\otimes X^{\vee\vee}
\ar[u]|{1_{Y}\otimes\tau_{X^{\vee},Y^{\vee}}\otimes1_{X^{\vee\vee}}}
\ar@/_{1pc}/[uur]_{\ev_{24,Y\otimes Y^{\vee}}^{\tau}} & }
\end{equation*}%
The region $\left( A\right) $ is commutative because $\left( C_{Y}\otimes
1_{Y\otimes Y^{\vee }}\right) \circ \left( \varphi _{f_{1}}\otimes
_{\epsilon }\varphi _{f_{2}}\right) =\left( 1_{Y\otimes Y^{\vee }}\otimes
\left( \varphi _{f_{1}}\otimes _{\epsilon }\varphi _{f_{2}}\right) \right)
\circ \left( C_{Y}\otimes 1_{S_{1}\otimes S_{2}\otimes X\otimes X^{\vee
}}\right) $ by functoriality of $\otimes $ and then%
\begin{equation*}
\left( C_{Y}\otimes 1_{Y\otimes Y^{\vee }}\right) \circ \left( \varphi
_{f_{1}}\otimes _{\epsilon }\varphi _{f_{2}}\right) \circ \left(
1_{S_{1}\otimes S_{2}}\otimes C_{X}\right) =\left( 1_{Y\otimes Y^{\vee
}}\otimes \left( \varphi _{f_{1}}\otimes _{\epsilon }\varphi _{f_{2}}\right)
\right) \circ \left( C_{Y}\otimes 1_{S_{1}\otimes S_{2}}\otimes C_{X}\right) 
\text{.}
\end{equation*}%
The region $\left( D\right) $ is commutative by definition of $%
\ev_{24,Y\otimes Y^{\vee }}^{\tau }$. We claim that the regions $\left(
B\right) $ and $\left( C\right) $ are commutative, from which we will deduce
that the external portion of this diagram is commutative, giving us the
claim.

\textbf{Region }$\left( B\right) $\textbf{\ is commutative.} Consider the
following diagram%
\begin{equation*}
\xymatrix{S_{1}\otimes S_{2} \ar@{}[dr]|(0.3){(\otimes)}
\ar[d]_{1_{S_{1}\otimes S_{2}}\otimes C_{X}\otimes C_{X^{\vee}}}
\ar[r]^{1_{S_{1}\otimes S_{2}}\otimes C_{X}} & S_{1}\otimes S_{2}\otimes
X\otimes X^{\vee} \ar[dl]|{1\otimes C_{X^{\vee}}}
\ar[ddr]^{\varphi_{f_{1}}\otimes_{\epsilon}\varphi_{f_{2}}} & \\
S_{1}\otimes S_{2}\otimes X\otimes X^{\vee}\otimes X^{\vee}\otimes
X^{\vee\vee} \ar@{}[dr]|(0.2){(E)} \ar[d]_{1_{S_{1}\otimes S_{2}\otimes
X}\otimes\tau_{X^{\vee},X^{\vee}}\otimes1_{X^{\vee\vee}}} & & \\
S_{1}\otimes S_{2}\otimes X\otimes X^{\vee}\otimes X^{\vee}\otimes
X^{\vee\vee} \ar[uur]|(0.6){1\otimes \ev_{X^{\vee}}^{\tau}}
\ar[r]^(0.6){\left(\varphi_{f_{1}}\otimes_{\epsilon}\varphi_{f_{2}}\right)%
\otimes1} & Y\otimes Y^{\vee}\otimes X^{\vee}\otimes X^{\vee\vee}
\ar@{}[uu]|{(\otimes)} \ar[r]^(0.7){1\otimes \ev_{X^{\vee}}^{\tau}} &
Y\otimes Y^{\vee}}
\end{equation*}%
In light of the definition of $D_{12}$ the commutativity of the region $%
\left( B\right) $ follows once we show that the external portion of this
diagram is commutative. We remark that, by functoriality of $\tau $, by an
explicit computation of the involved permutations and by definition of $%
\ev_{X^{\vee }}^{\tau }$ the following diagram is commutative:%
\begin{equation*}
\xymatrix{ X^{\vee} \ar[r]^(0.33){C_{X^{\vee}}\otimes1_{X^{\vee}}}
\ar[dr]|(0.4){1_{X^{\vee}}\otimes C_{X^{\vee}}} & X^{\vee}\otimes
X^{\vee\vee}\otimes X^{\vee} \ar[rr]^{1_{X^{\vee}}\otimes \ev_{X^{\vee}}}
\ar[dr]^{1_{X^{\vee}}\otimes\tau_{X^{\vee\vee},X^{\vee}}}
\ar[d]|{\tau_{X^{\vee}\otimes X^{\vee\vee},X^{\vee}}} & & X^{\vee} \\ &
X^{\vee}\otimes X^{\vee}\otimes X^{\vee\vee}
\ar[r]^{\tau_{X^{\vee},X^{\vee}}\otimes1_{X^{\vee\vee}}} & X^{\vee}\otimes
X^{\vee}\otimes X^{\vee\vee} \ar[ur]^{1_{X^{\vee}}\otimes
\ev_{X^{\vee}}^{\tau}} & }
\end{equation*}%
It follows that $\left( 1_{X^{\vee }}\otimes \ev_{X^{\vee }}^{\tau }\right)
\circ \left( \tau _{X^{\vee },X^{\vee }}\otimes 1_{X^{\vee \vee }}\right)
\circ \left( 1_{X^{\vee }}\otimes C_{X^{\vee }}\right) =\left( 1_{X^{\vee
}}\otimes \ev_{X^{\vee }}\right) \circ \left( C_{X^{\vee }}\otimes 1_{X^{\vee
}}\right) =1_{X^{\vee }}$, where the last equality follows from $2.$ of
Lemma \ref{S1 Casimir L Properties}. The commutativity of the region $\left(
E\right) $ follows.

\textbf{Region }$\left( C\right) $\textbf{\ is commutative.} Consider the
following diagram%
\begin{equation*}
\xymatrix{ & S_{1}\otimes S_{2} \ar@{}[d]|{(\tau)} \ar@{}[dd]|(0.7){(F)}
\ar[dl]_{1_{S_{1}\otimes S_{2}}\otimes C_{X}\otimes C_{X^{\vee}}}
\ar[dr]^{1_{S_{1}}\otimes C_{X}\otimes1_{S_{2}}\otimes C_{X^{\vee}}} & \\
S_{1}\otimes S_{2}\otimes X\otimes X^{\vee}\otimes X^{\vee}\otimes
X^{\vee\vee} \ar[dd]|{1_{S_{1}\otimes S_{2}\otimes
X}\otimes\tau_{X^{\vee},X^{\vee}}\otimes1_{X^{\vee\vee}}}
\ar[rr]^{1_{S_{1}}\otimes\tau_{S_{2},X\otimes
X^{\vee}}\otimes1_{X^{\vee}\otimes X^{\vee\vee}}}
\ar[dr]|{1_{S_{1}}\otimes\tau_{S_{2},X}\otimes\tau_{X^{\vee},X^{\vee}}%
\otimes1_{X^{\vee\vee}}} & & S_{1}\otimes X\otimes X^{\vee}\otimes
S_{2}\otimes X^{\vee}\otimes X^{\vee\vee} \ar[dl]|{1_{S_{1}\otimes
X}\otimes\tau_{X^{\vee},S_{2}\otimes X^{\vee}}\otimes1_{X^{\vee\vee}}}
\ar[d]|{\varphi_{f_{1}}\otimes1_{X^{\vee}}\otimes\varphi_{f_{2}}%
\otimes1_{X^{\vee\vee}}} \\ \ar@{}[r]|(0.45){(\otimes)} & S_{1}\otimes
X\otimes S_{2}\otimes X^{\vee}\otimes X^{\vee}\otimes X^{\vee\vee}
\ar[dr]|{\varphi_{f_{1}}\otimes\varphi_{f_{2}}\otimes1_{X^{\vee}\otimes
X^{\vee\vee}}} & Y\otimes X^{\vee}\otimes Y^{\vee}\otimes X^{\vee\vee}
\ar@{}[l]|{(\tau)}
\ar[d]|{1_{Y}\otimes\tau_{X^{\vee},Y^{\vee}}\otimes1_{X^{\vee\vee}}} \\
S_{1}\otimes S_{2}\otimes X\otimes X^{\vee}\otimes X^{\vee}\otimes
X^{\vee\vee}
\ar[ur]|{1_{S_{1}}\otimes\tau_{S_{2},X}\otimes1_{X^{\vee}\otimes
X^{\vee}\otimes X^{\vee\vee}}}
\ar[rr]^{\left(\varphi_{f_{1}}\otimes_{\epsilon}\varphi_{f_{2}}\right)%
\otimes1_{X^{\vee}\otimes X^{\vee\vee}}} & \ar@{}[u]|{(G)} & Y\otimes
Y^{\vee}\otimes X^{\vee}\otimes X^{\vee\vee}}
\end{equation*}%
Going around this diagram clockwisely (resp. counter-clockwisely) from $%
S_{1}\otimes S_{2}$ until $Y\otimes Y^{\vee }\otimes X^{\vee }\otimes
X^{\vee \vee }$ we find $\left( 1_{Y}\otimes \tau _{X^{\vee },Y^{\vee
}}\otimes 1_{X^{\vee \vee }}\right) \circ \left( D_{f_{1}}\otimes
D_{f_{2}}\right) $ (resp. $D_{12}$) because, by $\left( \text{\ref{S1
Casimir D DefD_f}}\right) $, we have $D_{f_{i}}=\left( \varphi
_{f_{i}}\otimes 1_{X_{i}^{\vee }}\right) \circ \left( 1_{S_{i}}\otimes
C_{X_{i}}\right) $ for $i=1,2$, where $X_{1}=X$ and $X_{2}=X^{\vee }$ (resp.
by definition of $D_{12}$). It follows that we have to show that the
external portion of this diagram is commutative. The region $\left( F\right) 
$ is commutative by an explicit computation of the involved permutations,
while $\left( G\right) $ by definition of $\left( \varphi _{f_{1}}\otimes
_{\epsilon }\varphi _{f_{2}}\right) $.
\end{proof}

\bigskip

In addition to the other assumptions we will assume in the following
proposition that $\alpha _{Y^{\vee },X^{\vee }}$ and $\alpha _{Y^{\vee \vee
},X^{\vee \vee }}$ are isomorphisms, so that $D_{\iota _{f_{1}}}$ and $%
D_{\iota _{f_{2}}}$ are defined, and that $Y^{\vee }$, $Y^{\vee \vee }$ have
a Casimir element.

\begin{proposition}
\label{S1 Casimir P1}The following diagrams are commutative.

\begin{enumerate}
\item[$\left( 1\right) $]
\begin{equation*}
\xymatrix{ S_{1}\otimes S_{2} \ar[rr]^-{C_{Y}\otimes1_{S_{1}\otimes S_{2}}\otimes C_{X}} \ar[d]_{D_{f_{1}}\otimes D_{f_{2}}} & & Y\otimes Y^{\vee}\otimes S_{1}\otimes S_{2}\otimes X\otimes X^{\vee} \ar[d]^{1_{Y^{\vee}\otimes Y}\otimes\left(\varphi_{f_{1}}\otimes_{\epsilon}\varphi_{f_{2}}\right)} \\ Y\otimes X^{\vee}\otimes Y^{\vee}\otimes X^{\vee\vee} \ar[r]^-{\ev_{13,24}^{\tau,\tau}} & \mathbb{I} & Y\otimes Y^{\vee}\otimes Y\otimes Y^{\vee} \ar[l]_-{\ev_{14,23}^{\tau,\phi}}}
\end{equation*}

\item[$\left( 2\right) $]
\begin{equation*}
\xymatrix{ S_{1}\otimes S_{2} \ar[rr]^-{C_{X^{\vee}}\otimes1_{S_{1}\otimes S_{2}}\otimes C_{Y^{\vee}}} \ar[d]_{D_{\iota_{f_{1}}}\otimes D_{\iota_{f_{2}}}} & & X^{\vee}\otimes X^{\vee\vee}\otimes S_{1}\otimes S_{2}\otimes Y^{\vee}\otimes Y^{\vee\vee} \ar[d]^{1_{X^{\vee}\otimes X^{\vee\vee}}\otimes\left(\varphi_{\iota_{f_{1}}}\otimes_{\epsilon}\varphi_{\iota_{f_{1}}}\right)} \\ X^{\vee}\otimes Y^{\vee\vee}\otimes X^{\vee\vee}\otimes Y^{\vee\vee\vee} \ar[r]^-{\ev_{13,24}^{\tau,\tau}} & \mathbb{I} & X^{\vee}\otimes X^{\vee\vee}\otimes X^{\vee}\otimes X^{\vee\vee} \ar[l]_-{\ev_{14,23}^{\tau,\phi}} }
\end{equation*}

\item[$\left( 3\right) $]
\begin{equation*}
\xymatrix{ S_{1}\otimes S_{2} \ar[rr]^-{1_{S_{1}\otimes S_{2}}\otimes C_{X}\otimes C_{Y^{\vee}}} \ar[d]_{D_{\iota_{f_{1}}}\otimes D_{\iota_{f_{2}}}} & & S_{1}\otimes S_{2}\otimes X\otimes X^{\vee}\otimes Y^{\vee}\otimes Y^{\vee\vee} \ar[d]^{\left(\varphi_{f_{1}}\otimes_{\epsilon}\varphi_{f_{1}}\right)\otimes1_{Y^{\vee}\otimes Y^{\vee\vee}}} \\ X^{\vee}\otimes Y^{\vee\vee}\otimes X^{\vee\vee}\otimes Y^{\vee\vee\vee} \ar[r]^-{\ev_{13,24}^{\tau,\tau}} & \mathbb{I} & Y\otimes Y^{\vee}\otimes Y^{\vee}\otimes Y^{\vee\vee} \ar[l]_-{\ev_{13,24}^{\tau,\tau}} }
\end{equation*}
\end{enumerate}
\end{proposition}

\begin{proof}
$\left( 1\right) $ According to Lemma \ref{S1 Casimir Ltensor3} we have%
\begin{equation*}
\ev_{14,23}^{\tau ,\phi }\circ \left( 1_{Y^{\vee }\otimes Y}\otimes \left(
\varphi _{f_{1}}\otimes _{\epsilon }\varphi _{f_{2}}\right) \right) \circ
\left( C_{Y}\otimes 1_{S_{1}\otimes S_{2}}\otimes C_{X}\right)
=\ev_{14,23}^{\tau ,\phi }\circ \left( C_{Y}\otimes 1_{Y\otimes Y^{\vee
}}\right) \circ \ev_{24,Y\otimes Y^{\vee }}^{\tau }\circ \left(
D_{f_{1}}\otimes D_{f_{2}}\right) \text{.}
\end{equation*}%
It follows from Lemma \ref{S1 Casimir L Properties} $\left( 1\right) $ we have $%
\ev_{Y}=\ev_{14,23}^{\tau ,\phi }\circ \left( C_{Y}\otimes \tau _{Y^{\vee
},Y}\right) =\ev_{14,23}^{\tau ,\phi }\circ \left( C_{Y}\otimes 1_{Y\otimes
Y^{\vee }}\right) \circ \tau _{Y^{\vee },Y}$ and, by definition, $%
\ev_{24,Y\otimes Y^{\vee }}^{\tau }=\left( 1_{Y\otimes Y^{\vee }}\otimes
\ev_{X^{\vee }}^{\tau }\right) \circ \left( 1_{Y}\otimes \tau _{X^{\vee
},Y^{\vee }}\otimes 1_{X^{\vee \vee }}\right) $. We deduce%
\begin{eqnarray*}
&&\ev_{14,23}^{\tau ,\phi }\circ \left( C_{Y}\otimes 1_{Y\otimes Y^{\vee
}}\right) \circ \ev_{24,Y\otimes Y^{\vee }}^{\tau }=\ev_{14,23}^{\tau ,\phi
}\circ \left( C_{Y}\otimes 1_{Y\otimes Y^{\vee }}\right) \circ \tau
_{Y^{\vee },Y}\circ \tau _{Y,Y^{\vee }}\circ \ev_{24,Y\otimes Y^{\vee
}}^{\tau } \\
&&\text{ }=\ev_{Y}\circ \tau _{Y,Y^{\vee }}\circ \left( 1_{Y\otimes Y^{\vee
}}\otimes \ev_{X^{\vee }}^{\tau }\right) \circ \left( 1_{Y}\otimes \tau
_{X^{\vee },Y^{\vee }}\otimes 1_{X^{\vee \vee }}\right) =\ev_{13,24}^{\tau
,\tau }\text{.}
\end{eqnarray*}%
Hence we find%
\begin{equation*}
\ev_{14,23}^{\tau ,\phi }\circ \left( 1_{Y^{\vee }\otimes Y}\otimes \left(
\varphi _{f_{1}}\otimes _{\epsilon }\varphi _{f_{2}}\right) \right) \circ
\left( C_{Y}\otimes 1_{S_{1}\otimes S_{2}}\otimes C_{X}\right)
=\ev_{13,24}^{\tau ,\tau }\circ \left( D_{f_{1}}\otimes D_{f_{2}}\right) 
\text{.}
\end{equation*}

$\left( 2\right) $ This is just our claim $1.$ applied to the couple $\left( \iota
_{f_{1}},\iota _{f_{2}}\right) $ rather than $\left( f_{1},f_{2}\right) $.

$\left( 3\right) $ Consider the following diagram:%
\begin{equation}
\scalebox{0.65}{\xymatrix{S_{1}\otimes S_{2}\otimes X\otimes X^{\vee}\otimes
Y^{\vee}\otimes Y^{\vee\vee} \ar@{}[rrrrdd]|{(B)}
\ar[dddd]|{\left(\varphi_{f_{1}}\otimes_{\epsilon}\varphi_{f_{2}}\right)%
\otimes1_{Y^{\vee}\otimes Y^{\vee\vee}}} \ar[rrrrr]^{\tau_{S_{1}\otimes
S_{2},X\otimes X^{\vee}}\otimes1_{Y^{\vee}\otimes Y^{\vee\vee}}} & & & & &
X\otimes X^{\vee}\otimes S_{1}\otimes S_{2}\otimes Y^{\vee}\otimes
Y^{\vee\vee} \ar[ddddd]|{1_{X\otimes
X^{\vee}}\otimes\left(\varphi_{\iota_{f_{1}}}\otimes_{\epsilon}\varphi_{%
\iota_{f_{2}}}\right)} \\ & & & & X^{\vee}\otimes X\otimes S_{1}\otimes
S_{2}\otimes Y^{\vee}\otimes Y^{\vee\vee} \ar@{}[dddr]|{(\tau)}
\ar[ur]|(0.45){\tau_{X^{\vee},X}\otimes1_{S_{1}\otimes S_{2}\otimes
Y^{\vee}\otimes Y^{\vee\vee}}} \ar[ddd]|{1_{X^{\vee}\otimes
X}\otimes\left(\varphi_{\iota_{f_{1}}}\otimes_{\epsilon}\varphi_{%
\iota_{f_{2}}}\right)} \ar[dl]|{1_{X^{\vee}}\otimes
i_{X}\otimes1_{S_{1}\otimes S_{2}\otimes Y^{\vee}\otimes Y^{\vee\vee}}} & \\
& S_{1}\otimes S_{2} \ar@{}[drr]|{(A)} \ar[uul]|{1_{S_{1}\otimes
S_{2}}\otimes C_{X}\otimes C_{Y^{\vee}}}
\ar[rr]^(0.35){C_{X^{\vee}}\otimes1_{S_{1}\otimes S_{2}}\otimes
C_{Y^{\vee}}} \ar[d]_{D_{\iota_{f_{1}}}\otimes D_{\iota_{f_{2}}}} & &
X^{\vee}\otimes X^{\vee\vee}\otimes S_{1}\otimes S_{2}\otimes
Y^{\vee}\otimes Y^{\vee\vee} \ar@{}[dr]|{(\tau)} \ar[d]|{1_{X^{\vee}\otimes
X^{\vee\vee}}\otimes\left(\varphi_{\iota_{f_{1}}}\otimes_{\epsilon}\varphi_{%
\iota_{f_{2}}}\right)} & & \\ & X^{\vee}\otimes Y^{\vee\vee}\otimes
X^{\vee\vee}\otimes Y^{\vee\vee\vee} \ar[r]^(0.8){\ev_{13,24}^{\tau,\tau}} &
\mathbb{I} & X^{\vee}\otimes X^{\vee\vee}\otimes X^{\vee}\otimes
X^{\vee\vee} \ar[l]_(0.7){\ev_{14,23}^{\phi,\tau}} & & \\ Y\otimes
Y^{\vee}\otimes Y^{\vee}\otimes Y^{\vee\vee}
\ar@/_{1pc}/[urr]_{\ev_{13,24}^{\tau,\tau}} & & & & X^{\vee}\otimes X\otimes
X^{\vee}\otimes X^{\vee\vee} \ar@{}[dl]|(0.3){(C)} \ar@{}[ull]|(0.6){(D)}
\ar[ul]|{1_{X^{\vee}}\otimes i_{X}\otimes1_{X^{\vee}\otimes X^{\vee\vee}}}
\ar@/^{1pc}/[ull]^{\ev_{14,23}^{\tau,\tau}}
\ar[dr]|{\tau_{X^{\vee},X}\otimes1_{X^{\vee}\otimes X^{\vee\vee}}} & \\ & &
& & & X\otimes X^{\vee}\otimes X^{\vee}\otimes X^{\vee\vee}
\ar@/^{2pc}/[uulll]^{\ev_{13,24}^{\tau,\tau}}}}  \label{S1 Casimir P1 D1}
\end{equation}%
We remark that our claim is the commutativity of the unlabeled region of $%
\left( \text{\ref{S1 Casimir P1 D1}}\right) $ and that this commutativity
follows once we show that the external part and the labeled regions of $%
\left( \text{\ref{S1 Casimir P1 D1}}\right) $ are commutative. The external
part of this diagram is commutative by Lemma \ref{S1 Casimir Ltensor2}.

\textbf{Commutativity of the labeled regions of }$\left( \text{\ref{S1
Casimir P1 D1}}\right) $. The region $\left( A\right) $ is commutative by
our claim $\left( 2\right)$, the region $\left( C\right) $ by the defining property of $%
i_{X}$, the region $\left( D\right) $ by the definitions of $%
\ev_{14,23}^{\tau ,\tau }$ and $\ev_{13,24}^{\tau ,\tau }$. The commutativity
of the region $\left( B\right) $ follows the equality $C_{X^{\vee }}=\left(
1_{X^{\vee }}\otimes i_{X}\right) \circ \tau _{X,X^{\vee }}\circ C_{X}$ of
Lemma \ref{S1 Casimir L Properties} $\left( 5\right) $, from which we deduce that%
\begin{equation*}
\left( C_{X^{\vee }}\otimes 1_{S_{1}\otimes S_{2}}\otimes C_{Y^{\vee
}}\right) =\left( 1_{X^{\vee }}\otimes i_{X}\otimes 1_{S_{1}\otimes
S_{2}\otimes Y^{\vee }\otimes Y^{\vee \vee }}\right) \circ \left( \tau
_{X,X^{\vee }}\otimes 1_{S_{1}\otimes S_{2}\otimes Y^{\vee }\otimes Y^{\vee
\vee }}\right) \circ \left( C_{X}\otimes 1_{S_{1}\otimes S_{2}}\otimes
C_{Y^{\vee }}\right) \text{,}
\end{equation*}%
together with the equality $\left( C_{X}\otimes 1_{S_{1}\otimes
S_{2}}\otimes C_{Y^{\vee }}\right) =\left( \tau _{S_{1}\otimes
S_{2},X\otimes X^{\vee }}\otimes 1_{Y^{\vee }\otimes Y^{\vee \vee }}\right)
\circ \left( 1_{S_{1}\otimes S_{2}}\otimes C_{X}\otimes C_{Y^{\vee }}\right) 
$ (by functoriality of $\tau $) and the fact that $\tau _{X^{\vee },X}=\tau
_{X,X^{\vee }}^{-1}$.
\end{proof}

\section{A formal Poincar\'{e} duality isomorphism}

We remark that, for every object $W$, we have a natural map%
\begin{equation*}
End\left( \mathbb{I}\right) \rightarrow End\left( W\right)
\end{equation*}%
defined by the rule%
\begin{equation*}
\lambda _{W}:W\overset{l_{W}}{\rightarrow }\mathbb{I}\otimes W\overset{%
\lambda \otimes 1_{W}}{\rightarrow }\mathbb{I}\otimes W\overset{l_{W}^{-1}}{%
\rightarrow }W\text{, }\lambda \in End\left( \mathbb{I}\right) \text{.}
\end{equation*}%
It defines a left action of the commutative ring $End\left( \mathbb{I}%
\right) $ on $W$ for which every $f:W_{1}\rightarrow W_{2}$ becomes $%
End\left( \mathbb{I}\right) $-equivariant and such that, if we have given $%
\varphi :U\otimes V\rightarrow W$, then $\varphi \circ \left( \lambda
_{U}\otimes 1_{V}\right) =\lambda _{W}\circ \varphi =\varphi \circ \left(
1_{U}\otimes \lambda _{V}\right) $\footnote{%
Since indeed $\lambda _{U}\otimes 1_{V}=\lambda _{U\otimes V}=1_{U}\otimes
\lambda _{V}$, this second statement follows from the first.}.

\bigskip

Suppose in this section that $\mathcal{C}$ is rigid. We assume that we have given  morphisms $f_{S,X}:S\rightarrow \hom \left( X,Y\right) $, $%
f_{X,S}:X\rightarrow \hom \left( S,Y\right) $, $f_{S^{\vee },X^{\vee
}}:S^{\vee }\rightarrow \hom \left( X^{\vee },Y^{\vee }\right) $ and $%
f_{X^{\vee },S^{\vee }}:X^{\vee }\rightarrow \hom \left( S^{\vee },Y^{\vee
}\right) $. We write $\varphi _{S,X}:S\otimes X\rightarrow Y$, $\varphi
_{X,S}:X\otimes S\rightarrow Y$, $\varphi _{S^{\vee },X^{\vee }}:S^{\vee
}\otimes X^{\vee }\rightarrow Y^{\vee }$ and $\varphi _{X^{\vee },S^{\vee
}}:X^{\vee }\otimes S^{\vee }\rightarrow Y^{\vee }$ for the associated
morphisms. We set $\iota _{S,X}:=\iota _{f_{S,X}}$, $\iota _{X,S}:=\iota
_{f_{X,S}}$, $\iota _{S^{\vee },X^{\vee }}:=\iota _{f_{S^{\vee },X^{\vee }}}$
and $\iota _{X^{\vee },S^{\vee }}:=\iota _{f_{X^{\vee },S^{\vee }}}$. We may consider:%
\begin{eqnarray*}
&&D_{S,X^{\vee }}:=D_{\iota _{S,X}}:S\rightarrow X^{\vee }\otimes Y^{\vee
\vee }\text{, }D_{X,S^{\vee }}:=D_{\iota _{X,S}}:X\rightarrow S^{\vee
}\otimes Y^{\vee \vee }\text{,} \\
&&D_{\iota _{S^{\vee },X^{\vee }}}:S^{\vee }\rightarrow X^{\vee \vee
}\otimes Y^{\vee \vee \vee }\text{ and }D_{\iota _{X^{\vee },S^{\vee
}}}:X^{\vee }\rightarrow S^{\vee \vee }\otimes Y^{\vee \vee \vee }\text{,}
\end{eqnarray*}%
as well as
\begin{equation*}
D_{S^{\vee },X}:=D_{\iota _{S^{\vee },X^{\vee }}^{\ast }}:S^{\vee
}\rightarrow X\otimes Y^{\vee }\text{ and }D_{X^{\vee },S}:=D_{\iota
_{X^{\vee },S^{\vee }}^{\ast }}:X^{\vee }\rightarrow S\otimes Y^{\vee }\text{%
.}
\end{equation*}
We note that all the results of
the previous section available.

\bigskip

It is easy to deduce, from $\left( \text{\ref{S1 AIM D1 reflexivity 2}}\right) $ and 
\ref{S1 Casimir L Properties} $\left( 5\right) $, the following equivalence:
\begin{equation*}
\left( Cas\right) _{\mu _{S,X}}:\mu _{S,X}\cdot C_{S}=\left( \varphi _{\iota
_{X^{\vee },S^{\vee }}^{\ast }}\otimes _{\epsilon }^{\tau }\varphi _{\iota
_{X,S}}\right) \circ \left( C_{X}\otimes C_{Y}\right) \text{ }%
\Leftrightarrow \text{ }\mu _{S,X}\cdot C_{S^{\vee }}=\left( \varphi _{\iota
_{X,S}}\otimes _{\epsilon }\varphi _{\iota _{X^{\vee },S^{\vee }}}\right)
\circ \left( C_{X}\otimes C_{Y^{\vee }}\right)
\end{equation*}%
for some $\mu _{S,X}\in End\left( \mathbb{I}\right) $. Exchanging the roles of $S$ and $X$ we also have, for some $\mu _{X,S}\in
End\left( \mathbb{I}\right) $,%
\begin{equation*}
\left( Cas\right) _{\mu _{X,S}}:\mu _{X,S}\cdot C_{X}=\left( \varphi _{\iota
_{S^{\vee },X^{\vee }}^{\ast }}\otimes _{\epsilon }^{\tau }\varphi _{\iota
_{S,X}}\right) \circ \left( C_{S}\otimes C_{Y}\right) \text{ }%
\Leftrightarrow \text{ }\mu _{X,S}\cdot C_{X^{\vee }}=\left( \varphi _{\iota
_{S,X}}\otimes _{\epsilon }\varphi _{\iota _{S^{\vee },X^{\vee }}}\right)
\circ \left( C_{S}\otimes C_{Y^{\vee }}\right) \text{.}
\end{equation*}

\bigskip

If $\left( V,W\right) =\left( S,X\right) $ or $\left( X,S\right) $ and $%
\lambda _{V,W},\lambda _{V^{\vee },W^{\vee }}\in End\left( \mathbb{I}\right) 
$, we will consider the following diagrams:%
\begin{equation*}
\xymatrix{ V\otimes W \ar@{}[dr]|{\left(Com\right)_{\lambda_{V,W}}} \ar[r]^-{\tau_{V,W}} \ar[d]_{\varphi_{V,W}} & W\otimes V \ar[d]^{\varphi_{W,V}} & V^{\vee}\otimes W^{\vee} \ar@{}[dr]|{\left(Com\right)_{\lambda_{V^{\vee},W^{\vee}}}} \ar[r]^-{\tau_{V^{\vee},W^{\vee}}} \ar[d]_{\varphi_{V^{\vee},W^{\vee}}} & W^{\vee}\otimes V^{\vee} \ar[d]^{\varphi_{W^{\vee},V^{\vee}}} \\ Y \ar[r]^-{\lambda_{V,W}} & Y\text{,} & Y^{\vee} \ar[r]^-{\lambda_{V^{\vee},W^{\vee}}} & Y\text{.}}
\end{equation*}%
It will be convenient to introduce the following shorthand. We set $\left[ W%
\right] :=W\otimes W^{\vee }$, $\varphi _{\left[ S\right] ,\left[ X\right]
}:=\varphi _{S,X}\otimes _{\epsilon }\varphi _{S^{\vee },X^{\vee }}$, $%
\varphi _{\left[ X\right] ,\left[ S\right] }:=\varphi _{X,S}\otimes
_{\epsilon }\varphi _{X^{\vee },S^{\vee }}$, $\varphi _{\iota _{\left[ S%
\right] ,\left[ X\right] }}:=\varphi _{\iota _{S,X}}\otimes _{\epsilon
}\varphi _{\iota _{S^{\vee },X^{\vee }}}$, $\varphi _{\iota _{\left[ X\right]
,\left[ S\right] }}:=\varphi _{\iota _{X,S}}\otimes _{\epsilon }\varphi
_{\iota _{X^{\vee },S^{\vee }}}$, $\varphi _{\iota _{\left[ S\right] ,\left[
X\right] }}^{\tau }:=\varphi _{\iota _{S^{\vee },X^{\vee }}^{\ast }}\otimes
_{\epsilon }^{\tau }\varphi _{\iota _{S,X}}$ and $\varphi _{\iota _{\left[ X%
\right] ,\left[ S\right] }}^{\tau }:=\varphi _{\iota _{X^{\vee },S^{\vee
}}^{\ast }}\otimes _{\epsilon }^{\tau }\varphi _{\iota _{X,S}}$. If $\left(
V,W\right) =\left( S,X\right) $ or $\left( X,S\right) $ and $\lambda _{\left[
V\right] ,\left[ W\right] }\in End\left( \mathbb{I}\right) $, we will
consider the following diagram:%
\begin{equation*}
\xymatrix{\left[V\right]\otimes\left[W\right] \ar@{}[dr]|{\left(Com\right)_{\lambda_{\left[V\right],\left[W\right]}}} \ar[r]^-{\tau_{\left[V\right],\left[W\right]}} \ar[d]_{\varphi_{\left[V\right],\left[W\right]}} & \left[W\right]\otimes\left[V\right] \ar[d]^{\varphi_{\left[W\right],\left[V\right]}} \\ \left[Y\right] \ar[r]^-{\lambda_{\left[V\right],\left[W\right]}} & \left[Y\right]\text{.} }
\end{equation*}

\begin{remark}
\label{FPD R1}We have%
\begin{equation*}
\left( Com\right) _{\lambda _{V,W}}\text{ and }\left( Com\right) _{\lambda
_{V^{\vee },W^{\vee }}}\text{ commutative }\Rightarrow \text{ }\left(
Com\right) _{\lambda _{\left[ V\right] ,\left[ W\right] }}\text{ with }%
\lambda _{\left[ V\right] ,\left[ W\right] }=\lambda _{V,W}\cdot \lambda
_{V^{\vee },W^{\vee }}\text{.}
\end{equation*}
\end{remark}

\bigskip

\begin{proposition}
\label{FPD P1}If $\left( Cas\right) _{\mu _{S,X}}$ is satisfied and $\left(
Com\right) _{\lambda _{\left[ S\right] ,\left[ X\right] }}$ is commutative,
then the following diagrams are commutative:%
\begin{equation*}
\xymatrix{ S\otimes S^{\vee} \ar@/^{0.75pc}/[dr]^-{\mu_{S,X}\cdot \ev_{S}^{\tau}} \ar[d]|{D_{S,X^{\vee}}\otimes D_{S^{\vee},X}} & & S^{\vee}\otimes S \ar@/^{0.75pc}/[dr]^-{\mu_{S,X}\cdot \ev_{S}} \ar[d]|{D_{S^{\vee},X}\otimes D_{S,X^{\vee}}} & \\ X^{\vee}\otimes Y^{\vee\vee}\otimes X\otimes Y^{\vee} \ar[r]_-{\lambda_{\left[S\right],\left[X\right]}\cdot \ev_{13,24}^{\phi,\phi}} & \mathbb{I} & X\otimes Y^{\vee}\otimes X^{\vee}\otimes Y^{\vee\vee} \ar[r]_-{\lambda_{\left[S\right],\left[X\right]}\cdot \ev_{13,24}^{\tau,\tau}} & \mathbb{I}\text{.}}
\end{equation*}%
Similarly if $\left( Cas\right) _{\mu _{X,S}}$ is satisfied and $\left(
Com\right) _{\lambda _{\left[ X\right] ,\left[ S\right] }}$ is commutative,
we get the analogue commutative diagram where $\left( S,X\right) $ is
replaced by $\left( X,S\right) $.
\end{proposition}

\begin{proof}
It is clear that the two diagrams are equivalently commutative, so that
suffices to prove the commutativity of the first diagram. It follows from Proposition \ref{S1 Casimir P1} $\left( 3\right)$  that we have%
\begin{equation*}
\lambda _{\left[ S\right] ,\left[ X\right] }\cdot \ev_{13,24}^{\tau ,\tau
}\circ \left( D_{\iota _{S,X}}\otimes D_{\iota _{S^{\vee },X^{\vee
}}}\right) =\lambda _{\left[ S\right] ,\left[ X\right] }\cdot
\ev_{13,24}^{\tau ,\tau }\circ \left( \varphi _{\left[ S\right] \otimes \left[
X\right] }\otimes 1_{\left[ Y^{\vee }\right] }\right) \circ \left( 1_{\left[
S\right] }\otimes C_{X}\otimes C_{Y^{\vee }}\right) \text{.}
\end{equation*}%
In order to compute the right hand side, consider the following diagram:%
\begin{equation*}
\xymatrix{ & \ar@{}[dr]|(0.57){(\tau)} & \left[ S \right]\ar@/_{1pc}/[dll]_{1_{\left[S\right]}\otimes C_{X}\otimes C_{Y^{\vee}}}\ar@/_{0.5pc}/[dl]|{C_{X}\otimes1_{\left[S\right]}\otimes C_{Y^{\vee}}}\ar[d]|{1_{\left[S\right]}\otimes C_{X}\otimes C_{Y^{\vee}}}\ar[r]^{i_{S}\otimes1_{S^\vee}} \ar@/^{3pc}/[dd]^{\mu_{S,X}\cdot 1_{\left[S\right]}\otimes C_{S^{\vee}}} & S^{\vee\vee}\otimes S^{\vee} \ar@{}[dl]|{(\otimes)}\ar[dd]^{\mu_{S,X}\cdot1_{S^{\vee\vee}\otimes S^{\vee}}\otimes C_{S^{\vee}}}\ar[r]^{\tau_{S^{\vee\vee},S^{\vee}}} & S^{\vee}\otimes S^{\vee\vee}\ar@{}[dl]|{(\tau)} \ar[dd]|(0.65){\mu_{S,X}\cdot 1_{S^{\vee}\otimes S^{\vee\vee}}\otimes C_{S^{\vee}}}  \\
\left[S\right]\otimes\left[X\right]\otimes\left[Y^{\vee}\right]
\ar[d]|{\varphi_{\left[S\right]\otimes\left[X\right]}\otimes1_{\left[Y^{%
\vee}\right]}}
\ar[r]^{\tau_{\left[S\right],\left[X\right]}\otimes1_{\left[Y^{\vee}%
\right]}} & \left[X\right]\otimes\left[S\right]\otimes\left[Y^{\vee}\right]
\ar@{}[u]|{(\tau)} \ar@{}[dl]|{(B)}
\ar[d]|{\varphi_{\left[X\right],\left[S\right]}\otimes1_{\left[Y^{\vee}%
\right]}}
\ar[r]^{\tau_{\left[X\right],\left[S\right]}\otimes1_{\left[Y^{\vee}%
\right]}} & \left[S\right]\otimes\left[X\right]\otimes\left[Y^{\vee}\right]
\ar@{}[dl]|{(A)}
\ar[d]|{1_{\left[S\right]}\otimes\varphi_{\iota_{\left[X\right],\left[S%
\right]}}} & & \\ \left[Y\right]\otimes\left[Y^{\vee}\right]
\ar@/_{1pc}/[drr]_{\lambda_{\left[S\right],\left[X\right]}\cdot
\ev_{13,24}^{\tau,\tau}} \ar[r]^{\lambda_{\left[S\right],\left[X\right]}} &
\left[Y\right]\otimes\left[Y^{\vee}\right] \ar@{}[dl]|(0.3){(D)}
\ar@/_{0.5pc}/[dr]|{\ev_{13,24}^{\tau,\tau}} &
\left[S\right]\otimes\left[S^{\vee}\right] \ar@{}[dr]|(0.3){(C)}
\ar[d]|{\ev_{13,24}^{\tau,\tau}}
\ar[r]^(0.43){i_{S}\otimes1_{S^{\vee}\otimes\left[S^{\vee}\right]}} &
S^{\vee\vee}\otimes S^{\vee}\otimes\left[S^{\vee}\right]
\ar@{}[dr]|(0.33){(E)} \ar@/^{0.5pc}/[dl]|{\ev_{13,24}^{\phi,\tau}}
\ar[r]^{\tau_{S^{\vee\vee},S^{\vee}}\otimes1_{\left[S^{\vee}\right]}} &
S^{\vee}\otimes S^{\vee\vee}\otimes\left[S^{\vee}\right]
\ar@/^{1pc}/[dll]^{\ev_{14,23}^{\tau,\phi}} \\ & & \mathbb{I} & &
}
\end{equation*}%
Here $\left( A\right) $ is commutative by the adjoint property of Lemma \ref%
{S1 Casimir Ltensor2}, $\left( B\right) =\left( Com\right) _{\lambda _{\left[
S\right] ,\left[ X\right] }}\otimes 1_{\left[ Y^{\vee }\right] }$ is
commutative by assumption, the equality $\mu _{S,X}\cdot 1_{\left[ S\right]
}\otimes C_{S^{\vee }}=\left( 1_{\left[ S\right] }\otimes \varphi _{\iota _{%
\left[ X\right] ,\left[ S\right] }}\right) \circ \left( 1_{\left[ S\right]
}\otimes C_{X}\otimes C_{Y^{\vee }}\right) $ is assured by $\left(
Cas\right) _{\mu _{S,X}}$, $\left( C\right) $ is commutative by definition
of $i_{S}$, $\left( D\right) $ is clearly commutative and $\left( E\right) $
by definition of $\ev_{13,24}^{\phi ,\tau }$ and $\ev_{14,23}^{\tau ,\phi }$.
We deduce the first of the subsequent equalities, while that second follows
from $\ev_{14,23}^{\tau ,\phi }\circ \left( \tau _{S^{\vee \vee },S^{\vee
}}\otimes C_{S^{\vee }}\right) =\ev_{S^{\vee }}$ granted by Lemma \ref%
{S1 Casimir L Properties} (1) and the third by definition of $i_{S}$:%
\begin{eqnarray*}
\lambda _{\left[ S\right] ,\left[ X\right] }\cdot \ev_{13,24}^{\tau ,\tau
}\circ \left( D_{\iota _{S,X}}\otimes D_{\iota _{S^{\vee },X^{\vee
}}}\right) &=&\mu _{S,X}\cdot \ev_{14,23}^{\tau ,\phi }\circ \left( \tau
_{S^{\vee \vee },S^{\vee }}\otimes C_{S^{\vee }}\right) \circ \left(
i_{S}\otimes 1_{S^{\vee }}\right) \\
&=&\mu _{S,X}\cdot \ev_{S^{\vee }}\circ \left( i_{S}\otimes 1_{S^{\vee
}}\right) =\mu _{S,X}\cdot \ev_{S}^{\tau }\text{.}
\end{eqnarray*}

We end the proof of the proposition by rearranging the left hand side of
this equality by looking at the following diagram:%
\begin{equation*}
\xymatrix{ & X^{\vee}\otimes Y^{\vee\vee}\otimes X^{\vee\vee}\otimes Y^{\vee\vee\vee} \ar@/^{0.75pc}/[dr]^-{\ev_{13,24}^{\tau,\tau}} & \\ S\otimes S^{\vee} \ar[r]_(0.4){D_{\iota_{S,X}}\otimes D_{\iota_{S^{\vee},X^{\vee}}^{\ast}}} \ar@/^{0.75pc}/[ur]^-{D_{\iota_{S,X}}\otimes D_{\iota_{S^{\vee},X^{\vee}}}} & X^{\vee}\otimes Y^{\vee\vee}\otimes X\otimes Y^{\vee} \ar@{}[ul]|(0.45){(F)} \ar@{}[ur]|(0.45){(G)} \ar[r]_-{\ev_{13,24}^{\phi,\phi}} \ar[u]|{1_{X^{\vee}\otimes Y^{\vee\vee}}\otimes i_{X}\otimes i_{Y^{\vee}}} & \mathbb{I}\text{.}}
\end{equation*}%
Here $\left( F\right) $ is commutative by $\left( \text{\ref{S1 AIM D7}}%
\right) $, while $\left( G\right) $ is commutative by definition of $i_{W}$
for $W=X$ and $Y$. The claimed commutativity follows.

Since the roles of $S$ and $X$ are symmetric, we get the same commutative
diagram where $\left( S,X\right) $ is replaced by $\left( X,S\right) $ if $%
\left( Cas\right) _{\mu _{X,S}}$ is satisfied and $\left( Com\right)
_{\lambda _{\left[ X\right] ,\left[ S\right] }}$ is commutative.
\end{proof}

\bigskip

\begin{lemma}
\label{FDP L1}If $\left( Com\right) _{\lambda _{X^{\vee },S^{\vee }}}$ is
commutative, the following diagrams are commutative:
\begin{equation*}
\xymatrix{ X^{\vee}\otimes S^{\vee} \ar[r]^-{D_{X^{\vee},S}\otimes1_{S^{\vee}}} \ar[d]|{1_{X^{\vee}}\otimes D_{S^{\vee},X}} & S\otimes Y^{\vee}\otimes S^{\vee} \ar[d]|{\lambda_{X^{\vee},S^{\vee}}\cdot \ev_{13,Y^{\vee}}^{\tau}} & S^{\vee}\otimes X^{\vee} \ar[r]^-{1_{S^{\vee}}\otimes D_{X^{\vee},S}} \ar[d]|{D_{S^{\vee},X}\otimes1_{X^{\vee}}} & S^{\vee}\otimes S\otimes Y^{\vee} \ar[d]|{\lambda_{X^{\vee},S^{\vee}}\cdot \ev_{S}\otimes1_{Y^{\vee}}} \\ X^{\vee}\otimes X\otimes Y^{\vee} \ar[r]^-{\ev_{X}\otimes1_{Y^{\vee}}} & Y^{\vee}\text{,} & X\otimes Y^{\vee}\otimes X^{\vee} \ar[r]^-{\ev_{13,Y^{\vee}}^{\tau}} & Y^{\vee}\text{.}}
\end{equation*}
Similarly, if $\left( Com\right) _{\lambda _{S^{\vee },X^{\vee }}}$ is
commutative, we get the analogue commutative diagram where $\left(
S,X\right) $ is replaced by $\left( X,S\right) $.

If $\left( Com\right) _{\lambda _{X,S}}$ is commutative, the following
diagrams are commutative:\newline
\begin{equation*}
\xymatrix{ X\otimes S \ar[r]^-{D_{X,S^{\vee}}\otimes1_{S}} \ar[d]|{1_{X}\otimes D_{S,X^{\vee}}} & S^{\vee}\otimes Y^{\vee\vee}\otimes S \ar[d]|{\lambda_{X,S}\cdot \ev_{13,Y^{\vee\vee}}^{\phi}} & S\otimes X \ar[r]^-{1_{S}\otimes D_{X,S^{\vee}}} \ar[d]|{D_{S,X^{\vee}}\otimes1_{X}} & S\otimes S^{\vee}\otimes Y^{\vee\vee} \ar[d]|{\lambda_{X,S}\cdot \ev_{S}^{\tau}\otimes1_{Y^{\vee\vee}}} \\ X\otimes X^{\vee}\otimes Y^{\vee\vee} \ar[r]^-{\ev_{X}^{\tau}\otimes1_{Y^{\vee\vee}}} & Y^{\vee\vee}\text{,} & X^{\vee}\otimes Y^{\vee\vee}\otimes X \ar[r]^-{\ev_{13,Y^{\vee\vee}}^{\phi}} & Y^{\vee\vee}\text{.}}
\end{equation*}
Similarly, if $\left( Com\right) _{\lambda _{S,X}}$ is commutative, we get
the analogue commutative diagram where $\left( S,X\right) $ is replaced by $%
\left( X,S\right) $.
\end{lemma}

\begin{proof}
It is clear that the first two diagrams are equivalently commutative, so
that suffices to prove the commutativity of the first diagram. Consider the
following diagram:%
\begin{equation*}
\xymatrix{ S^{\vee}\otimes X^{\vee} \ar@{}[dr]|{(\tau)}
\ar[d]|{1_{S^{\vee}\otimes X^{\vee}}\otimes C_{Y}} & X^{\vee}\otimes
S^{\vee} \ar@{}[dr]|{(\tau)} \ar[l]_{\tau_{X^{\vee},S^{\vee}}}
\ar[d]|{1_{X^{\vee}\otimes S^{\vee}}\otimes C_{Y}}
\ar[r]^{\tau_{X^{\vee},S^{\vee}}} & S^{\vee}\otimes X^{\vee}
\ar@{}[dr]|{(\tau)} \ar[r]^{\tau_{S^{\vee},X^{\vee}}}
\ar[d]|{1_{S^{\vee}\otimes X^{\vee}}\otimes C_{Y}} & X^{\vee}\otimes
S^{\vee} \ar[d]|{1_{X^{\vee}\otimes S^{\vee}}\otimes C_{Y}} \\
S^{\vee}\otimes X^{\vee}\otimes Y\otimes Y^{\vee} \ar@{}[ddr]|{(A)}
\ar[d]|{1_{S^{\vee}}\otimes\varphi_{\iota_{X^{\vee},S^{\vee}}^{\ast}}%
\otimes1_{Y^{\vee}}} & X^{\vee}\otimes S^{\vee}\otimes Y\otimes Y^{\vee}
\ar@{}[dr]|{(B)} \ar[l]_{\tau_{X^{\vee},S^{\vee}}\otimes1_{Y\otimes
Y^{\vee}}} \ar[r]^{\tau_{X^{\vee},S^{\vee}}\otimes1_{Y\otimes Y^{\vee}}}
\ar[d]|{\varphi_{X^{\vee},S^{\vee}}\otimes1_{Y\otimes Y^{\vee}}} &
S^{\vee}\otimes X^{\vee}\otimes Y\otimes Y^{\vee}
\ar[r]^{\tau_{S^{\vee},X^{\vee}}\otimes1_{Y\otimes Y^{\vee}}}
\ar[d]|{\varphi_{S^{\vee},X^{\vee}}\otimes1_{Y\otimes Y^{\vee}}} &
X^{\vee}\otimes S^{\vee}\otimes Y\otimes Y^{\vee} \ar@{}[ddl]|{(A)}
\ar[d]|{1_{X^{\vee}}\otimes\varphi_{\iota_{S^{\vee},X^{\vee}}^{\ast}}%
\otimes1_{Y^{\vee}}} \\ S^{\vee}\otimes S\otimes Y^{\vee}
\ar[d]|{\ev_{S}\otimes1_{Y^{\vee}}} & Y^{\vee}\otimes Y\otimes Y^{\vee}
\ar@/^{0.5pc}/[dl]^{\ev_{Y}\otimes1_{Y^{\vee}}}
\ar[r]^{\lambda_{X^{\vee},S^{\vee}}} & Y^{\vee}\otimes Y\otimes Y^{\vee}
\ar@/_{0.5pc}/[dr]_{\ev_{Y}\otimes1_{Y^{\vee}}} & X^{\vee}\otimes X\otimes
Y^{\vee} \ar[d]|{\ev_{X}\otimes1_{Y^{\vee}}} \\ Y^{\vee}
\ar[rrr]^{\lambda_{X^{\vee},S^{\vee}}} & & & Y^{\vee} }
\end{equation*}%
The regions $\left( A\right) $ are commutative by the adjoint property $%
\left( \text{\ref{S1 AIM D1 reflexivity 1}}\right) $ of $\varphi _{\iota
_{X^{\vee },S^{\vee }}^{\ast }}$ and $\varphi _{\iota _{S^{\vee },X^{\vee
}}^{\ast }}$, while $\left( B\right) =\left( Com\right) _{\lambda _{X^{\vee
},S^{\vee }}}\otimes 1_{Y\otimes Y^{\vee }}$ by our assumption. Recalling
that $D_{\iota _{X^{\vee },S^{\vee }}^{\ast }}=\left( \varphi _{\iota
_{X^{\vee },S^{\vee }}^{\ast }}\otimes 1_{Y^{\vee }}\right) \circ \left(
1_{X^{\vee }}\otimes C_{Y}\right) $ and $D_{\iota _{S^{\vee },X^{\vee
}}^{\ast }}=\left( \varphi _{\iota _{S^{\vee },X^{\vee }}^{\ast }}\otimes
1_{Y^{\vee }}\right) \circ \left( 1_{S^{\vee }}\otimes C_{Y}\right) $\ by $%
\left( \text{\ref{S1 Casimir D DefD_f}}\right) $, we deduce%
\begin{eqnarray*}
\left( \ev_{X}\otimes 1_{Y^{\vee }}\right) \circ \left( 1_{X^{\vee }}\otimes
D_{\iota _{S^{\vee },X^{\vee }}^{\ast }}\right) &=&\left( \ev_{X}\otimes
1_{Y^{\vee }}\right) \circ \left( 1_{X^{\vee }}\otimes \varphi _{\iota
_{S^{\vee },X^{\vee }}^{\ast }}\otimes 1_{Y^{\vee }}\right) \circ \left(
1_{X^{\vee }\otimes S^{\vee }}\otimes C_{Y}\right) \circ \tau _{S^{\vee
},X^{\vee }}\circ \tau _{X^{\vee },S^{\vee }} \\
&=&\lambda _{X^{\vee },S^{\vee }}\cdot \left( \ev_{S}\otimes 1_{Y^{\vee
}}\right) \circ \left( 1_{S^{\vee }}\otimes \varphi _{\iota _{X^{\vee
},S^{\vee }}^{\ast }}\otimes 1_{Y^{\vee }}\right) \circ \left( 1_{S^{\vee
}\otimes X^{\vee }}\otimes C_{Y}\right) \circ \tau _{X^{\vee },S^{\vee }} \\
&=&\lambda _{X^{\vee },S^{\vee }}\cdot \left( \ev_{S}\otimes 1_{Y^{\vee
}}\right) \circ \left( 1_{S^{\vee }}\otimes D_{\iota _{X^{\vee },S^{\vee
}}^{\ast }}\right) \circ \tau _{X^{\vee },S^{\vee }} \\
&=&\lambda _{X^{\vee },S^{\vee }}\circ \left( \ev_{S}\otimes 1_{Y^{\vee
}}\right) \circ \tau _{S\otimes Y^{\vee },S^{\vee }}\circ \left( D_{\iota
_{X^{\vee },S^{\vee }}^{\ast }}\otimes 1_{S^{\vee }}\right) \\
&=&\lambda _{X^{\vee },S^{\vee }}\circ \ev_{13,Y^{\vee }}^{\tau }\circ \left(
D_{\iota _{X^{\vee },S^{\vee }}^{\ast }}\otimes 1_{S^{\vee }}\right) \text{.}
\end{eqnarray*}%
The commutativity of the other diagrams is proved in a similar way.
\end{proof}

\bigskip

\begin{remark}
\label{Hom Reflexive Obj L}If $Y$ is a reflexive object the canonical
morphism%
\begin{equation*}
Hom\left( X,Y\right) \rightarrow Hom\left( Y^{\vee }\otimes X,\mathbb{I}%
\right)
\end{equation*}%
mapping $f:X\rightarrow Y$ to $\ev_{Y,\mathbb{I}}\circ \left( 1_{Y^{\vee
}}\otimes f\right) :Y^{\vee }\otimes X\overset{1_{Y^{\vee }}\otimes f}{%
\rightarrow }Y^{\vee }\otimes Y\overset{\ev_{Y,\mathbb{I}}}{\rightarrow }%
\mathbb{I}$ is a bijection.
\end{remark}

\bigskip

We can now prove the main result of this section.

\begin{theorem}
\label{FDP T}Suppose that $\left( Cas\right) _{\mu _{S,X}}$ is satisfied and
that $\left( Com\right) _{\lambda _{\left[ S\right] ,\left[ X\right] }}$ is
commutative. Then we have%
\begin{eqnarray*}
&&\mu _{S,X}:S\overset{D_{S,X^{\vee }}}{\rightarrow }X^{\vee }\otimes
Y^{\vee \vee }\overset{D_{X^{\vee },S}\otimes 1_{Y^{\vee \vee }}}{%
\rightarrow }S\otimes Y^{\vee }\otimes Y^{\vee \vee }\overset{1_{S}\otimes
\ev_{Y^{\vee }}^{\tau }}{\rightarrow }S\overset{\lambda _{\left[ S\right] ,%
\left[ X\right] }\lambda _{X^{\vee },S^{\vee }}}{\rightarrow }S\text{,} \\
&&\text{if }\left( Com\right) _{\lambda _{X^{\vee },S^{\vee }}}\text{ is
commutative,} \\
&&\mu _{S,X}\lambda _{S^{\vee },X^{\vee }}:S\overset{D_{S,X^{\vee }}}{%
\rightarrow }X^{\vee }\otimes Y^{\vee \vee }\overset{D_{X^{\vee },S}\otimes
1_{Y^{\vee \vee }}}{\rightarrow }S\otimes Y^{\vee }\otimes Y^{\vee \vee }%
\overset{1_{S}\otimes \ev_{Y^{\vee }}^{\tau }}{\rightarrow }S\overset{\lambda
_{\left[ S\right] ,\left[ X\right] }}{\rightarrow }S\text{,} \\
&&\text{if }\left( Com\right) _{\lambda _{S^{\vee },X^{\vee }}}\text{ is
commutative.}
\end{eqnarray*}

Suppose that $\left( Cas\right) _{\mu _{X,S}}$ is satisfied and that $\left(
Com\right) _{\lambda _{\left[ X\right] ,\left[ S\right] }}$ is commutative.
Then we have%
\begin{eqnarray*}
&&\mu _{X,S}:X^{\vee }\overset{D_{X^{\vee },S}}{\rightarrow }S\otimes
Y^{\vee }\overset{D_{S,X^{\vee }}\otimes 1_{Y^{\vee }}}{\rightarrow }X^{\vee
}\otimes Y^{\vee \vee }\otimes Y^{\vee }\overset{\ev_{Y^{\vee }}}{\rightarrow 
}X^{\vee }\overset{\lambda _{\left[ X\right] ,\left[ S\right] }\lambda _{S,X}%
}{\rightarrow }X^{\vee }\text{,} \\
&&\text{if }\left( Com\right) _{\lambda _{S,X}}\text{ is commutative,} \\
&&\mu _{X,S}\lambda _{X,S}:X^{\vee }\overset{D_{X^{\vee },S}}{\rightarrow }%
S\otimes Y^{\vee }\overset{D_{S,X^{\vee }}\otimes 1_{Y^{\vee }}}{\rightarrow 
}X^{\vee }\otimes Y^{\vee \vee }\otimes Y^{\vee }\overset{\ev_{Y^{\vee }}}{%
\rightarrow }X^{\vee }\overset{\lambda _{\left[ X\right] ,\left[ S\right] }}{%
\rightarrow }X^{\vee }\text{,} \\
&&\text{if }\left( Com\right) _{\lambda _{X,S}}\text{ is commutative.}
\end{eqnarray*}

We have the similar statements exchanging the roles of $S$ and $X$ in the
assumptions and the claims.
\end{theorem}

\begin{proof}
Suppose that $\left( Cas\right) _{\mu _{S,X}}$ is satisfied, that $\left(
Com\right) _{\lambda _{\left[ S\right] ,\left[ X\right] }}$ is commutative
and consider the following diagram:%
\begin{equation*}
\xymatrix{ S^{\vee}\otimes S \ar[dd]|{\mu_{S,X}\cdot \ev_{S}}
\ar[dr]|{D_{S^{\vee},X}\otimes D_{S,X^{\vee}}}
\ar[r]^(0.45){1_{S^{\vee}}\otimes D_{S,X^{\vee}}} & S^{\vee}\otimes
X^{\vee}\otimes Y^{\vee\vee} \ar@{}[dl]|(0.25){(\otimes)}
\ar[r]^{1_{S^{\vee}}\otimes D_{X^{\vee},S}\otimes1_{Y^{\vee\vee}}}
\ar[d]^{D_{S^{\vee},X}\otimes1_{X^{\vee}\otimes Y^{\vee\vee}}} &
S^{\vee}\otimes S\otimes Y^{\vee}\otimes Y^{\vee\vee}
\ar[d]|{\lambda_{X^{\vee},S^{\vee}}\cdot \ev_{S}\otimes1_{Y^{\vee}\otimes
Y^{\vee\vee}}} \ar[dr]^{1_{S^{\vee}\otimes S}\otimes \ev_{Y^{\vee}}^{\tau}} &
\\ \ar@{}[urr]|(0.73){(B)} & X\otimes Y^{\vee}\otimes X^{\vee}\otimes
Y^{\vee\vee} \ar@{}[l]|(0.7){(A)} \ar@{}[dr]|(0.3){(C)}
\ar[dl]|{\lambda_{\left[S\right],\left[X\right]}\cdot
\ev_{13,24}^{\tau,\tau}}
\ar[r]^(0.6){\ev_{13,Y^{\vee}}^{\tau}\otimes1_{Y^{\vee\vee}}} &
Y^{\vee}\otimes Y^{\vee\vee}
\ar@/^{0.5pc}/[dll]|{\lambda_{\left[S\right],\left[X\right]}\cdot
\ev_{Y^{\vee}}^{\tau}} & S^{\vee}\otimes S \ar@{}[l]|{(\otimes)}
\ar@/^{1pc}/[dlll]^{\lambda_{\left[S\right],\left[X\right]}\lambda_{X^{%
\vee},S^{\vee}}\cdot \ev_{S}} \\ \mathbb{I} & & & }
\end{equation*}%
Here the region $\left( A\right) $ is commutative by Proposition \ref{FPD P1}%
, $\left( B\right) $ by Lemma \ref{FDP L1} when $\left( Com\right) _{\lambda
_{X^{\vee },S^{\vee }}}$ is commutative and $\left( C\right) $ by by
definition of $\ev_{13,24}^{\tau ,\tau }$, $\ev_{13,Y^{\vee }}^{\tau }$ and $%
\ev_{Y^{\vee }}^{\tau }$. We deduce, setting $a:=\lambda _{\left[ S\right] ,%
\left[ X\right] }\lambda _{X^{\vee },S^{\vee }}\cdot \left( 1_{S}\otimes
\ev_{Y^{\vee }}^{\tau }\right) \circ \left( D_{X^{\vee },S}\otimes 1_{Y^{\vee
\vee }}\right) \circ D_{S,X^{\vee }}$, the equality%
\begin{equation*}
\ev_{S}\circ \left( 1_{S^{\vee }}\otimes a\right) =\ev_{S}\circ \left(
1_{S^{\vee }}\otimes \mu _{S,X}\right) \text{.}
\end{equation*}%
Hence, by Remark \ref{Hom Reflexive Obj L}, we get $a=\mu _{S,X}$. The commutativity of the other diagrams is proved in a similar way.
\end{proof}

\bigskip

As an immediate consequence of Theorem \ref{FDP T} we get, in light of
Remark \ref{FPD R1}, the following result.

\begin{corollary}
\label{FDP C1}Suppose that $\left( Cas\right) _{\mu _{S,X}}$ and $\left(
Cas\right) _{\mu _{X,S}}$ are satisfied, that $\left( Com\right) _{\lambda
_{S,X}}$, $\left( Com\right) _{\lambda _{S^{\vee },X^{\vee }}}$, $\left(
Com\right) _{\lambda _{X,S}}$ and $\left( Com\right) _{\lambda _{X^{\vee
},S^{\vee }}}$ are commutative, that $\mu _{S,X}$, $\mu _{X,S}$, $\lambda
_{S,X}$, $\lambda _{S^{\vee },X^{\vee }}$, $\lambda _{X,S}$ and $\lambda
_{X^{\vee },S^{\vee }}$ are invertible and that $Y$ is an invertible object.
Then $D_{S,X^{\vee }}$, $D_{X^{\vee },S}$, $D_{S^{\vee },X}$, $D_{X,S^{\vee
}}$, $f_{S,X}$, $f_{X,S}$, $f_{S^{\vee },X^{\vee }}$ and $f_{X^{\vee
},S^{\vee }}$ are isomorphisms.
\end{corollary}

\bigskip

Another important result for us will be the following corollary of Theorem %
\ref{FDP T}. Define the following morphisms%
\[
\begin{array}{ll}
\varphi _{X^{\vee },S^{\vee }}^{13}: & X^{\vee }\otimes Y^{\vee \vee
}\otimes S^{\vee }\otimes Y^{\vee \vee }\overset{1_{X^{\vee }}\otimes \tau
_{Y^{\vee \vee },S^{\vee }}\otimes 1_{Y^{\vee \vee }}}{\rightarrow }X^{\vee
}\otimes S^{\vee }\otimes Y^{\vee \vee }\otimes Y^{\vee \vee } \\ 
& \overset{\varphi _{X^{\vee },S^{\vee }}\otimes 1_{Y^{\vee \vee }\otimes
Y^{\vee \vee }}}{\rightarrow }Y^{\vee }\otimes Y^{\vee \vee }\otimes Y^{\vee
\vee }\text{,} \\ 
\varphi _{X,S}^{13}: & X\otimes Y^{\vee }\otimes S\otimes Y^{\vee }\overset{%
1_{X}\otimes \tau _{Y^{\vee },S}\otimes 1_{Y^{\vee }}}{\rightarrow }X\otimes
S\otimes Y^{\vee }\otimes Y^{\vee }\overset{\varphi _{X,S}\otimes 1_{Y^{\vee
}\otimes Y^{\vee }}}{\rightarrow }Y\otimes Y^{\vee }\otimes Y^{\vee }\text{,}%
\end{array}%
\]
as well as%
\begin{eqnarray*}
&&\varphi _{X^{\vee },S^{\vee }}^{13\rightarrow Y^{\vee \vee }}:X^{\vee
}\otimes Y^{\vee \vee }\otimes S^{\vee }\otimes Y^{\vee \vee }\overset{\varphi _{X^{\vee },S^{\vee }}^{13}}{\rightarrow } 
Y^{\vee }\otimes Y^{\vee \vee }\otimes Y^{\vee \vee }\overset{\ev_{Y^{\vee
}}^{\tau }\otimes 1_{Y^{\vee \vee }}}{\rightarrow }Y^{\vee \vee }\text{,} \\
&&\varphi _{X,S}^{13\rightarrow Y^{\vee }}:X\otimes Y^{\vee }\otimes
S\otimes Y^{\vee }\overset{\varphi _{X,S}^{13}}{\rightarrow } Y\otimes Y^{\vee }\otimes Y^{\vee }\overset{%
\ev_{Y}^{\tau }\otimes 1_{Y^{\vee }}}{\rightarrow }Y^{\vee }\text{.}
\end{eqnarray*}

\begin{corollary}
\label{FDP C2}Suppose that $\left( Cas\right) _{\mu _{S,X}}$ is satisfied
and that $\left( Com\right) _{\lambda _{\left[ S\right] ,\left[ X\right] }}$%
, $\left( Com\right) _{\lambda _{X,S}}$ and $\left( Com\right) _{\lambda
_{X^{\vee },S^{\vee }}}$ are commutative. Then, setting $\mu :=\mu _{S,X}$
and $\lambda :=\lambda _{\left[ S\right] ,\left[ X\right] }\lambda _{X^{\vee
},S^{\vee }}\lambda _{X,S}$, the following diagrams are commutative:%
\begin{equation*}
\xymatrix{ S\otimes X \ar[r]^-{\varphi_{S,X}} \ar[d]|{D_{S,X^{\vee}}\otimes
D_{X,S^{\vee}}} & Y \ar[d]|{\mu\cdot i_{Y}} & S^{\vee}\otimes X^{\vee}
\ar[r]^-{\varphi_{S^{\vee},X^{\vee}}} \ar[d]|{D_{S^{\vee},X}\otimes
D_{X^{\vee},S}} & Y^{\vee} \ar[d]|{\mu} \\ X^{\vee}\otimes
Y^{\vee\vee}\otimes S^{\vee}\otimes Y^{\vee\vee}
\ar[r]^-{\lambda\cdot\varphi _{X^{\vee },S^{\vee }}^{13\rightarrow Y^{\vee\vee }}} & Y^{\vee\vee} &
X\otimes Y^{\vee}\otimes S\otimes Y^{\vee}
\ar[r]^-{\lambda\cdot\varphi _{X,S}^{13\rightarrow Y^{\vee }}} & Y^{\vee}}
\end{equation*}
\end{corollary}

\begin{proof}
Suppose that $\left( Cas\right) _{\mu _{S,X}}$ is satisfied and that $\left(
Com\right) _{\lambda _{\left[ S\right] ,\left[ X\right] }}$, $\left(
Com\right) _{\lambda _{X,S}}$ and $\left( Com\right) _{\lambda _{X^{\vee
},S^{\vee }}}$ are commutative. Consider the following diagram:%
\begin{equation*}
\xymatrix{ S\otimes X \ar@{}[rd]|(0.4){(A)}
\ar[r]^(0.4){D_{S,X^{\vee}}\otimes1_{X}} \ar[d]|{1_{S}\otimes
D_{X,S^{\vee}}} & X^{\vee}\otimes Y^{\vee\vee}\otimes X
\ar[r]^(0.6){\mu_{S,X}\cdot \ev_{13,Y^{\vee\vee}}^{\phi}} & Y^{\vee\vee} & \\
S\otimes S^{\vee}\otimes Y^{\vee\vee} \ar@{}[drr]|{(C)}
\ar[urr]|{\mu_{S,X}\lambda_{X,S}\cdot \ev_{S}^{\tau}\otimes1_{Y^{\vee\vee}}}
\ar[rr]^{\mu_{S,X}} \ar[d]|{D_{S,X^{\vee}}\otimes1_{S^{\vee}\otimes
Y^{\vee\vee}}} & & S\otimes S^{\vee}\otimes Y^{\vee\vee}
\ar@{}[ul]|(0.4){(B)} \ar[u]|{\lambda_{X,S}\cdot
\ev_{S}^{\tau}\otimes1_{Y^{\vee\vee}}} & \\ X^{\vee}\otimes
Y^{\vee\vee}\otimes S^{\vee}\otimes Y^{\vee\vee} \ar@{}[drr]|{(\tau)}
\ar[rr]^{D_{X^{\vee},S}\otimes1_{Y^{\vee\vee}\otimes S^{\vee}\otimes
Y^{\vee\vee}}}
\ar[d]|{1_{X^{\vee}}\otimes\tau_{Y^{\vee\vee},S^{\vee}}\otimes1_{Y^{\vee%
\vee}}} & & S\otimes Y^{\vee}\otimes Y^{\vee\vee}\otimes S^{\vee}\otimes
Y^{\vee\vee} \ar@{}[ur]|(0.3){(D)}
\ar[u]|{\lambda_{\left[S\right],\left[X\right]}\lambda_{X^{\vee},S^{\vee}}%
\cdot1_{S}\otimes \ev_{Y^{\vee}}^{\tau}\otimes1_{S^{\vee}\otimes
Y^{\vee\vee}}} \ar[rd]^{\ev_{14}^{\tau}} \ar[d]|{1_{S\otimes
Y^{\vee}}\otimes\tau_{Y^{\vee\vee},S^{\vee}}\otimes1_{Y^{\vee\vee}}} & \\
X^{\vee}\otimes S^{\vee}\otimes Y^{\vee\vee}\otimes Y^{\vee\vee}
\ar[rr]^{D_{X^{\vee},S}\otimes1_{S^{\vee}\otimes Y^{\vee\vee}\otimes
Y^{\vee\vee}}} & & S\otimes Y^{\vee}\otimes S^{\vee}\otimes
Y^{\vee\vee}\otimes Y^{\vee\vee} \ar@{}[ur]|(0.3){(D)}
\ar[r]^(0.57){\ev_{13}^{\tau}} & Y^{\vee}\otimes Y^{\vee\vee}\otimes
Y^{\vee\vee}
\ar@/_{1pc}/[uuul]|{\lambda_{\left[S\right],\left[X\right]}\lambda_{X^{%
\vee},S^{\vee}}\lambda_{X,S}\cdot
\ev_{Y^{\vee}}^{\tau}\otimes1_{Y^{\vee\vee}}}}
\end{equation*}%
The region $\left( A\right) $ is commutative by Lemma \ref{FDP L1}, the
commutativity of $\left( B\right) $ is clear, $\left( C\right) $ is
commutative by Theorem \ref{FDP T} and $\left( D\right) $ by definition of
the evaluation maps (we have written $\ev_{14}^{\tau }:=\ev_{14,Y^{\vee
}\otimes Y^{\vee \vee }\otimes Y^{\vee \vee }}^{\tau }$ for shortness and
similarly for $\ev_{13}^{\tau }$). Recalling that we have, by definition, $%
D_{S,X^{\vee }}=D_{\iota _{S,X}}$ and $D_{X^{\vee },S}=D_{\iota _{X^{\vee
},S^{\vee }}^{\ast }}$, it follows from Lemma \ref{S1 AIM L3} we have $%
i_{Y}\circ \varphi _{S,X}=\ev_{13,Y^{\vee \vee }}^{\phi }\circ \left(
D_{\iota _{S,X}}\otimes 1_{X}\right) $ and $\varphi _{X^{\vee },S^{\vee
}}\otimes 1_{Y^{\vee \vee }\otimes Y^{\vee \vee }}=\ev_{13,Y^{\vee }\otimes
Y^{\vee \vee }\otimes Y^{\vee \vee }}^{\tau }\circ \left( D_{X^{\vee
},S}\otimes 1_{S^{\vee }\otimes Y^{\vee \vee }\otimes Y^{\vee \vee }}\right) 
$. Hence the commutative diagram gives the claimed equality:%
\begin{eqnarray*}
&&\mu _{S,X}\cdot i_{Y}\circ \varphi _{S,X}=\lambda _{\left[ S\right] ,\left[
X\right] }\lambda _{X^{\vee },S^{\vee }}\lambda _{X,S}\cdot \left(
\ev_{Y^{\vee }}^{\tau }\otimes 1_{Y^{\vee \vee }}\right) \circ \left( \varphi
_{X^{\vee },S^{\vee }}\otimes 1_{Y^{\vee \vee }\otimes Y^{\vee \vee }}\right)
\\
&&\text{ }\circ \left( 1_{X^{\vee }}\otimes \tau _{Y^{\vee \vee },S^{\vee
}}\otimes 1_{Y^{\vee \vee }}\right) \circ \left( D_{S,X^{\vee }}\otimes
D_{X,S^{\vee }}\right) =\lambda _{\left[ S\right] ,\left[ X\right] }\lambda
_{X^{\vee },S^{\vee }}\lambda _{X,S}\cdot \varphi _{X^{\vee },S^{\vee
}}^{13\rightarrow Y^{\vee \vee }}\circ \left( D_{S,X^{\vee }}\otimes D_{X,S^{\vee }}\right) \text{.}
\end{eqnarray*}%
The commutativity of the other diagram is proved in a similar way.
\end{proof}

\bigskip

\begin{corollary}
\label{FDP C3}Suppose that $\left( Cas\right) _{\mu _{S,X}}$ is satisfied,
that $\left( Com\right) _{\lambda _{\left[ S\right] ,\left[ X\right] }}$, $%
\left( Com\right) _{\lambda _{X,S}}$ and $\left( Com\right) _{\lambda
_{X^{\vee },S^{\vee }}}$ are commutative and that $Y$ is invertible of rank $%
r_{Y}$ (so that $r_{Y}\in \left\{ \pm 1\right\} $). For every morphism $%
g:A\rightarrow \hom \left( Y^{\vee },B\right) $ and $h:C\rightarrow \hom
\left( Y,D\right) $ the following diagrams are commutative, where $\mu $ and 
$\lambda $ are as in Corollary \ref{FDP C2}:%
\begin{equation*}
\xymatrix{ A\otimes S\otimes X \ar[d]|{1_{A}\otimes D_{S,X^{\vee}}\otimes D_{X,S^{\vee}}} \ar[r]^-{D_{g}\otimes\varphi_{S,X}} & B\otimes Y^{\vee\vee}\otimes Y \ar[dd]|{\mu\cdot1_{B\otimes Y^{\vee\vee}}\otimes i_{Y}} & C\otimes S^{\vee}\otimes X^{\vee} \ar[d]|{1_{C}\otimes D_{S^{\vee},X}\otimes D_{X^{\vee},S}} \ar[r]^-{D_{h}\otimes\varphi_{S^{\vee},X^{\vee}}} & D\otimes Y^{\vee}\otimes Y^{\vee} \ar[dd]|{\mu} \\ A\otimes X^{\vee}\otimes Y^{\vee\vee}\otimes S^{\vee}\otimes Y^{\vee\vee} \ar[d]|{1_{A}\otimes\varphi_{X^{\vee},S^{\vee}}^{13}} & & C\otimes X\otimes Y^{\vee}\otimes S\otimes Y^{\vee} \ar[d]|{1_{C}\otimes\varphi_{X,S}^{13}} & \\ A\otimes Y^{\vee}\otimes Y^{\vee\vee}\otimes Y^{\vee\vee} \ar[r]^-{\lambda r_{Y}\cdot\varphi_{g}\otimes1_{Y^{\vee\vee}\otimes Y^{\vee\vee}}} & B\otimes Y^{\vee\vee}\otimes Y^{\vee\vee}, & C\otimes Y\otimes Y^{\vee}\otimes Y^{\vee} \ar[r]^-{\lambda r_{Y}\cdot\varphi_{h}\otimes1_{Y^{\vee}\otimes Y^{\vee}}} & D\otimes Y^{\vee}\otimes Y^{\vee}. }
\end{equation*}
\end{corollary}

\begin{proof}
Consider the following diagram:%
\begin{equation*}
\xymatrix{ A\otimes S\otimes X \ar[r]^-{1_{A}\otimes\varphi_{S,X}} \ar[d]_{1_{A}\otimes D_{S,X^{\vee}}\otimes D_{X,S^{\vee}}} & A\otimes Y \ar[rr]^-{D_{g}\otimes1_{Y}} \ar[dd]|{\mu\cdot1_{A}\otimes i_{Y}} \ar[dr]|-{1_{A}\otimes C_{Y^{\vee}}\otimes1_{Y}} & & B\otimes Y^{\vee\vee}\otimes Y \ar[d]^{\mu\cdot1_{B\otimes Y^{\vee\vee}}\otimes i_{Y}} \\ A\otimes X^{\vee}\otimes Y^{\vee\vee}\otimes S^{\vee}\otimes Y^{\vee\vee} \ar@{}[r]|(0.67){(A)} \ar[d]_{1_{A}\otimes\varphi_{X^{\vee},S^{\vee}}^{13}} & & A\otimes Y^{\vee}\otimes Y^{\vee\vee}\otimes Y \ar@{}[u]|{(C)} \ar@{}[r]|{(\otimes)} \ar[ur]|-{\varphi_{g}\otimes1_{Y^{\vee\vee}\otimes Y}} \ar[dr]|-{\mu\cdot1_{A\otimes Y^{\vee}\otimes Y^{\vee\vee}}\otimes i_{Y}} & B\otimes Y^{\vee\vee}\otimes Y^{\vee\vee} \\  A\otimes Y^{\vee}\otimes Y^{\vee\vee}\otimes Y^{\vee\vee} \ar[r]^-{\lambda\cdot1_{A}\otimes \ev_{Y^{\vee}}^{\tau}\otimes1_{Y^{\vee\vee}}} & A\otimes Y^{\vee\vee} \ar@{}[ur]|(0.6){(B)} \ar[rr]^-{r_{Y}\cdot1_{A}\otimes\left(\ev_{Y^{\vee}}^{\tau}\right)^{-1}\otimes1_{Y^{\vee\vee}}} & & A\otimes Y^{\vee}\otimes Y^{\vee\vee}\otimes Y^{\vee\vee}. \ar[u]_{\varphi_{g}\otimes1_{Y^{\vee\vee}\otimes Y^{\vee\vee}}} }
\end{equation*}%
The region $\left( A\right) $ is commutative by Corollary \ref{FDP C2}. The
region $\left( B\right) $ because $r_{Y}=r_{Y^{\vee }}:=\ev_{Y^{\vee }}^{\tau
}\circ C_{Y^{\vee }}$, implying that $C_{Y^{\vee }}=r_{Y}\cdot \left(
\ev_{Y^{\vee }}^{\tau }\right) ^{-1}$ and the functoriality of $\otimes $.
Finally $\left( C\right) $ is commutative by $\left( \text{\ref{S1 Casimir D
DefD_f}}\right) $. The commutativity of the first diagram follows and the
commutativity of the second diagram is proved in the same way.
\end{proof}

\section{Application to $\Delta $-graded algebras in $\mathcal{C}$}

If $\Delta $ is a commutative integral\footnote{%
Integrality means that we may left (and right) simplify.} semigroup, a $%
\Delta $-graded algebra in $\mathcal{C}$\ is a family $A=\left(
A_{i},\varphi _{i,j}^{A}\right) _{i,j\in \Delta }$ where $\varphi
_{i,j}=\varphi _{i,j}^{A}:A_{i}\otimes A_{j}\rightarrow A_{i+j}$ are
morphisms making the following diagrams commutative:%
\begin{equation}
\xymatrix{ A_{i}\otimes A_{j}\otimes A_{k} \ar[r]^-{1_{A_{i}}\otimes\varphi_{j,k}} \ar[d]_{\varphi_{i,j}\otimes1_{A_{k}}} & A_{i}\otimes A_{j+k} \ar[d]^{\varphi_{i,j+k}} \\ A_{i+j}\otimes A_{k} \ar[r]^-{\varphi_{i+j,k}} & A_{i+j+k}\text{.} }
\label{S1 Algebras D1}
\end{equation}%
There is an obvious notion of morphisms of $\Delta $-graded algebras and of
direct sum decomposition, given component-wise. If $i,j\in \Delta $, we
write $j\geq i$ to mean that $\exists \left( j-i\right) \in \Delta $ (unique
because $\Delta $ is integral)\ such that $\left( j-i\right) +i=j$.

If $j\geq i$, we have $\varphi _{i,j-i}=\varphi _{f_{i,j-i}}:A_{i}\otimes
A_{j-i}\rightarrow A_{j}$, where $f_{i,j-i}=f_{i,j-i}^{A}:A_{i}\rightarrow
\hom \left( A_{j-i},A_{j}\right) $, so that we may consider the associated
internal multiplication morphism%
\begin{equation*}
\iota _{i,j}:=\iota _{f_{i,j-i}}:A_{i}\rightarrow \hom \left( A_{j}^{\vee
},A_{j-i}^{\vee }\right) \text{ corresponding to }\varphi _{\iota
_{i,j}}:A_{i}\otimes A_{j}^{\vee }\rightarrow A_{j-i}^{\vee }\text{.}
\end{equation*}

\bigskip

Suppose that, for every $i\in \Delta $, we have given a biproduct
decomposition $A_{i}=A_{i}^{+}\oplus A_{i}^{-}$ and write $i_{W}^{\pm }$, $%
p_{W}^{\pm }$ and $e_{W}^{\pm }$ as usual for $W=A_{i}$ or $A_{i}^{\vee }$
and we also set $i_{i}^{\pm }=i_{W}^{\pm }$, $p_{i}^{\pm }=p_{W}^{\pm }$ and 
$e_{i}^{\pm }=e_{W}^{\pm }$ when $W=A_{i}$ and $p_{i,\vee }^{\pm
}=p_{W}^{\pm }$ when $W=A_{i}^{\vee }$. We make a choice $\varepsilon
:\Delta \rightarrow \left\{ \pm \right\} $ of factors for every $i\in \Delta 
$ that we may assume, without loss of generality, to be given by the
constant function $\varepsilon _{i}=+$.

Next we consider%
\begin{eqnarray*}
&&f_{i,j}^{+}:A_{i}^{+}\overset{i_{i}^{+}}{\rightarrow }A_{i}\overset{f_{i,j}%
}{\rightarrow }\hom \left( A_{j},A_{i+j}\right) \overset{p_{\hom \left(
A_{j}^{+},A_{i+j}^{+}\right) }}{\rightarrow }\hom \left(
A_{j}^{+},A_{i+j}^{+}\right) \text{,} \\
&&\varphi _{i,j}^{+}:A_{i}^{+}\otimes A_{j}^{+}\overset{i_{i}^{+}\otimes
i_{j}^{+}}{\rightarrow }A_{i}\otimes A_{j}\overset{\varphi _{i,j}}{%
\rightarrow }A_{i+j}\overset{p_{i+j}^{+}}{\rightarrow }A_{i+j}^{+}
\end{eqnarray*}%
and, for $j\geq i$,%
\begin{eqnarray*}
&&\iota _{i,j}^{+}:A_{i}^{+}\overset{i_{i}^{+}}{\rightarrow }A_{i}\overset{%
\iota _{i,j}}{\rightarrow }\hom \left( A_{j}^{\vee },A_{j-i}^{\vee }\right) 
\overset{p_{\hom \left( A_{j}^{+\vee },A_{j-i}^{+\vee }\right) }}{%
\rightarrow }\hom \left( A_{j}^{+\vee },A_{j-i}^{+\vee }\right) \text{ and}
\\
&&\varphi _{\iota _{i,j}}^{+}:A_{i}^{+}\otimes A_{j}^{+}\overset{%
i_{i}^{+}\otimes i_{j}^{+}}{\rightarrow }A_{i}\otimes A_{j}^{\vee }\overset{%
\varphi _{\iota _{i,j}}}{\rightarrow }A_{j-i}^{\vee }\overset{p_{j-i,\vee
}^{+}}{\rightarrow }A_{j-i}^{+\vee }\text{.}
\end{eqnarray*}

The following result is a restatement of Lemma \ref{S1 AIM L1}\ and\
Proposition \ref{S1 AIM P1}.
\begin{proposition}
\label{S1 Algebras P1}The following diagrams are commutative.

\begin{enumerate}
\item[$\left( 1\right) $] When $j\geq i$%
\begin{equation*}
\xymatrix@C=100pt{ A_{i}\otimes A_{j-i}\otimes A_{j}^{\vee} \ar[r]^-{(1_{A_{j-i}}\otimes\varphi_{\iota_{i,j}})\circ(\tau_{A_{i},A_{j-i}}\otimes1_{A_{j}^{\vee}})} \ar[d]_{\varphi_{i,j-i}\otimes1_{A_{j}^{\vee}}} & A_{j-i}\otimes A_{j-i}^{\vee} \ar[d]^{\ev_{A_{j-i}}^{\tau}} \\ A_{j}\otimes A_{j}^{\vee} \ar[r]^-{\ev_{A_{j}}^{\tau}} & \mathbb{I}\text{.} }
\end{equation*}

\item[$\left( 2\right) $] When $k\geq i$ and $k-i\geq j$,%
\begin{equation*}
\xymatrix{ A_{j}\otimes A_{i}\otimes A_{k}^{\vee} \ar[r]^-{1_{A_{j}}\otimes\varphi_{\iota_{i,k}}} \ar[d]|{\varphi_{i,j}^{\tau}\otimes1_{A_{k}^{\vee}}} & A_{j}\otimes A_{k-i}^{\vee} \ar[d]|{\varphi_{\iota_{j,k-i}}} & A_{i}\otimes A_{j} \ar[r]^-{\iota_{i,k}\otimes\iota_{j,k-i}} \ar[d]|{\varphi_{i,j}} & \hom\left(A_{k}^{\vee},A_{k-i}^{\vee}\right)\otimes\hom\left(A_{k-i}^{\vee},A_{k-i-j}^{\vee}\right) \ar[d]|{c_{A_{k}^{\vee},A_{k-i}^{\vee},A_{k-i-j}^{\vee}}^{\tau}} \\ A_{i+j}\otimes A_{k}^{\vee} \ar[r]^-{\varphi_{\iota_{i+j,k}}} & A_{k-i-j}^{\vee} & A_{i+j} \ar[r]^-{\iota_{i+j,k}} & \hom\left(A_{k}^{\vee},A_{k-i-j}^{\vee}\right)\text{,}}
\end{equation*}%
where%
\begin{equation*}
\varphi _{i,j}^{\tau }:A_{j}\otimes A_{i}\overset{\tau _{A_{j},A_{i}}}{%
\rightarrow }A_{i}\otimes A_{j}\overset{\varphi _{i,j}}{\rightarrow }A_{i+j}%
\text{.}
\end{equation*}
\end{enumerate}

Suppose that $e_{i+j}^{+}\circ \varphi _{i,j}\circ \left( e_{i}^{+}\otimes
1_{A_{j}}\right) =e_{i+j}^{+}\circ \varphi _{i,j}$ and $e_{i+j}^{+}\circ
\varphi _{i,j}\circ \left( 1_{A_{i}}\otimes e_{j}^{+}\right)
=e_{i+j}^{+}\circ \varphi _{i,j}$. Then $A^{+}:=\left( A_{i}^{+},\varphi
_{i,j}^{+}\right) _{i,j\in \Delta }$ is a $\Delta $-graded algebra in $%
\mathcal{C}$, we have $\varphi _{f_{i,j}^{+}}=\varphi _{i,j}^{+}$, the
internal multiplication morphisms $\iota _{f_{i,j}^{+}}$\ associated to this
algebra structure are explicitly given by $\iota _{f_{i,j}^{+}}=\iota
_{i,j}^{+}$ and $\varphi _{\iota _{f_{i,j}^{+}}}=\varphi _{\iota
_{i,j}^{+}}=\varphi _{\iota _{i,j}}^{+}$ and we have the analogue of the
above commutative diagrams with the $+$ sign inserted. Furthermore, we have
a biproduct decomposition $A=A^{+}\oplus A^{-}$ as $\Delta $-graded
algebras,\ where $A^{-}:=\left( A_{i}^{-},\varphi _{i,j}^{-}\right) _{i,j\in
\Delta }$ satisfies the analogue results.
\end{proposition}

\bigskip

We will assume from now on that we have given $\Delta $-graded algebras $%
A=\left( A_{i},\varphi _{i,j}^{A}\right) _{i,j\in \Delta }$ and $A^{\vee
}=\left( A_{i}^{\vee },\varphi _{i,j}^{A^{\vee }}\right) _{i,j\in \Delta }$
and $\mathcal{C}$ is rigid. Then we define a $\Delta
\times \Delta $-graded family $A\otimes A^{\vee }:=\left( A_{j}^{i},\varphi
_{j,l}^{i,k}\right) _{\left( i,j\right) ,\left( k,l\right) \in \Delta \times
\Delta }$ by the rule $A_{j}^{i}:=A_{i}\otimes A_{j}^{\vee }$ and%
\begin{equation*}
\varphi _{j,l}^{i,k}:=\varphi _{i,k}^{A}\otimes _{\epsilon }\varphi
_{j,l}^{A^{\vee }}:A_{j}^{i}\otimes A_{l}^{k}\rightarrow A_{j+l}^{i+k}\text{%
, associated to }f_{j,l}^{i,k}:=f_{i,k}^{A}\otimes _{\epsilon
}f_{j,l}^{A^{\vee }}:A_{j}^{i}\rightarrow \hom \left(
A_{l}^{k},A_{j+l}^{i+k}\right)
\end{equation*}%
by Lemma \ref{S1 Casimir Ltensor1}. It is easily checked that $A\otimes
A^{\vee }$ is indeed a $\Delta \times \Delta $-graded algebra.

Next we define, when $l\geq i$ and $k\geq j$,%
\begin{eqnarray*}
&&\delta _{i,l}^{A}:=\varphi _{\iota _{i,l}}=\varphi _{\iota
_{f_{i,l-i}^{A}}}:A_{i}\otimes A_{l}^{\vee }\rightarrow A_{l-i}^{\vee }\text{%
,} \\
&&\delta _{j,k}^{A^{\vee }}:=\varphi _{\iota _{f_{j,k-j}^{A^{\vee }}}^{\ast
}}:A_{j}^{\vee }\otimes A_{k}\rightarrow A_{k-j}
\end{eqnarray*}
and%
\begin{equation*}
\delta _{j,l}^{i,k}:=\delta _{j,k}^{A^{\vee }}\otimes _{\epsilon }^{\tau
}\delta _{i,l}^{A}:A_{j}^{i}\otimes A_{l}^{k}\rightarrow A_{l-i}^{k-j}\text{%
, associated to }\iota _{j,l}^{i,k}:=\iota _{f_{j,k-j}^{A^{\vee }}}^{\ast
}\otimes _{\epsilon }^{\tau }\iota _{f_{i,l-i}^{A}}:A_{j}^{i}\rightarrow
\hom \left( A_{l}^{k},A_{l-i}^{k-j}\right)
\end{equation*}%
by Lemma \ref{S1 Casimir Ltensor1}.

Applying Proposition \ref{S1 Algebras P1} to $A\otimes A^{\vee }$ we find,
thanks to Corollary \ref{S1 Casimir R} and $\left( \text{\ref{S1 AIM D1 reflexivity 2}}\right) $, 
the following result.

\begin{corollary}
\label{S1 Algebras C1}The following diagrams are commutative.

\begin{enumerate}
\item[$\left( 1\right) $] When $l\geq i$ and $k\geq j$,%
\begin{equation*}
\xymatrix@C=120pt{ A_{j}^{i}\otimes A_{k-j}^{l-i}\otimes A_{l}^{k} \ar[r]^-{(1_{A_{k-j}^{l-i}}\otimes\delta_{j,l}^{i,k})\circ(\tau_{A_{j}^{i},A_{k-j}^{l-i}}\otimes1_{A_{l}^{k}})} \ar[d]_{\varphi_{j,k-j}^{i,l-i}\otimes1_{A_{l}^{k}}} & A_{k-j}^{l-i}\otimes A_{l-i}^{k-j} \ar[d]^{\ev_{14,23}^{\tau,\phi}} \\ A_{k}^{l}\otimes A_{l}^{k} \ar[r]^-{\ev_{14,23}^{\tau,\phi}} & \mathbb{I}\text{.} }
\end{equation*}

\item[$\left( 2\right) $] When $n\geq i$, $m\geq j$, $n-i\geq k$ and $%
m-j\geq l$,%
\begin{equation*}
\xymatrix@C=33pt{ A_{l}^{k}\otimes A_{j}^{i}\otimes A_{n}^{m} \ar[r]^-{1_{A_{l}^{k}}\otimes\delta_{j,n}^{i,m}} \ar[d]|{\varphi_{j,l}^{i,k,\tau}\otimes1_{A_{n}^{m}}} & A_{l}^{k}\otimes A_{n-i}^{m-j} \ar[d]|{\delta_{l,n-i}^{k,m-j}} & A_{j}^{i}\otimes A_{l}^{k} \ar[r]^-{\iota_{j,n}^{i,m}\otimes\iota_{l,n-i}^{k,m-j}} \ar[d]|{\varphi_{j,l}^{i,k}} & \hom\left(A_{n}^{m},A_{n-i}^{m-j}\right)\otimes\hom\left(A_{n-i}^{m-j},A_{n-i-k}^{m-j-l}\right) \ar[d]|{c_{A_{n}^{m},A_{n-i}^{m-j},A_{n-i-k}^{m-j-l}}^{\tau}} \\ A_{j+l}^{i+k}\otimes A_{n}^{m} \ar[r]^-{\delta_{j+l,n}^{i+k,m}} & A_{n-i-k}^{m-j-l} & A_{j+l}^{i+k} \ar[r]^-{\iota_{j+l,n}^{i+k,m}} & \hom\left(A_{n}^{m},A_{n-i-k}^{m-j-l}\right)\text{,}}
\end{equation*}%
where%
\begin{equation*}
\varphi _{j,l}^{i,k,\tau }:A_{l}^{k}\otimes A_{j}^{i}\overset{\tau
_{A_{l}^{k},A_{j}^{i}}}{\rightarrow }A_{j}^{i}\otimes A_{l}^{k}\overset{%
\varphi _{j,l}^{i,k}}{\rightarrow }A_{j+l}^{i+k}\text{.}
\end{equation*}
\end{enumerate}
\end{corollary}

We define, when $g\geq i$,%
\begin{equation*}
D^{i,g}:=D_{\iota _{f_{i,g-i}^{A}}}:A_{i}\rightarrow A_{g-i}^{\vee }\otimes
A_{g}^{\vee \vee }\text{ and }D_{i,g}:=D_{\iota _{f_{i,g-i}^{A^{\vee
}}}^{\ast }}:A_{i}^{\vee }\rightarrow A_{g-i}\otimes A_{g}^{\vee }\text{.}
\end{equation*}%
We leave to the reader to restate the result of the previous section in this context.

\subsection{Symmetric and alternating algebras}

Suppose now that we have given a rigid and pseudo-abelian object $V\in 
\mathcal{C}$ and that $\mathcal{C}$ is $\mathbb{Q}$-linear. The
associativity constraint of $\mathcal{C}$ implies that we may define an $%
\mathbb{N}$-graded tensor algebra $\otimes ^{\cdot }V$ by the rule $\varphi
_{i,j}^{V,t}:=1_{\otimes ^{i+j}V}:\left( \otimes ^{i}V\right) \otimes \left(
\otimes ^{j}V\right) \rightarrow \otimes ^{i+j}V$.

The permutation group $S_{i}$ acts on $\otimes ^{i}V$ and, for a character $%
\chi $ of $S_{i}$ and a subset $S\subset S_{i}$, we define%
\begin{equation}
e_{S}^{\chi }:=\frac{1}{\#S}\tsum\nolimits_{\chi \in S_{i}}\chi ^{-1}\left(
\sigma \right) \sigma \text{.}  \label{Alternating algebras F idemp}
\end{equation}%
There are exactly two characters of $S_{i}$, namely $\chi =\varepsilon $
(the sign character) and $\chi =1$ (the trivial character), which are
distinct when $i\geq 2$. We let $\wedge ^{i}V$ (resp. $\vee ^{i}V$) the
biproduct factor of $\otimes ^{i}V$\ which corresponds to the idempotent $%
e_{a,V}^{i}:=e_{S_{i}}^{\varepsilon }$ (resp. $e_{s,V}^{i}:=e_{S_{i}}^{1}$%
),which exists because we assume that $V$ is pseudo-abelian. We write $%
i_{V,a}^{i}:\wedge ^{i}V\rightarrow \otimes ^{i}V$ (resp. $i_{V,s}^{i}:\vee
^{i}V\rightarrow \otimes ^{i}V$) and $p_{V,a}^{i}:\otimes ^{i}V\rightarrow
\wedge ^{i}V$ (resp. $p_{V,s}^{i}:\otimes ^{i}V\rightarrow \vee ^{i}V$) for
the associated injective and surjective morphisms. Next we claim that,
setting%
\begin{eqnarray*}
\varphi _{i,j}^{V,a} &:&\wedge ^{i}V\otimes \wedge ^{j}V\overset{%
i_{V,a}^{i}\otimes i_{V,a}^{j}}{\rightarrow }\left( \otimes ^{i}V\right)
\otimes \left( \otimes ^{j}V\right) \overset{\varphi _{i,j}^{V,t}}{%
\rightarrow }\otimes ^{i+j}V\overset{p_{V,a}^{i+j}}{\rightarrow }\wedge
^{i+j}V\text{,} \\
\varphi _{i,j}^{V,s} &:&\vee ^{i}V\otimes \vee ^{j}V\overset{%
i_{V,s}^{i}\otimes i_{V,s}^{j}}{\rightarrow }\left( \otimes ^{i}V\right)
\otimes \left( \otimes ^{j}V\right) \overset{\varphi _{i,j}^{V,t}}{%
\rightarrow }\otimes ^{i+j}V\overset{p_{V,s}^{i+j}}{\rightarrow }\vee ^{i+j}V
\end{eqnarray*}%
give rise to $\mathbb{N}$-graded algebras, called respectively the
alternating and the symmetric algebras on $V$. Since $\varphi
_{i,j}^{V,t}=1_{\otimes ^{i+j}V}$ we may check, according to Proposition \ref%
{S1 Algebras P1}, that we have $e_{S_{i+j}}^{\chi }\circ \left(
e_{S_{i}}^{\chi }\otimes 1_{\otimes ^{j}V}\right) =e_{S_{i+j}}^{\chi }$ and $%
e_{S_{i+j}}^{\chi }\circ \left( 1_{\otimes ^{i}V}\otimes e_{S_{j}}^{\chi
}\right) =e_{S_{i+j}}^{\chi }$ for $\chi =\varepsilon $ or $\chi =1$. But we
have that the action of $e_{S_{i}}^{\chi }\otimes 1_{\otimes ^{j}V}\in
End\left( \otimes ^{i+j}V\right) $ (resp. $1_{\otimes ^{i}V}\otimes
e_{S_{j}}^{\chi }\in End\left( \otimes ^{i+j}V\right) $)\ is obtained by
identifying $S_{i}\simeq S_{\left\{ 1,...,i\right\} }\subset S_{i+j}$ (resp. 
$S_{j}\simeq S_{\left\{ i+1,...,i+j\right\} }\subset S_{i+j}$) and it is
given by $e_{S_{\left\{ 1,...,i\right\} }}^{\chi }\in \mathbb{Q}\left[
S_{i+j}\right] $ (resp. $e_{S_{\left\{ i+1,...,i+j\right\} }}^{\chi }\in 
\mathbb{Q}\left[ S_{i+j}\right] $). Hence the claimed relation follows from
the identity $e_{S_{i+j}}^{\chi }e_{S_{\left\{ 1,...,i\right\} }}^{\chi
}=e_{S_{i+j}}^{\chi }$ (resp. $e_{S_{i+j}}^{\chi }e_{S_{\left\{
i+1,...,i+j\right\} }}^{\chi }=e_{S_{i+j}}^{\chi }$) taking place in $%
\mathbb{Q}\left[ S_{i+j}\right] $.

We note that it follows from the definitions that $\varphi _{i,j}^{V,a}$ and 
$\varphi _{i,j}^{V,s}$ are uniquely characterized by the property of making
the following diagrams commutative:%
\begin{equation}
\xymatrix{ \left(\otimes^{i}V\right)\otimes\left(\otimes^{j}V\right) \ar[r]^-{1_{\otimes^{i+j}V}} \ar[d]|{p_{V,a}^{i}\otimes p_{V,a}^{j}} & \otimes^{i+j}V \ar[d]|{p_{V,a}^{i+j}} & \left(\otimes^{i}V\right)\otimes\left(\otimes^{j}V\right) \ar[r]^-{1_{\otimes^{i+j}V}} \ar[d]|{p_{V,s}^{i}\otimes p_{V,s}^{j}} & \otimes^{i+j}V \ar[d]|{p_{V,s}^{i+j}} \\ \wedge^{i}V\otimes\wedge^{j}V \ar[r]^-{\varphi_{i,j}^{V,a}} & \wedge^{i+j}V\text{,} & \vee^{i}V\otimes\vee^{j}V \ar[r]^-{\varphi_{i,j}^{V,s}} & \vee^{i+j}V\text{.}}
\label{Induced morphisms D1}
\end{equation}

\bigskip 

Next we remark that we may use rigidity of $V$ to canonically identify $%
\epsilon _{V}^{i}:\left( \otimes ^{i}V^{\vee },\ev_{V}^{i}\right) \rightarrow
\left( \left( \otimes ^{i}V\right) ^{\vee },\ev_{\otimes ^{i}V}\right) $: in other
words $\left( \otimes ^{i}V^{\vee },\ev_{V}^{i}\right) $ is a dual pair for $%
\otimes ^{i}V$. Next, we define%
\begin{eqnarray*}
\ev_{V,a}^{i} &:&\wedge ^{i}V^{\vee }\otimes \wedge ^{i}V\overset{i_{V^{\vee
},a}^{i}\otimes i_{V,a}^{i}}{\rightarrow }\left( \otimes ^{i}V^{\vee
}\right) \otimes \left( \otimes ^{i}V\right) \overset{\ev_{V}^{i}}{%
\rightarrow }\mathbb{I}\text{,} \\
\ev_{V,s}^{i} &:&\vee ^{i}V^{\vee }\otimes \vee ^{i}V\overset{i_{V^{\vee
},s}^{i}\otimes i_{V,s}^{i}}{\rightarrow }\left( \otimes ^{i}V^{\vee
}\right) \otimes \left( \otimes ^{i}V\right) \overset{\ev_{V}^{i}}{%
\rightarrow }\mathbb{I}
\end{eqnarray*}%
and we claim that $\left( \wedge ^{i}V^{\vee },\ev_{V,a}^{i}\right) $ is a
dual pair for $\wedge ^{i}V$ and $\left( \vee ^{i}V^{\vee
},\ev_{V,a}^{i}\right) $ is a dual pair for $\vee ^{i}V$. Indeed we have $%
\ev_{V}^{i}\circ \left( \sigma \otimes \sigma \right) =\ev_{V}^{i}$ for every $%
\sigma \in S_{i}$, because the canonical morphism $\otimes ^{i}\mathbb{I}%
\rightarrow \mathbb{I}$ appearing in the definition of $\ev_{V}^{i}$ is $S_{i}
$-invariant. Equivalently, $\ev_{V}^{i}\circ \left( \sigma ^{-1}\otimes
1_{\otimes ^{i}V}\right) =\ev_{V}^{i}\circ \left( 1_{\otimes ^{i}V^{\vee
}}\otimes \sigma \right) $ proving that $\sigma ^{-1}=\sigma ^{\vee }$ from
which it follows that $\left( e_{S_{i}}^{\chi }\right) ^{\vee
}=e_{S_{i}}^{\chi }$. Then our claim follows from $\left( \text{\ref{S1 FDec
evaluation}}\right) $ (with $Y=Y^{\pm }=\mathbb{I}$).

\section{A Poincar\'{e} duality isomorphism for the alternating algebras}

In this section we suppose that $\mathcal{C}$ is rigid, $\mathbb{Q}$-linear and pseudo-abelian.
We consider an object $V\in \mathcal{C}$ and we apply the results on $\Delta $-graded algebras with $A=\left( \wedge
^{\cdot }V,\varphi _{i,j}^{V,a}\right) $ and $A^{\vee }=\left( \wedge
^{\cdot }V^{\vee },\varphi _{i,j}^{V^{\vee },a}\right) $. We will use the
shorter notations $i_{V}^{p}:=i_{V,a}^{p}$, $p_{V}^{p}:=p_{V,a}^{p}$, $%
e_{V}^{p}:=e_{V,a}^{p}$, $\varphi _{i,j}=\varphi _{i,j}^{V}:=\varphi
_{i,j}^{V,a}$ and $\varphi _{i,j}^{V^{\vee }}:=\varphi _{i,j}^{V^{\vee },a}$%
. In order to make explicit the internal multiplication morphisms we define,
for every $j\geq i$,%
\begin{equation*}
\varphi _{\iota _{i,j}^{V,t}}:=\ev_{V}^{i,\tau }\otimes 1_{\otimes
^{j-i}V^{\vee }}:\left( \otimes ^{i}V\right) \otimes \left( \otimes
^{j}V^{\vee }\right) \rightarrow \otimes ^{j-i}V^{\vee }\text{.}
\end{equation*}%
It is readily checked that $\varphi _{\iota _{i,j}^{V,t}}$ satisfies the
characterizing property $\left( \text{\ref{S1 AIM D1}}\right) $ of
Proposition \ref{S1 AIM R1} with $\varphi _{f}=\varphi
_{i,j-i}^{V,t}=1_{\otimes ^{j}V}$. Since the alternating algebra is obtained
from the tensor algebra as in Proposition \ref{S1 Algebras P1}, it follows
that the internal multiplication $\iota _{i,j}=\iota _{i,j}^{V,a}:=\iota
_{i,j}^{A}$ of the alternating algebra is explicitly given by%
\begin{equation*}
\varphi _{\iota _{i,j}}:\wedge ^{i}V\otimes \wedge ^{j}V^{\vee }\overset{%
i_{V}^{i}\otimes i_{V^{\vee }}^{j}}{\rightarrow }\left( \otimes ^{i}V\right)
\otimes \left( \otimes ^{j}V^{\vee }\right) \overset{\ev_{V}^{i,\tau }\otimes
1_{\otimes ^{j-i}V^{\vee }}}{\rightarrow }\otimes ^{j-i}V^{\vee }\overset{%
p_{V^{\vee }}^{i-j}}{\rightarrow }\wedge ^{j-i}V^{\vee }\text{.}
\end{equation*}%
In order to make this morphism completely explicit, we note that $%
S_{i}\times S_{j}$ acts on $\left( \otimes ^{i}V\right) \otimes \left(
\otimes ^{j}V^{\vee }\right) $ and we may identify $S_{j-i}\simeq S_{\left\{
i+1,...,j\right\} }\subset S_{j}$ acting on $\left( \otimes ^{i}V\right)
\otimes \left( \otimes ^{j}V^{\vee }\right) $. With this identification,%
\begin{equation}
\varphi _{\iota _{i,j}^{V,t}}\circ \sigma =\sigma \circ \varphi _{\iota
_{i,j}^{V,t}}\text{ for every }\sigma \in S_{j-i}.
\label{Alternating algebras F equiv}
\end{equation}%
Let%
\begin{equation*}
\mathcal{P}^{i\leq j}:=\mathcal{P}^{\left\{ i+1,...,j\right\} ,\left\{
1,...,j\right\} }=\left\{ p=\left( p_{1},...,p_{i}\right) \in \left\{
1,...,j\right\} ^{i}:p_{k}\neq p_{l}\text{ for every }k\neq l\right\} \text{,%
}
\end{equation*}%
and, for every $p=\left( p_{1},...,p_{i}\right) \in \mathcal{P}^{i\leq j}$,
write $\delta _{p}^{i\leq j}\in S_{j}$ for a fixed permutation such that $%
\delta _{p}^{i\leq j}\left( p_{k}\right) =k$ for $k\in \left\{
1,...,i\right\} $. Then $\left\{ \delta _{p}^{i\leq j}:p\in \mathcal{P}%
^{i\leq j}\right\} $ is a set of coset representatives for $S_{\left\{
i+1,...,j\right\} }\backslash S_{j}$ and, hence, $R:=\left\{ \delta \delta
_{p}^{i\leq j}:p\in \mathcal{P}^{i\leq j},\delta \in S_{i}\right\} $ is a
set of coset representatives for $S_{\left\{ i+1,...,j\right\} }\backslash
S_{i}\times S_{j}$. We have, setting $e^{i\leq j}:=\frac{\left( j-i\right) !%
}{j!}\tsum\nolimits_{p\in \mathcal{P}^{i\leq j}}\varepsilon ^{-1}\left(
\delta _{p}^{i\leq j}\right) \delta _{p}^{i\leq j}$,%
\begin{equation}
\frac{\left( j-i\right) !}{j!i!}\tsum\nolimits_{\delta \in S_{i},p\in 
\mathcal{P}^{i\leq j}}\varepsilon ^{-1}\left( \delta \delta _{p}^{i\leq
j}\right) \delta \delta _{p}^{i\leq j}=e_{S_{i}}^{\varepsilon }\cdot
e^{i\leq j}\text{,}  \label{Alternating algebras F repr}
\end{equation}%
which acts on $\left( \otimes ^{i}V\right) \otimes \left( \otimes
^{j}V^{\vee }\right) $ via $e_{V}^{i}\otimes e^{i\leq j}$. In particular,
when $i=1$, we have $\mathcal{P}^{1\leq j}=\left\{ 1,...,j\right\} $ and $%
\delta _{p}^{1\leq j}$ is any fixed permutation such that $\delta
_{p}^{1\leq j}\left( p\right) =1$; we may take, for example, $\delta
_{p}^{1\leq j}=\left( 1,...,p\right) =:c_{p}$ in this case and then $%
\varepsilon ^{-1}\left( \delta _{p}^{1\leq j}\right) =\left( -1\right) ^{p-1}
$. We can now prove the following result.

\begin{lemma}
\label{Alternating algebras L1}Setting%
\begin{equation*}
\widetilde{\varphi }_{\iota _{i,j}}:=\left( \ev_{V}^{i,\tau }\otimes
1_{\otimes ^{j-i}V^{\vee }}\right) \circ \left( e_{V}^{i}\otimes e^{i\leq
j}\right) =\frac{\left( j-i\right) !}{j!}\tsum\nolimits_{p\in \mathcal{P}%
^{i\leq j}}\varepsilon ^{-1}\left( \delta _{p}^{i\leq j}\right) \cdot \left(
\ev_{V}^{i,\tau }\otimes 1_{\otimes ^{j-i}V}\right) \circ \left(
e_{V}^{i}\otimes \delta _{p}^{i\leq j}\right) 
\end{equation*}%
we have that $\varphi _{\iota _{i,j}}$ is the unique morphism making the
following diagram commutative:%
\begin{equation*}
\xymatrix{ \left(\otimes^{i}V\right)\otimes\left(\otimes^{j}V^{\vee}\right) \ar[r]^-{\widetilde{\varphi}_{\iota_{i,j}}} \ar[d]_{p_{V}^{i}\otimes p_{V^{\vee}}^{j}} & \otimes^{j-i}V^{\vee} \ar[d]^{p_{V^{\vee}}^{j}} \\ \wedge^{i}V\otimes\wedge^{j}V^{\vee} \ar[r]^-{\varphi_{\iota_{i,j}}} & \wedge^{j-i}V^{\vee}\text{.} }
\end{equation*}%
In particular, when $i=1$, setting $\ev_{V,p}^{\tau }:=\left( \ev_{V}^{\tau
}\otimes 1_{\otimes ^{j-1}V^{\vee }}\right) \circ c_{p}$\footnote{%
We have, symbolically,%
\begin{equation*}
\ev_{0,p}\left( f\otimes x_{1}\otimes ...\otimes x_{p}\otimes ...\otimes
x_{j}\right) =\left\langle f,x_{p}\right\rangle x_{1}\otimes ...\otimes 
\widehat{x}_{p}\otimes ...\otimes x_{j}
\end{equation*}%
}, we may take%
\begin{equation*}
\widetilde{\varphi }_{\iota _{1,j}}=\frac{1}{j}\tsum\nolimits_{p=1}^{j}%
\left( -1\right) ^{p-1}\ev_{0,p}\text{.}
\end{equation*}
\end{lemma}

\begin{proof}
Suppose that we have given a subgroup $H\subset G$ of a finite group $G$,
that $G$ acts on $X$, $H$ acts on $Y$ and that we have given $f:X\rightarrow
Y$ which is $H$-equivariant. If $\chi $ is a character of $G$ and $%
R_{H\backslash G}$ is a set of coset representative for $H\backslash G$, we
may consider the elements $e_{G}^{\chi }$, $e_{H}^{\chi }$ and $%
e_{R_{H\backslash G}}^{\chi }$\ defined as in $\left( \text{\ref{Alternating
algebras F idemp}}\right) $ with $S=G$, $H$ or $R_{H\backslash G}$ and $S_{i}
$ replaced by a more general group $G$. We have that $e_{G}^{\chi }$ and $%
e_{H}^{\chi }$ are idempotents and we let $p_{X}^{\chi }:X\rightarrow
X^{\chi }$ and $p_{Y}^{\chi }:Y\rightarrow Y^{\chi }$ be the associated
surjective morphisms and $i_{X}^{\chi }$ the associated injective morphism.
Then it is a general fact that, setting $f_{R_{H\backslash G}}:=f\circ
e_{R_{H\backslash G}}^{\chi }$ and $f^{\chi }:=p_{Y}^{\chi }\circ f\circ
i_{X}^{\chi }$, the morphism $f^{\chi }$ is characterized as the unique
morphism such that $f^{\chi }\circ p_{X}^{\chi }=p_{Y}^{\chi }\circ
f_{R_{H\backslash G}}$.

We may apply this general remark with $X=\left( \otimes ^{i}V\right) \otimes
\left( \otimes ^{j}V^{\vee }\right) $, $Y=\otimes ^{j-i}V^{\vee }$, $%
f=\ev_{V}^{i,\tau }\otimes 1_{\otimes ^{j-i}V^{\vee }}$, $H=S_{\left\{
i+1,...,j\right\} }$ and $G=S_{i}\times S_{j}$; since $f$ is $H$-equivariant
by $\left( \text{\ref{Alternating algebras F equiv}}\right) $ and $R$ is a
set of coset representatives for $S_{\left\{ i+1,...,j\right\} }\backslash
S_{i}\times S_{j}$, we deduce that $f^{\varepsilon }=\varphi _{\iota _{i,j}}$
is the unique morphism such that $f^{\chi }\circ p_{X}^{\chi }=p_{Y}^{\chi
}\circ f_{R_{H\backslash G}}$, where thanks to $\left( \text{\ref%
{Alternating algebras F repr}}\right) $ and the equality $\#R=\frac{\left(
j-i\right) !}{j!i!}$,%
\begin{equation*}
f_{R_{H\backslash G}}=\left( \ev_{V}^{i,\tau }\otimes 1_{\otimes
^{j-i}V^{\vee }}\right) \circ \left( e_{V}^{i}\otimes e^{i\leq j}\right) 
\text{.}
\end{equation*}%
\end{proof}

\bigskip

Beside the properties encoded in Proposition \ref{S1 Algebras P1}, the
internal multiplication morphisms $\iota _{1,j}$ has another key property.
In symbols it says that the normalized family $\iota _{j}:=j\cdot \iota
_{1,j}$ is an antiderivation, i.e.%
\begin{equation*}
\iota _{j+l}\left( x\right) \left( \omega _{j}\wedge \omega _{l}\right)
=\iota _{j+l}\left( x\right) \left( \omega _{j}\right) +\left( -1\right)
^{j}\iota _{j+l}\left( x\right) \left( \omega _{l}\right) \text{ for }x\in V%
\text{, }\omega _{j}\in \wedge ^{j}V^{\vee }\text{ and }\omega _{l}\in
\wedge ^{l}V^{\vee }\text{.}
\end{equation*}%
This is the content of the following proposition whose proof, based on Lemma %
\ref{Alternating algebras L1}, we leave to the reader.

\begin{proposition}
\label{Alternating algebras P1}The following diagram is commutative, when $%
j,l\geq 1$:%
\begin{equation*}
\xymatrix@C=200pt{ V\otimes\wedge^{j}V^{\vee}\otimes\wedge^{l}V^{\vee} \ar[r]^-{\left(j\cdot\varphi_{\iota_{1,j}}\otimes1_{\wedge^{l}V^{\vee}},\left(-1\right)^{j}l\cdot\left(1_{\wedge^{j}V^{\vee}}\otimes\varphi_{\iota_{1,l}}\right)\circ\left(\tau_{V,\wedge^{j}V^{\vee}}\otimes1_{\wedge^{l}V^{\vee}}\right)\right)} \ar[d]_{1_{V}\otimes\varphi_{j,l}} & \wedge^{j-1}V^{\vee}\otimes\wedge^{l}V^{\vee}\oplus\wedge^{j}V^{\vee}\otimes\wedge^{l-1}V^{\vee} \ar[d]^{\varphi_{j-1,l}\oplus\varphi_{j,l-1}} \\ V\otimes\wedge^{j+l}V^{\vee} \ar[r]^-{\left(j+l\right)\varphi_{\iota_{1,j+l}}} & \wedge^{j+l-1}V^{\vee}\text{.} }
\end{equation*}
\end{proposition}

Working with the dual algebras one easily sees that, setting%
\begin{equation*}
\varphi _{\iota _{i,j}^{V^{\vee },t,\ast }}:=\ev_{V}^{i}\otimes
1_{\otimes ^{j-i}V}:\left( \otimes ^{i}V^{\vee }\right) \otimes \left(
\otimes ^{j}V\right) \rightarrow \otimes ^{j-i}V\text{,}
\end{equation*}%
the morphism $\varphi _{\iota _{i,j}^{V^{\vee },t,\ast }}$ satisfies the
property $\left( \text{\ref{S1 AIM D1 reflexivity 1}}\right) $ with $\varphi
_{g}=\varphi _{i,j-i}^{V^{\vee },t}=1_{\otimes ^{j}V}$, which is of course
characterizing. It follows that $\iota _{i,j}^{\ast }=\iota _{i,j}^{V^{\vee
},a,\ast }:=\iota _{i,j}^{A^{\vee },\ast }$ is obtained in the analogous way
as $\iota _{i,j}$ was obtained and the analogous of Proposition \ref%
{Alternating algebras P1} is true.

Exactly as we did with more general $\Delta $-graded algebras, we can now
define, when $l\geq i$ and $k\geq j$, $\delta _{i,l}^{A}:=\varphi _{\iota
_{i,l}}:\wedge ^{i}V\otimes \wedge ^{l}V^{\vee }\rightarrow \wedge
^{l-i}V^{\vee }$, $\delta _{j,k}^{A^{\vee }}:=\varphi _{\iota _{j,k}^{\ast
}}:\wedge ^{j}V^{\vee }\otimes \wedge ^{k}V\rightarrow \wedge ^{k-j}V^{\vee
} $ and%
\begin{equation*}
\delta _{j,l}^{i,k}:=\delta _{j,k}^{A^{\vee }}\otimes _{\epsilon }^{\tau
}\delta _{i,l}^{A}:\wedge _{j}^{i}V\otimes \wedge _{l}^{k}V\rightarrow
\wedge _{l-i}^{k-j}V\text{, associated to }\iota _{j,l}^{i,k}:=\iota
_{j,k}^{\ast }\otimes _{\epsilon }^{\tau }\iota _{i,l}:\wedge
_{j}^{i}V\rightarrow \hom \left( \wedge _{l}^{k}V,\wedge
_{l-i}^{k-j}V\right) \text{,}
\end{equation*}%
where $\wedge _{q}^{p}V:=\wedge ^{p}V\otimes \wedge ^{q}V^{\vee }$. Beside
the properties encoded in Corollary \ref{S1 Algebras C1}, the following
property is enjoyed by the families $\delta _{0,q}^{1,p}$ and $\delta
_{1,q}^{0,p}$: the proof is just an application of Proposition \ref%
{Alternating algebras P1} and its dual statement for the second diagram.

\begin{corollary}
\label{Alternating algebras C1}If $j,l\geq 1$ then the following diagram is commutative:
\begin{equation*}
\xymatrix@C=200pt{ \wedge_{0}^{1}V\otimes\wedge_{j}^{i}V\otimes\wedge_{l}^{k}V \ar[r]^-{\left(j\cdot\delta_{0,j}^{1,i}\otimes1_{\wedge_{l}^{k}V},\left(-1\right)^{j}l\cdot\left(1_{\wedge_{j}^{i}V}\otimes\delta_{0,l}^{1,k}\right)\circ\left(\tau_{\wedge_{0}^{1}V,\wedge_{j}^{i}V}\otimes1_{\wedge_{l}^{k}V}\right)\right)} \ar[d]_{1_{\wedge_{0}^{1}V}\otimes\varphi_{j,l}^{i,k}} & \wedge_{j-1}^{i}V\otimes\wedge_{l}^{k}V\oplus\wedge_{j}^{i}V\otimes\wedge_{l-1}^{k}V \ar[d]^{\varphi_{j-1,l}^{i,k}\oplus\varphi_{j,l-1}^{i,k}} \\ \wedge_{0}^{1}V\otimes\wedge_{j+l}^{i+k}V \ar[r]^-{\left(j+l\right)\cdot\delta_{0,j+l}^{1,i+k}} & \wedge_{j+l-1}^{i+k}V }
\end{equation*}

The following diagram is commutative, when $i,k\geq 1$:%
\begin{equation*}
\xymatrix@C=200pt{ \wedge_{1}^{0}V\otimes\wedge_{j}^{i}V\otimes\wedge_{l}^{k}V \ar[r]^-{\left(i\cdot\delta_{1,j}^{0,i}\otimes1_{\wedge_{l}^{k}V},\left(-1\right)^{i}k\cdot\left(1_{\wedge_{j}^{i}V}\otimes\delta_{1,l}^{0,k}\right)\circ\left(\tau_{\wedge_{1}^{0}V,\wedge_{j}^{i}V}\otimes1_{\wedge_{l}^{k}V}\right)\right)} \ar[d]_{1_{\wedge_{1}^{0}V}\otimes\varphi_{j,l}^{i,k}} & \wedge_{j}^{i-1}V\otimes\wedge_{l}^{k}V\oplus\wedge_{j}^{i}V\otimes\wedge_{l}^{k-1}V \ar[d]^{\varphi_{j,l}^{i-1,k}\oplus\varphi_{j,l}^{i,k-1}} \\ \wedge_{1}^{0}V\otimes\wedge_{j+l}^{i+k}V \ar[r]^-{\left(i+k\right)\cdot\delta_{1,j+l}^{0,i+k}} & \wedge_{j+l}^{i+k-1}V\text{.} }
\end{equation*}
\end{corollary}

\bigskip

The proof of the following lemma, which is postponed to the subsequent
subsection, is based on Corollaries \ref{S1 Algebras C1} and \ref%
{Alternating algebras C1}.

\begin{lemma}
\label{Alternating algebras L key}Let $r:=\mathrm{rank}\left( V\right) $ be
the rank of $V$, defined as the composition%
\begin{equation*}
r:\mathbb{I}\overset{C_{V}}{\rightarrow }V\otimes V^{\vee }\overset{%
\ev_{V}^{\tau }}{\rightarrow }\mathbb{I}\text{.}
\end{equation*}%
For every $g\geq i$ we have the equality%
\begin{equation*}
\binom{g}{i}^{-1}\binom{r+i-g}{i}\cdot C_{\wedge ^{g-i}V}=\delta
_{i,g}^{i,g}\circ \left( C_{\wedge ^{i}V}\otimes C_{\wedge ^{g}V}\right) 
\text{,}
\end{equation*}%
where, for every $k\in \mathbb{N}_{\geq 1}$,%
\begin{equation*}
\binom{T}{k}:=\frac{1}{k!}T\left( T-1\right) ...\left( T-k+1\right) \in 
\mathbb{Q}\left[ T\right] \text{ and }\binom{T}{0}=1\text{.}
\end{equation*}
\end{lemma}

\bigskip

We define, when $g\geq i$,%
\begin{equation*}
D^{i,g}:=D_{\iota _{i,g}}:\wedge ^{i}V\rightarrow \wedge ^{g-i}V^{\vee
}\otimes \wedge ^{g}V^{\vee \vee }\text{ and }D_{i,g}:=D_{\iota _{i,g}^{\ast
}}:\wedge ^{i}V^{\vee }\rightarrow \wedge ^{g-i}V\otimes \wedge ^{g}V^{\vee }%
\text{.}
\end{equation*}%
Thanks to Lemma \ref{Alternating algebras L key}, the commutative diagrams of Proposition \ref{FPD P1}, Lemma \ref{FDP L1}, Theorem \ref{FDP T} and, respectively, Corollary \ref{FDP C2}, translate into the following result.

\begin{theorem}
\label{Alternating algebras T}The following diagrams are commutative, for
every $g\geq i\geq 0$.

\begin{enumerate}
\item[$\left( 1\right) $] 
\begin{equation*}
\xymatrix@C=60pt{ \wedge^{i}V\otimes\wedge^{i}V^{\vee} \ar@/^{0.75pc}/[dr]|{\binom{g}{g-i}^{-1}\binom{r-i}{g-i}\ev_{V,a}^{i,\tau}} \ar[d]|{D^{i,g}\otimes D_{i,g}} \\ \wedge^{g-i}V^{\vee}\otimes\wedge^{g}V^{\vee\vee}\otimes\wedge^{g-i}V\otimes\wedge^{g}V^{\vee} \ar[r]^-{\ev_{13,24}^{\phi,\phi}} & \mathbb{I}\text{.} }
\end{equation*}

\item[$\left( 2\right) $] 
\begin{equation*}
\xymatrix{ \wedge^{i}V^{\vee}\otimes\wedge^{g-i}V^{\vee} \ar[r]^-{1_{\wedge^{i}V^{\vee}}\otimes D_{g-i,g}} \ar[d]|{D_{i,g}\otimes1_{\wedge^{g-i}V^{\vee}}} & \wedge^{i}V^{\vee}\otimes\wedge^{i}V\otimes\wedge^{g}V^{\vee} \ar[d]|{\left(-1\right)^{i\left(g-i\right)}\cdot \ev_{V,a}^{i}\otimes1_{\wedge^{g}V^{\vee}}} & \wedge^{i}V\otimes\wedge^{g-i}V \ar[r]^-{1_{\wedge^{i}V}\otimes D^{g-i,g}} \ar[d]|{D^{i,g}\otimes1_{\wedge^{g-i}V}} & \wedge^{i}V\otimes\wedge^{i}V^{\vee}\otimes\wedge^{g}V^{\vee\vee} \ar[d]|{\left(-1\right)^{i\left(g-i\right)}\cdot \ev_{V,a}^{i,\tau}\otimes1_{\wedge^{g}V^{\vee\vee}}} \\ \wedge^{g-i}V\otimes\wedge^{g}V^{\vee}\otimes\wedge^{g-i}V^{\vee} \ar[r]^-{\ev_{13,\wedge^{g}V^{\vee}}^{\tau}} & \wedge^{g}V^{\vee}\text{,} & \wedge^{g-i}V^{\vee}\otimes\wedge^{g}V^{\vee\vee}\otimes\wedge^{g-i}V \ar[r]^-{\ev_{13,\wedge^{g}V^{\vee\vee}}^{\phi}} & \wedge^{g}V^{\vee\vee}\text{.} }
\end{equation*}

\item[$\left( 3\right) $] 
\begin{equation*}
\xymatrix@C=40pt{ \wedge^{i}V \ar@/^{2pc}/[rrr]^-{\left(-1\right)^{i\left(g-i\right)}\binom{g}{g-i}^{-1}\binom{r-i}{g-i}} \ar[r]_-{D^{i,g}} & \wedge^{g-i}V^{\vee}\otimes\wedge^{g}V^{\vee\vee} \ar[r]_-{D_{g-i,g}\otimes1_{\wedge^{g}V^{\vee\vee}}} & \wedge^{i}V\otimes\wedge^{g}V^{\vee}\otimes\wedge^{g}V^{\vee\vee} \ar[r]_-{1_{\wedge^{i}V}\otimes \ev_{V^{\vee},a}^{g,\tau}} & \wedge^{i}V }
\end{equation*}
and
\begin{equation*}
\xymatrix@C=40pt{ \wedge^{g-i}V^{\vee} \ar@/^{2pc}/[rrr]^-{\left(-1\right)^{i\left(g-i\right)}\binom{g}{i}^{-1}\binom{r+i-g}{i}} \ar[r]_-{D_{g-i,g}} & \wedge^{i}V\otimes\wedge^{g}V^{\vee} \ar[r]_-{D^{i,g}\otimes1_{\wedge^{g}V^{\vee}}} & \wedge^{g-i}V^{\vee}\otimes\wedge^{g}V^{\vee\vee}\otimes\wedge^{g}V^{\vee} \ar[r]_-{1_{\wedge^{g-i}V^{\vee}}\otimes \ev_{V^{\vee},a}^{g}} & \wedge^{g-i}V^{\vee}\text{.} }
\end{equation*}

\item[$\left( 4\right) $] 
\begin{equation*}
\xymatrix{ \wedge^{i}V\otimes\wedge^{g-i}V \ar[r]^-{\varphi_{i,g-i}} \ar[d]|{D^{i,g}\otimes D^{g-i,g}} & \wedge^{g}V \ar[d]|{\binom{g}{g-i}^{-1}\binom{r-i}{g-i}\cdot i_{\wedge^{g}V}} & \wedge^{i}V^{\vee}\otimes\wedge^{g-i}V^{\vee} \ar[r]^-{\varphi_{i,g-i}} \ar[d]|{D_{i,g}\otimes D_{g-i,g}} & \wedge^{g}V^{\vee} \ar[d]|{\binom{g}{g-i}^{-1}\binom{r-i}{g-i}} \\ \wedge^{g-i}V^{\vee}\otimes\wedge^{g}V^{\vee\vee}\otimes\wedge^{i}V^{\vee}\otimes\wedge^{g}V^{\vee\vee} \ar[r]^-{\varphi_{g-i,i}^{13\rightarrow \wedge^{g}V^{\vee\vee}}} & \wedge^{g}V^{\vee\vee}\text{,} & \wedge^{g-i}V\otimes\wedge^{g}V^{\vee}\otimes\wedge^{i}V\otimes\wedge^{g}V^{\vee} \ar[r]^-{\varphi_{g-i,i}^{13\rightarrow \wedge^{g}V^{\vee}}} & \wedge^{g}V^{\vee}\text{.}}
\end{equation*}
\end{enumerate}
\end{theorem}

\begin{proof}
The commutative diagrams $\left( 1\right) $, $\left( 2\right) $, $\left(
3\right) $ and $\left( 4\right) $ are just Proposition \ref{FPD P1}, Lemma %
\ref{FDP L1}, Theorem \ref{FDP T} and, respectively, Corollary \ref{FDP C2}
with $\left( S,X,Y\right) =\left( \wedge ^{i}V,\wedge ^{g-i}V,\wedge
^{g}V\right) $, $\varphi _{S,X}=\varphi _{i,g-i}^{V}$, $\varphi
_{X,S}=\varphi _{g-i,i}^{V}$, $\varphi _{S^{\vee },X^{\vee }}=\varphi
_{i,g-i}^{V^{\vee }}$, $\varphi _{X^{\vee },S^{\vee }}=\varphi
_{g-i,i}^{V^{\vee }}$, $D_{S,X^{\vee }}=D^{i,g}$, $D_{X,S^{\vee }}=D^{g-i,g}$%
, $D_{S^{\vee },X}=D_{i,g}$, $D_{X^{\vee },S}=D_{g-i,g}$. Indeed we have in
this case $\mu _{S,X}=\mu _{g-i,g}=\binom{g}{g-i}^{-1}\binom{r-i}{g-i}$, $%
\mu _{X,S}=\mu _{i,g}=\binom{g}{i}^{-1}\binom{r+i-g}{i}$, $\lambda
_{S,X}=\lambda _{X,S}=\lambda _{S^{\vee },X^{\vee }}=\lambda _{X^{\vee
},S^{\vee }}=\left( -1\right) ^{i\left( g-i\right) }$ and $\lambda _{\left[ X%
\right] ,\left[ S\right] }=\lambda _{\left[ S\right] ,\left[ X\right] }=1$.
\end{proof}

\bigskip

We say that $V$ has \emph{alternating rank }$g\in \mathbb{N}_{\geq 1}$ if $%
\wedge ^{g}V$ is an invertible object and $\binom{r-i}{g-i}$ and $\binom{%
r+i-g}{i}$ are invertible for every $0\leq i\leq g$. For example, when $%
End\left( \mathbb{I}\right) $ is a field or $r\in \mathbb{Q}$, the second
condition means that $r$ is not a root of the polynomials $\binom{T-i}{g-i}%
\in \mathbb{Q}\left[ T\right] $ and $\binom{T+i-g}{i}\in \mathbb{Q}\left[ T%
\right] $ for every $0\leq i\leq g$, i.e. that $r\neq i,i+1,...,g-1$ and $%
r\neq g-i,g-i+1,...,g-1$ for every $1\leq i\leq g$.

We say that $V$ has \emph{strong alternating rank }$g\in \mathbb{N}_{\geq 1}$ if $\wedge
^{g}V $ is an invertible object and $r=g$ (hence $V$ has alternating rank $g$).
With these notations Corollary \ref{FDP C1} specializes to the following
result.

\begin{corollary}
\label{Alternating algebras CT}If $V$ has weakly geometric rank $g\in 
\mathbb{N}$ then, for every $0\leq i\leq g$, the morphisms $D^{i,g}$, $%
D_{g-i,g}$, $D^{g-i,g}$ and $D_{i,g}$ are isomorphisms and the
multiplication maps $\varphi _{i,g-i}^{V}$, $\varphi _{g-i,i}^{V}$, $\varphi
_{i,g-i}^{V^{\vee }}$ and $\varphi _{g-i,i}^{V^{\vee }}$ are perfect
pairings (meaning that the associate $\hom $ valued morphisms are
isomorphisms). Furthermore, when $V$ has geometric rank $g$, we have $\binom{%
r-i}{g-i}=\binom{r+i-g}{i}=1$ in the commutative diagrams of Theorem \ref%
{Alternating algebras T}.
\end{corollary}

We end this section with the following result.

\begin{proposition}
\label{Alternating algebras P2}The following diagrams are commutative, when $%
\wedge ^{g}V$ is invertible of rank $r_{\wedge ^{g}V}$ (hence $r_{\wedge
^{g}V}\in \left\{ \pm 1\right\} $):%
\[
\xymatrix{ \wedge^{i}V\otimes\wedge^{g-i}V\otimes V \ar[r]^-{\tau_{\wedge^{i}V\otimes\wedge^{g-i}V,V}} \ar[d]|{\left(1_{\wedge^{i}V}\otimes\varphi_{g-i,1},\left(1_{\wedge^{g-i}V}\otimes\varphi_{i,1}\right)\circ\left(\tau_{\wedge^{i}V,\wedge^{g-i}V}\otimes1_{V}\right)\right)} & V\otimes\wedge^{i}V\otimes\wedge^{g-i}V \ar[d]|{D^{1,g}\otimes\varphi_{i,g-i}}\\
\wedge^{i}V\otimes\wedge^{g-i+1}V\oplus\wedge^{g-i}V\otimes\wedge^{i+1}V \ar[d]|{D^{i,g}\otimes D^{g-i+1,g}\oplus D^{g-i,g}\otimes D^{i+1,g}} & \wedge^{g-1}V^{\vee}\otimes\wedge^{g}V^{\vee\vee}\otimes\wedge^{g}V \ar[d]|{r_{\wedge^{g}V}g\binom{g}{g-i}^{-1}\binom{r-i}{g-i}\cdot1_{\wedge^{g-1}V^{\vee}\otimes\wedge^{g}V^{\vee\vee}}\otimes i_{\wedge^{g}V}} \\
M \ar[r]_-{\left(-1\right)^{g-i}i\cdot\varphi_{g-i,i-1}^{13}\oplus\left(-1\right)^{i\left(g-i-1\right)}\left(g-i\right)\cdot\varphi_{i,g-i-1}^{13}} & \wedge^{g-1}V^{\vee}\otimes\wedge^{g}V^{\vee\vee}\otimes\wedge^{g}V^{\vee\vee}
}
\]%
where
\[
M = \wedge^{g-i}V^{\vee}\otimes\wedge^{g}V^{\vee\vee}\otimes\wedge^{i-1}V^{\vee}\otimes\wedge^{g}V^{\vee\vee}\oplus\wedge^{i}V^{\vee}\otimes\wedge^{g}V^{\vee\vee}\otimes\wedge^{g-i-1}V^{\vee}\otimes\wedge^{g}V^{\vee\vee}
\]
and%
\[
\xymatrix{ \wedge^{i}V^{\vee}\otimes\wedge^{g-i}V^{\vee}\otimes V^{\vee} \ar[r]^-{\tau_{\wedge^{i}V^{\vee}\otimes\wedge^{g-i}V^{\vee},V^{\vee}}} \ar[d]|{\left(1_{\wedge^{i}V^{\vee}}\otimes\varphi_{g-i,1},\left(1_{\wedge^{g-i}V^{\vee}}\otimes\varphi_{i,1}\right)\circ\left(\tau_{\wedge^{i}V^{\vee},\wedge^{g-i}V^{\vee}}\otimes1_{V^{\vee}}\right)\right)} & V^{\vee}\otimes\wedge^{i}V^{\vee}\otimes\wedge^{g-i}V^{\vee} \ar[d]|{D_{1,g}\otimes\varphi_{i,g-i}} \\
\wedge^{i}V^{\vee}\otimes\wedge^{g-i+1}V^{\vee}\oplus\wedge^{g-i}V^{\vee}\otimes\wedge^{i+1}V^{\vee} \ar[d]|{D_{i,g}\otimes D_{g-i+1,g}\oplus D_{g-i,g}\otimes D_{i+1,g}} & \wedge^{g-1}V\otimes\wedge^{g}V^{\vee}\otimes\wedge^{g}V^{\vee} \ar[d]|{r_{\wedge^{g}V}g\binom{g}{g-i}^{-1}\binom{r-i}{g-i}\cdot1_{\wedge^{g-1}V^{\vee}\otimes\wedge^{g}V^{\vee}\otimes\wedge^{g}V^{\vee}}} \\
\wedge^{g-i}V\otimes\wedge^{g}V^{\vee}\otimes\wedge^{i-1}V\otimes\wedge^{g}V^{\vee}\oplus\wedge^{i}V\otimes\wedge^{g}V^{\vee}\otimes\wedge^{g-i-1}V\otimes\wedge^{g}V^{\vee} \ar[r]_-{\left(-1\right)^{g-i}i\cdot\varphi_{g-i,i-1}^{13}\oplus\left(-1\right)^{i\left(g-i-1\right)}\left(g-i\right)\cdot\varphi_{i,g-i-1}^{13}} & \wedge^{g-1}V\otimes\wedge^{g}V^{\vee}\otimes\wedge^{g}V^{\vee}\text{.}
}
\]
\end{proposition}

\begin{proof}
We first apply Corollary \ref{FDP C3} with%
\[
g=\iota _{1,g}:V\rightarrow \hom \left( \wedge ^{g}V^{\vee },\wedge
^{g-1}V^{\vee }\right) \text{ and }\left( S,X,Y\right) =\left( \wedge
^{i}V,\wedge ^{g-i}V,\wedge ^{g}V\right) \text{.}
\]%
Since $r_{\wedge ^{g}V}=r_{\wedge ^{g}V}^{-1}$, the result is that, setting $%
\mu :=\binom{g}{g-i}^{-1}\binom{r-i}{g-i}$,%
\begin{eqnarray}
&&r_{\wedge ^{g}V}\mu \cdot \left( 1_{\wedge ^{g-1}V^{\vee }\otimes \wedge
^{g}V^{\vee \vee }}\otimes i_{\wedge ^{g}V}\right) \circ \left(
D^{1,g}\otimes \varphi _{i,g-i}\right) \circ \tau _{\wedge ^{i}V\otimes
\wedge ^{g-i}V,V}  \nonumber \\
&&\text{ }=\left( \varphi _{\iota _{1,g}}\otimes 1_{\wedge ^{g}V^{\vee \vee
}\otimes \wedge ^{g}V^{\vee \vee }}\right) \circ \left( 1_{V}\otimes \varphi
_{g-i,g}^{13}\right) \circ \left( 1_{V}\otimes D^{i,g}\otimes
D^{g-i,g}\right) \circ \tau _{\wedge ^{i}V\otimes \wedge ^{g-i}V,V}\text{.}
\label{Alternating algebras P2 F1}
\end{eqnarray}%
Here we recall that, by definition,%
\[
\varphi _{g-i,g}^{13}:=\left( \varphi _{g-i,i}\otimes 1_{\wedge ^{g}V^{\vee
\vee }\otimes \wedge ^{g}V^{\vee \vee }}\right) \circ \left( 1_{\wedge
^{g-i}V^{\vee }}\otimes \tau _{\wedge ^{g}V^{\vee \vee },\wedge ^{i}V^{\vee
}}\otimes 1_{\wedge ^{g}V^{\vee \vee }}\right) \text{.}
\]%
Hence we find%
\begin{eqnarray}
&&g\cdot \left( \varphi _{\iota _{1,g}}\otimes 1_{\wedge ^{g}V^{\vee \vee
}\otimes \wedge ^{g}V^{\vee \vee }}\right) \circ \left( 1_{V}\otimes \varphi
_{g-i,g}^{13}\right) =g\cdot \left( \varphi _{\iota _{1,g}}\otimes 1_{\wedge
^{g}V^{\vee \vee }\otimes \wedge ^{g}V^{\vee \vee }}\right)  \nonumber \\
&&\text{ }\circ \left( 1_{V}\otimes \varphi _{g-i,i}\otimes 1_{\wedge
^{g}V^{\vee \vee }\otimes \wedge ^{g}V^{\vee \vee }}\right) \circ \left(
1_{V\otimes \wedge ^{g-i}V^{\vee }}\otimes \tau _{\wedge ^{g}V^{\vee \vee
},\wedge ^{i}V^{\vee }}\otimes 1_{\wedge ^{g}V^{\vee \vee }}\right) 
\nonumber \\
&&\text{ }=g\cdot \left( \varphi _{\iota _{1,g}}\circ \left( 1_{V}\otimes
\varphi _{g-i,i}\right) \otimes 1_{\wedge ^{g}V^{\vee \vee }\otimes \wedge
^{g}V^{\vee \vee }}\right) \circ \left( 1_{V\otimes \wedge ^{g-i}V^{\vee
}}\otimes \tau _{\wedge ^{g}V^{\vee \vee },\wedge ^{i}V^{\vee }}\otimes
1_{\wedge ^{g}V^{\vee \vee }}\right) =a+b  \label{Alternating algebras P2 F2}
\end{eqnarray}%
where%
\begin{eqnarray*}
&&a:=\left( g-i\right) \cdot \left( \varphi _{g-i-1,i}\otimes 1_{\wedge
^{g}V^{\vee \vee }\otimes \wedge ^{g}V^{\vee \vee }}\right) \circ \left(
\varphi _{\iota _{1,g-i}}\otimes \tau _{\wedge ^{g}V^{\vee \vee },\wedge
^{i}V^{\vee }}\otimes 1_{\wedge ^{g}V^{\vee \vee }}\right) \text{,} \\
&&b:=\left( -1\right) ^{g-i}i\cdot \left( \varphi _{g-i,i-1}\otimes
1_{\wedge ^{g}V^{\vee \vee }\otimes \wedge ^{g}V^{\vee \vee }}\right) \circ
\left( 1_{\wedge ^{g-i}V^{\vee }}\otimes \varphi _{\iota _{1,i}}\otimes
1_{\wedge ^{g}V^{\vee \vee }\otimes \wedge ^{g}V^{\vee \vee }}\right) \\
&&\text{ \ \ \ \ \ }\circ \left( \tau _{V,\wedge ^{g-i}V^{\vee }}\otimes
\tau _{\wedge ^{g}V^{\vee \vee },\wedge ^{i}V^{\vee }}\otimes 1_{\wedge
^{g}V^{\vee \vee }}\right)
\end{eqnarray*}%
and where we have use the equality%
\begin{eqnarray*}
&&g\cdot \varphi _{\iota _{1,g}}\circ \left( 1_{V}\otimes \varphi
_{g-i,i}\right) =\left( g-i\right) \cdot \varphi _{g-i-1,i}\circ \left(
\varphi _{\iota _{1,g-i}}\otimes 1_{\wedge ^{i}V^{\vee }}\right) \\
&&\text{ \ \ \ \ \ \ \ \ \ }+\left( -1\right) ^{g-i}i\cdot \varphi
_{g-i,i-1}\circ \left( 1_{\wedge ^{g-i}V^{\vee }}\otimes \varphi _{\iota
_{1,i}}\right) \circ \left( \tau _{V,\wedge ^{g-i}V^{\vee }}\otimes
1_{\wedge ^{i}V^{\vee }}\right)
\end{eqnarray*}%
of Proposition \ref{Alternating algebras P1} at the end. Inserting $\left( 
\text{\ref{Alternating algebras P2 F2}}\right) $ in $\left( \text{\ref%
{Alternating algebras P2 F1}}\right) $ yields%
\begin{eqnarray}
&&r_{\wedge ^{g}V}\mu g\cdot \left( D^{1,g}\otimes \varphi _{i,g-i}\right)
\circ \tau _{\wedge ^{i}V\otimes \wedge ^{g-i}V,V}=a\circ \left(
1_{V}\otimes D^{i,g}\otimes D^{g-i,g}\right) \circ \tau _{\wedge
^{i}V\otimes \wedge ^{g-i}V,V}  \nonumber \\
&&\text{ \ \ }+b\circ \left( 1_{V}\otimes D^{i,g}\otimes D^{g-i,g}\right)
\circ \tau _{\wedge ^{i}V\otimes \wedge ^{g-i}V,V}\text{.}
\label{Alternating algebras P2 F3}
\end{eqnarray}

We now compute $a\circ \left( 1_{V}\otimes D^{i,g}\otimes D^{g-i,g}\right)
\circ \tau _{\wedge ^{i}V\otimes \wedge ^{g-i}V,V}$, using the following
formulas:%
\begin{eqnarray}
&&D^{i,g}\otimes D^{g-i,g}  \nonumber \\
&&\text{ \ \ }=\left( \varphi _{\iota _{i,g}}\otimes 1_{\wedge ^{g}V^{\vee
\vee }}\otimes \varphi _{\iota _{g-i,g}}\otimes 1_{\wedge ^{g}V^{\vee \vee
}}\right) \circ \left( 1_{\wedge ^{i}V}\otimes C_{\wedge ^{g}V^{\vee
}}\otimes 1_{\wedge ^{g-i}V}\otimes C_{\wedge ^{g}V^{\vee }}\right) \text{
(by }\left( \text{\ref{S1 Casimir D DefD_f}}\right) \text{)}
\label{Alternating algebras P2 F1 ab1} \\
&&\left( \varphi _{g-i-1,i}\otimes 1_{\wedge ^{g}V^{\vee \vee }\otimes
\wedge ^{g}V^{\vee \vee }}\right) \circ \left( 1_{\wedge ^{g-i-1}V^{\vee
}}\otimes \tau _{\wedge ^{g}V^{\vee \vee },\wedge ^{i}V^{\vee }}\otimes
1_{\wedge ^{g}V^{\vee \vee }}\right)  \nonumber \\
&&\text{ \ \ }=\varphi _{g-i-1,i}^{13}\text{ (by definition)}
\label{Alternating algebras P2 F1 a2} \\
&&\varphi _{\iota _{1,g-i}}\circ \left( 1_{V}\otimes \varphi _{\iota
_{i,g}}\right) =\varphi _{\iota _{i+1,g}}\circ \left( \varphi _{i,1}^{\tau
}\otimes 1_{\wedge ^{g}V^{\vee }}\right)  \nonumber \\
&&\text{ \ \ }=\left( -1\right) ^{i}\cdot \varphi _{\iota _{i+1,g}}\circ
\left( \varphi _{1,i}\otimes 1_{\wedge ^{g}V^{\vee }}\right) \text{ (by
Prop. \ref{S1 Algebras P1} }\left( 2\right) \text{)}
\label{Alternating algebras P2 F1 a3} \\
&&\left( \varphi _{\iota _{i+1,g}}\otimes 1_{\wedge ^{g}V^{\vee \vee
}}\otimes \varphi _{\iota _{g-i,g}}\otimes 1_{\wedge ^{g}V^{\vee \vee
}}\right) \circ \left( 1_{\wedge ^{i+1}V}\otimes C_{\wedge ^{g}V^{\vee
}}\otimes 1_{\wedge ^{g-i}V}\otimes C_{\wedge ^{g}V^{\vee }}\right) 
\nonumber \\
&&\text{ \ \ }=D^{i+1,g}\otimes D^{g-i,g}\text{ (by }\left( \text{\ref{S1
Casimir D DefD_f}}\right) \text{).}  \label{Alternating algebras P2 F1 a4}
\end{eqnarray}%
We have%
\begin{eqnarray}
&&a\circ \left( 1_{V}\otimes D^{i,g}\otimes D^{g-i,g}\right) \circ \tau
_{\wedge ^{i}V\otimes \wedge ^{g-i}V,V}=\left( g-i\right) \cdot \left(
\varphi _{g-i-1,i}\otimes 1_{\wedge ^{g}V^{\vee \vee }\otimes \wedge
^{g}V^{\vee \vee }}\right)  \nonumber \\
&&\text{ \ \ \ \ \ }\circ \left( \varphi _{\iota _{1,g-i}}\otimes \tau
_{\wedge ^{g}V^{\vee \vee },\wedge ^{i}V^{\vee }}\otimes 1_{\wedge
^{g}V^{\vee \vee }}\right) \circ \left( 1_{V}\otimes D^{i,g}\otimes
D^{g-i,g}\right) \circ \tau _{\wedge ^{i}V\otimes \wedge ^{g-i}V,V}\text{
(by }\left( \text{\ref{Alternating algebras P2 F1 ab1}}\right) \text{)} 
\nonumber \\
&&\text{ }=\left( g-i\right) \cdot \left( \varphi _{g-i-1,i}\otimes
1_{\wedge ^{g}V^{\vee \vee }\otimes \wedge ^{g}V^{\vee \vee }}\right) \circ
\left( 1_{\wedge ^{g-i-1}V^{\vee }}\otimes \tau _{\wedge ^{g}V^{\vee \vee
},\wedge ^{i}V^{\vee }}\otimes 1_{\wedge ^{g}V^{\vee \vee }}\right) 
\nonumber \\
&&\text{ \ \ \ \ \ }\circ \left( \varphi _{\iota _{1,g-i}}\otimes 1_{\wedge
^{g}V^{\vee \vee }\otimes \wedge ^{i}V^{\vee }\otimes \wedge ^{g}V^{\vee
\vee }}\right) \circ \left( 1_{V}\otimes \varphi _{\iota _{i,g}}\otimes
1_{\wedge ^{g}V^{\vee \vee }}\otimes \varphi _{\iota _{g-i,g}}\otimes
1_{\wedge ^{g}V^{\vee \vee }}\right)  \nonumber \\
&&\text{ \ \ \ \ \ }\circ \left( 1_{V\otimes \wedge ^{i}V}\otimes C_{\wedge
^{g}V^{\vee }}\otimes 1_{\wedge ^{g-i}V}\otimes C_{\wedge ^{g}V^{\vee
}}\right) \circ \tau _{\wedge ^{i}V\otimes \wedge ^{g-i}V,V}\text{ (by }%
\left( \text{\ref{Alternating algebras P2 F1 a2}}\right) \text{)}  \nonumber
\\
&&\text{ }=\left( g-i\right) \cdot \varphi _{g-i-1,i}^{13}\circ \left(
\varphi _{\iota _{1,g-i}}\otimes 1_{\wedge ^{g}V^{\vee \vee }\otimes \wedge
^{i}V^{\vee }\otimes \wedge ^{g}V^{\vee \vee }}\right) \circ \left(
1_{V}\otimes \varphi _{\iota _{i,g}}\otimes 1_{\wedge ^{g}V^{\vee \vee
}}\otimes \varphi _{\iota _{g-i,g}}\otimes 1_{\wedge ^{g}V^{\vee \vee
}}\right)  \nonumber \\
&&\text{ \ \ \ \ \ }\circ \left( 1_{V\otimes \wedge ^{i}V}\otimes C_{\wedge
^{g}V^{\vee }}\otimes 1_{\wedge ^{g-i}V}\otimes C_{\wedge ^{g}V^{\vee
}}\right) \circ \tau _{\wedge ^{i}V\otimes \wedge ^{g-i}V,V}\text{ (by }%
\left( \text{\ref{Alternating algebras P2 F1 a3}}\right) \text{)}  \nonumber
\\
&&\text{ }=\left( -1\right) ^{i}\left( g-i\right) \cdot \varphi
_{g-i-1,i}^{13}\circ \left( \varphi _{\iota _{i+1,g}}\otimes 1_{\wedge
^{g}V^{\vee \vee }\otimes \wedge ^{i}V^{\vee }\otimes \wedge ^{g}V^{\vee
\vee }}\right)  \nonumber \\
&&\text{ \ \ \ \ \ }\circ \left( \varphi _{1,i}\otimes 1_{\wedge ^{g}V^{\vee
}\otimes \wedge ^{g}V^{\vee \vee }}\otimes \varphi _{\iota _{g-i,g}}\otimes
1_{\wedge ^{g}V^{\vee \vee }}\right)  \nonumber \\
&&\text{ \ \ \ \ \ }\circ \left( 1_{V\otimes \wedge ^{i}V}\otimes C_{\wedge
^{g}V^{\vee }}\otimes 1_{\wedge ^{g-i}V}\otimes C_{\wedge ^{g}V^{\vee
}}\right) \circ \tau _{\wedge ^{i}V\otimes \wedge ^{g-i}V,V}  \nonumber \\
&&\text{ }=\left( -1\right) ^{i}\left( g-i\right) \cdot \varphi
_{g-i-1,i}^{13}\circ \left( \varphi _{\iota _{i+1,g}}\otimes 1_{\wedge
^{g}V^{\vee \vee }}\otimes \varphi _{\iota _{g-i,g}}\otimes 1_{\wedge
^{g}V^{\vee \vee }}\right)  \nonumber \\
&&\text{ \ \ \ \ \ }\circ \left( 1_{\wedge ^{i+1}V}\otimes C_{\wedge
^{g}V^{\vee }}\otimes 1_{\wedge ^{g-i}V}\otimes C_{\wedge ^{g}V^{\vee
}}\right)  \nonumber \\
&&\text{ \ \ \ \ \ }\circ \left( \varphi _{1,i}\otimes 1_{\wedge
^{g-i}V}\right) \circ \tau _{\wedge ^{i}V\otimes \wedge ^{g-i}V,V}\text{ (by 
}\left( \text{\ref{Alternating algebras P2 F1 a4}}\right) \text{)}  \nonumber
\\
&&\text{ }=\left( -1\right) ^{i}\left( g-i\right) \cdot \varphi
_{g-i-1,i}^{13}\circ \left( D^{i+1,g}\otimes D^{g-i,g}\right) \circ \left(
\varphi _{1,i}\otimes 1_{\wedge ^{g-i}V}\right) \circ \tau _{\wedge
^{i}V\otimes \wedge ^{g-i}V,V}  \label{Alternating algebras P2 F1 a}
\end{eqnarray}

We compute $b\circ \left( 1_{V}\otimes D^{i,g}\otimes D^{g-i,g}\right) \circ
\tau _{\wedge ^{i}V\otimes \wedge ^{g-i}V,V}$, using the following formulas:%
\begin{eqnarray}
&&\varphi _{g-i,i-1}\otimes 1_{\wedge ^{g}V^{\vee \vee }\otimes \wedge
^{g}V^{\vee \vee }}  \nonumber \\
&&\text{ \ \ }=\varphi _{g-i,i-1}^{13}\circ \left( 1_{\wedge ^{g-i}V^{\vee
}}\otimes \tau _{\wedge ^{i-1}V^{\vee },\wedge ^{g}V^{\vee \vee }}\otimes
1_{\wedge ^{g}V^{\vee \vee }}\right) \text{ (by definition)}
\label{Alternating algebras P2 F1 b2} \\
&&\varphi _{\iota _{1,i}}\circ \left( 1_{V}\otimes \varphi _{\iota
_{g-i,g}}\right) =\varphi _{\iota _{g-i+1,g}}\circ \left( \varphi
_{g-i,1}^{\tau }\otimes 1_{\wedge ^{g}V^{\vee }}\right)   \nonumber \\
&&\text{ \ \ }=\left( -1\right) ^{g-i}\cdot \varphi _{\iota _{g-i+1,g}}\circ
\left( \varphi _{1,g-i}\otimes 1_{\wedge ^{g}V^{\vee }}\right) \text{ (by
Prop. \ref{S1 Algebras P1} }\left( 2\right) \text{)}
\label{Alternating algebras P2 F1 b3} \\
&&\left( \varphi _{\iota _{i,g}}\otimes 1_{\wedge ^{g}V^{\vee \vee }}\otimes
\varphi _{\iota _{g-i+1,g}}\otimes 1_{\wedge ^{g}V^{\vee \vee }}\right)
\circ \left( 1_{\wedge ^{i}V}\otimes C_{\wedge ^{g}V^{\vee }}\otimes
1_{\wedge ^{g-i+1}V}\otimes C_{\wedge ^{g}V^{\vee }}\right)   \nonumber \\
&&\text{ \ \ }=D^{i,g}\otimes D^{g-i+1,g}\text{ (by }\left( \text{\ref{S1
Casimir D DefD_f}}\right) \text{),}  \label{Alternating algebras P2 F1 b4}
\end{eqnarray}%
together with the following equality, which is the consequence of a boring
computation involving the functoriality of the $\otimes $-operation, that of
the $\tau $-constraint and the anti-commutativity constraint in the
alternating algebra:%
\begin{eqnarray}
&&\left( 1_{\wedge ^{g-i}V^{\vee }}\otimes \tau _{\wedge ^{i-1}V^{\vee
},\wedge ^{g}V^{\vee \vee }}\otimes 1_{\wedge ^{g}V^{\vee \vee }}\right)
\circ \left( 1_{\wedge ^{g-i}V^{\vee }}\otimes \varphi _{\iota
_{g-i+1,g}}\otimes 1_{\wedge ^{g}V^{\vee \vee }\otimes \wedge ^{g}V^{\vee
\vee }}\right)   \nonumber \\
&&\text{ \ \ \ \ }\circ \left( \varphi _{\iota _{i,g}}\otimes \varphi
_{1,g-i}\otimes 1_{\wedge ^{g}V^{\vee }\otimes \wedge ^{g}V^{\vee \vee
}\otimes \wedge ^{g}V^{\vee \vee }}\right) \circ \left( \tau _{V,\wedge
^{i}V\otimes \wedge ^{g}V^{\vee }}\otimes \tau _{\wedge ^{g}V^{\vee \vee
},\wedge ^{g-i}V\otimes \wedge ^{g}V^{\vee }}\otimes 1_{\wedge ^{g}V^{\vee
\vee }}\right)   \nonumber \\
&&\text{ \ \ \ \ }\circ \left( 1_{V\otimes \wedge ^{i}V}\otimes C_{\wedge
^{g}V^{\vee }}\otimes 1_{\wedge ^{g-i}V}\otimes C_{\wedge ^{g}V^{\vee
}}\right) \circ \tau _{\wedge ^{i}V\otimes \wedge ^{g-i}V,V}  \nonumber \\
&&\text{ \ }=\left( -1\right) ^{g-i}\left( \varphi _{\iota _{i,g}}\otimes
1_{\wedge ^{g}V^{\vee \vee }}\otimes \varphi _{\iota _{g-i+1,g}}\otimes
1_{\wedge ^{g}V^{\vee \vee }}\right) \circ \left( 1_{\wedge ^{i}V}\otimes
C_{\wedge ^{g}V^{\vee }}\otimes 1_{\wedge ^{g-i+1}V}\otimes C_{\wedge
^{g}V^{\vee }}\right)   \nonumber \\
&&\text{ \ \ \ \ }\circ \left( 1_{\wedge ^{i}V^{\vee }}\otimes \varphi
_{g-i,1}\right) \text{.}  \label{Alternating algebras P2 F1 b5}
\end{eqnarray}%
We have%
\begin{eqnarray}
&&b\circ \left( 1_{V}\otimes D^{i,g}\otimes D^{g-i,g}\right) \circ \tau
_{\wedge ^{i}V\otimes \wedge ^{g-i}V,V}=\left( -1\right) ^{g-i}i\cdot \left(
\varphi _{g-i,i-1}\otimes 1_{\wedge ^{g}V^{\vee \vee }\otimes \wedge
^{g}V^{\vee \vee }}\right)   \nonumber \\
&&\text{ \ \ \ \ }\circ \left( 1_{\wedge ^{g-i}V^{\vee }}\otimes \varphi
_{\iota _{1,i}}\otimes 1_{\wedge ^{g}V^{\vee \vee }\otimes \wedge
^{g}V^{\vee \vee }}\right) \circ \left( \tau _{V,\wedge ^{g-i}V^{\vee
}}\otimes \tau _{\wedge ^{g}V^{\vee \vee },\wedge ^{i}V^{\vee }}\otimes
1_{\wedge ^{g}V^{\vee \vee }}\right)   \nonumber \\
&&\text{ \ \ \ \ }\circ \left( 1_{V}\otimes D^{i,g}\otimes D^{g-i,g}\right)
\circ \tau _{\wedge ^{i}V\otimes \wedge ^{g-i}V,V}\text{ (by }\left( \text{%
\ref{Alternating algebras P2 F1 ab1}}\right) \text{)}  \nonumber \\
&&\text{ }=\left( -1\right) ^{g-i}i\cdot \left( \varphi _{g-i,i-1}\otimes
1_{\wedge ^{g}V^{\vee \vee }\otimes \wedge ^{g}V^{\vee \vee }}\right) \circ
\left( 1_{\wedge ^{g-i}V^{\vee }}\otimes \varphi _{\iota _{1,i}}\otimes
1_{\wedge ^{g}V^{\vee \vee }\otimes \wedge ^{g}V^{\vee \vee }}\right)  
\nonumber \\
&&\text{ \ \ \ \ }\circ \left( \tau _{V,\wedge ^{g-i}V^{\vee }}\otimes \tau
_{\wedge ^{g}V^{\vee \vee },\wedge ^{i}V^{\vee }}\otimes 1_{\wedge
^{g}V^{\vee \vee }}\right) \circ \left( 1_{V}\otimes \varphi _{\iota
_{i,g}}\otimes 1_{\wedge ^{g}V^{\vee \vee }}\otimes \varphi _{\iota
_{g-i,g}}\otimes 1_{\wedge ^{g}V^{\vee \vee }}\right)   \nonumber \\
&&\text{ \ \ \ \ }\circ \left( 1_{V\otimes \wedge ^{i}V}\otimes C_{\wedge
^{g}V^{\vee }}\otimes 1_{\wedge ^{g-i}V}\otimes C_{\wedge ^{g}V^{\vee
}}\right) \circ \tau _{\wedge ^{i}V\otimes \wedge ^{g-i}V,V}\text{ (by }%
\left( \text{\ref{Alternating algebras P2 F1 b2}}\right) \text{)}  \nonumber
\\
&&\text{ }=\left( -1\right) ^{g-i}i\cdot \varphi _{g-i,i-1}^{13}\circ \left(
1_{\wedge ^{g-i}V^{\vee }}\otimes \tau _{\wedge ^{i-1}V^{\vee },\wedge
^{g}V^{\vee \vee }}\otimes 1_{\wedge ^{g}V^{\vee \vee }}\right) \circ \left(
1_{\wedge ^{g-i}V^{\vee }}\otimes \varphi _{\iota _{1,i}}\otimes 1_{\wedge
^{g}V^{\vee \vee }\otimes \wedge ^{g}V^{\vee \vee }}\right)   \nonumber \\
&&\text{ \ \ \ \ }\circ \left( \tau _{V,\wedge ^{g-i}V^{\vee }}\otimes \tau
_{\wedge ^{g}V^{\vee \vee },\wedge ^{i}V^{\vee }}\otimes 1_{\wedge
^{g}V^{\vee \vee }}\right) \circ \left( 1_{V}\otimes \varphi _{\iota
_{i,g}}\otimes 1_{\wedge ^{g}V^{\vee \vee }}\otimes \varphi _{\iota
_{g-i,g}}\otimes 1_{\wedge ^{g}V^{\vee \vee }}\right)   \nonumber \\
&&\text{ \ \ \ \ }\circ \left( 1_{V\otimes \wedge ^{i}V}\otimes C_{\wedge
^{g}V^{\vee }}\otimes 1_{\wedge ^{g-i}V}\otimes C_{\wedge ^{g}V^{\vee
}}\right) \circ \tau _{\wedge ^{i}V\otimes \wedge ^{g-i}V,V}  \nonumber \\
&&\text{ }=\left( -1\right) ^{g-i}i\cdot \varphi _{g-i,i-1}^{13}\circ \left(
1_{\wedge ^{g-i}V^{\vee }}\otimes \tau _{\wedge ^{i-1}V^{\vee },\wedge
^{g}V^{\vee \vee }}\otimes 1_{\wedge ^{g}V^{\vee \vee }}\right) \circ \left(
1_{\wedge ^{g-i}V^{\vee }}\otimes \varphi _{\iota _{1,i}}\otimes 1_{\wedge
^{g}V^{\vee \vee }\otimes \wedge ^{g}V^{\vee \vee }}\right)   \nonumber \\
&&\text{ \ \ \ \ }\circ \left( \varphi _{\iota _{i,g}}\otimes 1_{V}\otimes
\varphi _{\iota _{g-i,g}}\otimes 1_{\wedge ^{g}V^{\vee \vee }\otimes \wedge
^{g}V^{\vee \vee }}\right) \circ \left( \tau _{V,\wedge ^{i}V\otimes \wedge
^{g}V^{\vee }}\otimes \tau _{\wedge ^{g}V^{\vee \vee },\wedge ^{g-i}V\otimes
\wedge ^{g}V^{\vee }}\otimes 1_{\wedge ^{g}V^{\vee \vee }}\right)   \nonumber
\\
&&\text{ \ \ \ \ }\circ \left( 1_{V\otimes \wedge ^{i}V}\otimes C_{\wedge
^{g}V^{\vee }}\otimes 1_{\wedge ^{g-i}V}\otimes C_{\wedge ^{g}V^{\vee
}}\right) \circ \tau _{\wedge ^{i}V\otimes \wedge ^{g-i}V,V}\text{ (by }%
\left( \text{\ref{Alternating algebras P2 F1 b3}}\right) \text{)}  \nonumber
\\
&&\text{ }=i\cdot \varphi _{g-i,i-1}^{13}\circ \left( 1_{\wedge
^{g-i}V^{\vee }}\otimes \tau _{\wedge ^{i-1}V^{\vee },\wedge ^{g}V^{\vee
\vee }}\otimes 1_{\wedge ^{g}V^{\vee \vee }}\right) \circ \left( 1_{\wedge
^{g-i}V^{\vee }}\otimes \varphi _{\iota _{g-i+1,g}}\otimes 1_{\wedge
^{g}V^{\vee \vee }\otimes \wedge ^{g}V^{\vee \vee }}\right)   \nonumber \\
&&\text{ \ \ \ \ }\circ \left( \varphi _{\iota _{i,g}}\otimes \varphi
_{1,g-i}\otimes 1_{\wedge ^{g}V^{\vee }\otimes \wedge ^{g}V^{\vee \vee
}\otimes \wedge ^{g}V^{\vee \vee }}\right) \circ \left( \tau _{V,\wedge
^{i}V\otimes \wedge ^{g}V^{\vee }}\otimes \tau _{\wedge ^{g}V^{\vee \vee
},\wedge ^{g-i}V\otimes \wedge ^{g}V^{\vee }}\otimes 1_{\wedge ^{g}V^{\vee
\vee }}\right)   \nonumber \\
&&\text{ \ \ \ \ }\circ \left( 1_{V\otimes \wedge ^{i}V}\otimes C_{\wedge
^{g}V^{\vee }}\otimes 1_{\wedge ^{g-i}V}\otimes C_{\wedge ^{g}V^{\vee
}}\right) \circ \tau _{\wedge ^{i}V\otimes \wedge ^{g-i}V,V}\text{ (by }%
\left( \text{\ref{Alternating algebras P2 F1 b5}}\right) \text{)}  \nonumber
\\
&&\text{ }=\left( -1\right) ^{g-i}i\cdot \varphi _{g-i,i-1}^{13}\circ \left(
\varphi _{\iota _{i,g}}\otimes 1_{\wedge ^{g}V^{\vee \vee }}\otimes \varphi
_{\iota _{g-i+1,g}}\otimes 1_{\wedge ^{g}V^{\vee \vee }}\right)   \nonumber
\\
&&\text{ \ \ \ \ }\circ \left( 1_{\wedge ^{i}V}\otimes C_{\wedge ^{g}V^{\vee
}}\otimes 1_{\wedge ^{g-i+1}V}\otimes C_{\wedge ^{g}V^{\vee }}\right) \circ
\left( 1_{\wedge ^{i}V^{\vee }}\otimes \varphi _{g-i,1}\right) \text{ (by }%
\left( \text{\ref{Alternating algebras P2 F1 b4}}\right) \text{)}  \nonumber
\\
&&\text{ }=\left( -1\right) ^{g-i}i\cdot \varphi _{g-i,i-1}^{13}\circ \left(
D^{i,g}\otimes D^{g-i+1,g}\right) \circ \left( 1_{\wedge ^{i}V^{\vee
}}\otimes \varphi _{g-i,1}\right) \text{.}
\label{Alternating algebras P2 F1 b}
\end{eqnarray}

Inserting $\left( \text{\ref{Alternating algebras P2 F1 a}}\right) $ and $%
\left( \text{\ref{Alternating algebras P2 F1 b}}\right) $ in $\left( \text{%
\ref{Alternating algebras P2 F3}}\right) $ gives%
\begin{eqnarray}
&&r_{\wedge ^{g}V}\mu g\cdot \left( D^{1,g}\otimes \varphi _{i,g-i}\right)
\circ \tau _{\wedge ^{i}V\otimes \wedge ^{g-i}V,V}  \nonumber \\
&&\text{ }=\left( -1\right) ^{i}\left( g-i\right) \cdot \varphi
_{g-i-1,i}^{13}\circ \left( D^{i+1,g}\otimes D^{g-i,g}\right) \circ \left(
\varphi _{1,i}\otimes 1_{\wedge ^{g-i}V}\right) \circ \tau _{\wedge
^{i}V\otimes \wedge ^{g-i}V,V}  \nonumber \\
&&\text{ \ \ \ \ }+\left( -1\right) ^{g-i}i\cdot \varphi
_{g-i,i-1}^{13}\circ \left( D^{i,g}\otimes D^{g-i+1,g}\right) \circ \left(
1_{\wedge ^{i}V^{\vee }}\otimes \varphi _{g-i,1}\right) \text{.}
\label{Alternating algebras P2 F4}
\end{eqnarray}%
Another computation involving the functoriality of the $\otimes $-operation,
that of the $\tau $-constraint and the anti-commutativity constraint in the
alternating algebra reveals that:%
\begin{eqnarray}
&&\left( -1\right) ^{i}\left( g-i\right) \cdot \varphi _{g-i-1,i}^{13}\circ
\left( D^{i+1,g}\otimes D^{g-i,g}\right) \circ \left( \varphi _{1,i}\otimes
1_{\wedge ^{g-i}V}\right) \circ \tau _{\wedge ^{i}V^{\vee }\otimes \wedge
^{g-i}V^{\vee },V^{\vee }}  \nonumber \\
&&\text{ }=\left( -1\right) ^{i\left( g-i-1\right) }\left( g-i\right) \cdot
\varphi _{i,g-i-1}^{13}\circ \left( D^{g-i,g}\otimes D^{i+1,g}\right) \circ
\left( 1_{\wedge ^{g-i}V}\otimes \varphi _{i,1}\right) \circ \left( \tau
_{\wedge ^{i}V,\wedge ^{g-i}V}\otimes 1_{V}\right) \text{.}
\label{Alternating algebras P2 F5}
\end{eqnarray}%
The commutativity of the first diagram now follows from $\left( \text{\ref%
{Alternating algebras P2 F4}}\right) $ and $\left( \text{\ref{Alternating
algebras P2 F5}}\right) $.

The second commutative diagram is obtained in a similar way, starting with
Corollary \ref{FDP C3} applied with $h=\iota _{1,g}^{\ast }:V^{\vee
}\rightarrow \hom \left( \wedge ^{g}V,\wedge ^{g-1}V\right) $ and $\left(
S,X,Y\right) =\left( \wedge ^{i}V,\wedge ^{g-i}V,\wedge ^{g}V\right) $ and
employing the appropriate dual statements.
\end{proof}

\subsection{Proof of Lemma \protect\ref{Alternating algebras L key}}

The proof of Lemma \ref{Alternating algebras L key} will be divided in
several steps. We will use the shorthand $C_{p}:=C_{\wedge ^{p}V}:\mathbb{I}%
\rightarrow \wedge _{p}^{p}V$ in the sequel.

\bigskip

\textbf{Step 1}

We claim the commutativity of the following diagrams for every $m\geq 1$:%
\begin{equation}
\xymatrix{ & \wedge_{0}^{1}V\otimes\wedge_{1}^{1}V\otimes\wedge_{m-1}^{m-1}V \ar[r]^{\delta_{0,1}^{1,1}\otimes1_{\wedge_{m-1}^{m-1}V}} & \wedge_{0}^{1}V\otimes\wedge_{m-1}^{m-1}V \ar[d]|{\varphi_{0,m-1}^{1,m-1}} & & \wedge_{1}^{0}V\otimes\wedge_{1}^{1}V\otimes\wedge_{m-1}^{m-1}V \ar[r]^{\delta_{1,1}^{0,1}\otimes1_{\wedge_{m-1}^{m-1}V}} & \wedge_{1}^{0}V\otimes\wedge_{m-1}^{m-1}V \ar[d]|{\varphi_{1,m-1}^{0,m-1}} \\ V \ar[r]^{1_{V}\otimes C_{m}} \ar@/^{0.75pc}/[ur]|-{1_{V}\otimes C_{1}\otimes C_{m-1}} & \wedge_{0}^{1}V\otimes\wedge_{m}^{m}V \ar[r]^{\delta_{0,m}^{1,m}} & \wedge_{m-1}^{m}V & V^{\vee} \ar[r]^{1_{V^{\vee}}\otimes C_{m}} \ar@/^{0.75pc}/[ur]|-{1_{V^{\vee}}\otimes C_{1}\otimes C_{m-1}} & \wedge_{1}^{0}V\otimes\wedge_{m}^{m}V \ar[r]^{\delta_{1,m}^{0,m}} & \wedge_{m}^{m-1}V\text{.} }  \label{key claim D1}
\end{equation}%
The proof of the commutativity of the second diagram is identical to that of
the first one, so we will concentrate on the first. The case $m=1$ is
trivial and the general case is done by induction, assuming it true for $m$.

We will first need a simple lemma, whose proof is just an application of Lemma \ref{S1 Casimir L Properties} $\left( 3\right)$, Lemma \ref{S1 Casimir L Properties} $\left( 4\right)$ and $\left( 
\text{\ref{Induced morphisms D1}}\right) $.

\begin{lemma}
\label{key claim L1}The following diagram is commutative, for every $p\geq 0$%
,%
\begin{equation*}
\xymatrix{ & \wedge_{p}^{p}V\otimes\wedge_{1}^{1}V \ar[d]^{\varphi_{p,1}^{p,1}} \\ \mathbb{I} \ar[r]^-{C_{p+1}} \ar@/^{0.75pc}/[ur]|-{C_{p}\otimes C_{1}} & \wedge_{p+1}^{p+1}V\text{.}}
\end{equation*}
\end{lemma}

We now consider the following diagram, where%
\begin{equation*}
f:=m\cdot \delta _{0,m}^{1,m}\otimes 1_{\wedge _{1}^{1}V}\text{ and }%
g:=\left( -1\right) ^{m}\cdot \left( 1_{\wedge _{m}^{m}V}\otimes \delta
_{0,1}^{1,1}\right) \circ \left( \tau _{\wedge _{0}^{1}V,\wedge
_{m}^{m}V}\otimes 1_{\wedge _{1}^{1}V}\right) \text{:}
\end{equation*}%
\begin{equation*}
\xymatrix{ & \wedge_{0}^{1}V\otimes\wedge_{m}^{m}V\otimes\wedge_{1}^{1}V \ar@{}[dr]|{(B)} \ar[d]|{1_{\wedge_{0}^{1}V}\otimes\varphi_{m,1}^{m,1}} \ar[r]^-{\left(f,g\right)} & \wedge_{m-1}^{m}V\otimes\wedge_{1}^{1}V\oplus\wedge_{m}^{m}V\otimes\wedge_{0}^{1}V \ar[d]^{\varphi_{m-1,1}^{m,1}\oplus\varphi_{m,0}^{m,1}} \\ V \ar[r]^-{1_{V}\otimes C_{m+1}} \ar@/^{0.75pc}/[ur]|-{1_{V}\otimes C_{m}\otimes C_{1}} & \wedge_{0}^{1}V\otimes\wedge_{m+1}^{m+1}V \ar@{}[ul]|(0.37){(A)} \ar[r]^-{\left(m+1\right)\cdot\delta_{0,m+1}^{1,m+1}} & \wedge_{m}^{m+1}V\text{.}}
\end{equation*}%
The region $\left( A\right) $ commutes by Lemma \ref{key claim L1} and the
region $\left( B\right) $ is commutative thanks to  the first diagram of
Corollary \ref{Alternating algebras C1}. We deduce that we have:%
\begin{equation}
\left( m+1\right) \cdot \delta _{0,m+1}^{1,m+1}\circ \left( 1_{V}\otimes
C_{m+1}\right) =\left( 1_{\wedge _{m}^{m+1}V}\oplus 1_{\wedge
_{m}^{m+1}V}\right) \circ \left( a,b\right) =a+b\text{,}
\label{key claim D1 F}
\end{equation}%
where%
\begin{eqnarray*}
a &=&m\cdot \varphi _{m-1,1}^{m,1}\circ \left( \delta _{0,m}^{1,m}\otimes
1_{\wedge _{1}^{1}V}\right) \circ \left( 1_{V}\otimes C_{m}\otimes
C_{1}\right) \text{,} \\
b &=&\left( -1\right) ^{m}\cdot \varphi _{m,0}^{m,1}\circ \left( 1_{\wedge
_{m}^{m}V}\otimes \delta _{0,1}^{1,1}\right) \circ \left( \tau _{\wedge
_{0}^{1}V,\wedge _{m}^{m}V}\otimes 1_{\wedge _{1}^{1}V}\right) \circ \left(
1_{V}\otimes C_{m}\otimes C_{1}\right) \text{.}
\end{eqnarray*}

We will now derive an alternative expression for $a$ by looking at the
following diagram:%
\begin{equation*}
\xymatrix{ & \wedge_{0}^{1}V\otimes\wedge_{m}^{m}V\otimes\wedge_{1}^{1}V \ar@{}[dr]|(0.4){(C)} \ar[r]^-{\delta_{0,m}^{1,m}\otimes1_{\wedge_{1}^{1}V}} & \wedge_{m-1}^{m}V\otimes\wedge_{1}^{1}V \ar@/^{0.75pc}/[dr]|-{\varphi_{m-1,1}^{m,1}} & \\ V \ar@/^{0.75pc}/[ur]|-{1_{V}\otimes C_{m}\otimes C_{1}} \ar[r]^-{1_{V}\otimes C_{1}\otimes C_{m-1}\otimes C_{1}} \ar@/_{0.75pc}/[dr]|-{1_{V}\otimes C_{1}\otimes C_{m}} & \wedge_{0}^{1}V\otimes\wedge_{1}^{1}V\otimes\wedge_{m-1}^{m-1}V\otimes\wedge_{1}^{1}V \ar@{}[dr]|{(\otimes)} \ar@{}[dl]|(0.3){(A)} \ar[r]^-{\delta_{0,1}^{1,1}\otimes1_{\wedge_{m-1}^{m-1}V\otimes\wedge_{1}^{1}V}} \ar[d]|{1_{\wedge_{0}^{1}V\otimes\wedge_{1}^{1}V}\otimes\varphi_{m-1,1}^{m-1,1}} & \wedge_{0}^{1}V\otimes\wedge_{m-1}^{m-1}V\otimes\wedge_{1}^{1}V \ar[u]|{\varphi_{0,m-1}^{1,m-1}\otimes1_{\wedge_{1}^{1}V}} \ar[d]|{1_{\wedge_{0}^{1}V}\otimes\varphi_{m-1,1}^{m-1,1}} & \wedge_{m}^{m+1}V \ar@{}[l]|(0.35){(D)} \\ & \wedge_{0}^{1}V\otimes\wedge_{1}^{1}V\otimes\wedge_{m}^{m}V \ar[r]^-{\delta_{0,1}^{1,1}\otimes1_{\wedge_{m}^{m}V}} & \wedge_{0}^{1}V\otimes\wedge_{m}^{m}V \ar@/_{0.75pc}/[ur]|-{\varphi_{0,m}^{1,m}} & }
\end{equation*}%
Here the region $\left( A\right) $ is again commutative by Lemma \ref{key
claim L1}, the region $\left( C\right) $ is commutative by our induction
assumption $\left( \text{\ref{key claim D1}}\right) $ and the functoriality
of $\otimes $ and $\left( D\right) $ commutes by the associativity
constraint. We deduce%
\begin{equation}
a=m\cdot \varphi _{0,m}^{1,m}\circ \left( \delta _{0,1}^{1,1}\otimes
1_{\wedge _{m}^{m}V}\right) \circ \left( 1_{V}\otimes C_{1}\otimes
C_{m}\right) \text{.}  \label{key claim D1 Fa}
\end{equation}

We now compute $b$, noticing that $\tau _{\wedge _{0}^{1}V,\wedge
_{m}^{m}V}\otimes 1_{\wedge _{1}^{1}V}=\tau _{\wedge _{0}^{1}V\otimes \wedge
_{1}^{1}V,\wedge _{m}^{m}V}\circ \left( 1_{\wedge _{0}^{1}V}\otimes \tau
_{\wedge _{m}^{m}V,\wedge _{1}^{1}V}\right) $ in the first of the subsequent
equalities, employing the functoriality of $\tau $ in the second one and
finally appealing to the commutativity constraint $\varphi _{m,0}^{m,1}\circ
\tau _{\wedge _{0}^{1}V,\wedge _{m}^{m}V}=\left( -1\right) ^{m}\varphi
_{0,m}^{1,m}$ in third equality:%
\begin{eqnarray}
b &=&\left( -1\right) ^{m}\cdot \varphi _{m,0}^{m,1}\circ \left( 1_{\wedge
_{m}^{m}V}\otimes \delta _{0,1}^{1,1}\right) \circ \tau _{\wedge
_{0}^{1}V\otimes \wedge _{1}^{1}V,\wedge _{m}^{m}V}\circ \left( 1_{\wedge
_{0}^{1}V}\otimes \tau _{\wedge _{m}^{m}V,\wedge _{1}^{1}V}\right) \circ
\left( 1_{V}\otimes C_{m}\otimes C_{1}\right)  \notag \\
&=&\left( -1\right) ^{m}\cdot \varphi _{m,0}^{m,1}\circ \tau _{\wedge
_{0}^{1}V,\wedge _{m}^{m}V}\circ \left( \delta _{0,1}^{1,1}\otimes 1_{\wedge
_{m}^{m}V}\right) \circ \left( 1_{V}\otimes C_{1}\otimes C_{m}\right)  \notag
\\
&=&\varphi _{0,m}^{1,m}\circ \left( \delta _{0,1}^{1,1}\otimes 1_{\wedge
_{m}^{m}V}\right) \circ \left( 1_{V}\otimes C_{1}\otimes C_{m}\right) \text{.%
}  \label{key claim D1 Fb}
\end{eqnarray}

Inserting $\left( \text{\ref{key claim D1 Fa}}\right) $ and $\left( \text{%
\ref{key claim D1 Fb}}\right) $ in $\left( \text{\ref{key claim D1 F}}%
\right) $ we deduce%
\begin{equation*}
\left( m+1\right) \cdot \delta _{0,m+1}^{1,m+1}\circ \left( 1_{V}\otimes
C_{m+1}\right) =\left( m+1\right) \cdot \varphi _{0,m}^{1,m}\circ \left(
\delta _{0,1}^{1,1}\otimes 1_{\wedge _{m}^{m}V}\right) \circ \left(
1_{V}\otimes C_{1}\otimes C_{m}\right) \text{,}
\end{equation*}%
from which the claim follows.

\bigskip

\textbf{Step 2}

Noticing that $\delta _{0,1}^{1,1}=\left( 1_{V}\otimes \ev_{V}^{\tau }\right)
\circ \left( \tau _{V,V}\otimes 1_{V^{\vee }}\right) $ and $\delta
_{1,1}^{0,1}=\ev_{V}\otimes 1_{V^{\vee }}$, we deduce from Lemma \ref{S1
Casimir L Properties} $\left( 2\right) $ that we have%
\begin{eqnarray*}
&&\delta _{0,1}^{1,1}\circ \left( 1_{V}\otimes C_{1}\right) =\left(
1_{V}\otimes \ev_{V}\right) \circ \left( 1_{V}\otimes \tau _{V,V^{\vee
}}\right) \circ \left( \tau _{V,V}\otimes 1_{V^{\vee }}\right) \circ \left(
1_{V}\otimes C_{V}\right) \\
&&\text{ }=\left( 1_{V}\otimes \ev_{V}\right) \circ \tau _{V,V\otimes V^{\vee
}}\circ \left( 1_{V}\otimes C_{V}\right) =\left( 1_{V}\otimes \ev_{V}\right)
\circ \left( C_{V}\otimes 1_{V}\right) =1_{V}\text{,} \\
&&\delta _{1,1}^{0,1}\circ \left( 1_{V}\otimes C_{1}\right) =\left(
\ev_{V}\otimes 1_{V^{\vee }}\right) \circ \left( 1_{V}\otimes C_{V}\right)
=1_{V^{\vee }}\text{.}
\end{eqnarray*}%
Hence it follows from $\left( \text{\ref{key claim D1}}\right) $ that the
following diagrams are commutative:%
\begin{equation}
\xymatrix{ V \ar[r]^-{1_{V}\otimes C_{m-1}} \ar[d]|{1_{V}\otimes C_{m}} & \wedge_{0}^{1}V\otimes\wedge_{m-1}^{m-1}V \ar[d]|{\varphi_{0,m-1}^{1,m-1}} & V^{\vee} \ar[r]^-{1_{V^{\vee}}\otimes C_{m-1}} \ar[d]|{1_{V^{\vee}}\otimes C_{m}} & \wedge_{1}^{0}V\otimes\wedge_{m-1}^{m-1}V \ar[d]|{\varphi_{1,m-1}^{0,m-1}} \\ \wedge_{0}^{1}V\otimes\wedge_{m}^{m}V \ar[r]^-{\delta_{0,m}^{1,m}} & \wedge_{m-1}^{m}V\text{,} & \wedge_{1}^{0}V\otimes\wedge_{m}^{m}V \ar[r]^-{\delta_{1,m}^{0,m}} & \wedge_{m}^{m-1}V\text{.}}  \label{key claim D2}
\end{equation}

\bigskip

\textbf{Step 3}

Next we claim that the following diagram is commutative for every $m\geq 2$:%
\begin{equation}
\xymatrix@C=110pt{ \wedge_{1}^{1}V \ar[r]^-{\left(1_{\wedge_{1}^{1}V}\otimes C_{m-1},\left(1-m\right)\cdot1_{\wedge_{1}^{1}V}\otimes C_{m-2}\right)} \ar[d]_{1_{\wedge_{1}^{1}V}\otimes C_{m}} & \wedge_{1}^{1}V\otimes\wedge_{m-1}^{m-1}V\oplus\wedge_{1}^{1}V\otimes\wedge_{m-2}^{m-2}V \ar[d]^{\ev_{V}^{\tau}\otimes1_{\wedge_{m-1}^{m-1}V}\oplus\varphi_{1,m-2}^{1,m-2}} \\ \wedge_{1}^{1}V\otimes\wedge_{m}^{m}V \ar[r]^-{m\cdot\delta_{1,m}^{1,m}} & \wedge_{m-1}^{m-1}V\text{.} }  \label{key claim D3}
\end{equation}%
Consider the following diagram, where%
\begin{equation*}
f:=\delta _{1,0}^{0,1}\otimes 1_{\wedge _{m-1}^{m-1}V}\text{ and }g:=\left(
1-m\right) \cdot \left( 1_{\wedge _{0}^{1}V}\otimes \delta
_{1,m-1}^{0,m-1}\right) \circ \left( \tau _{\wedge _{1}^{0}V,\wedge
_{0}^{1}V}\otimes 1_{\wedge _{m-1}^{m-1}V}\right) \text{:}
\end{equation*}%
\begin{equation*}
\xymatrix{ \wedge_{1}^{1}V \ar@{}[dr]|{(\otimes)} \ar[r]^-{\tau_{V,V^{\vee}}} \ar[d]|{1_{V\otimes V^{\vee}}\otimes C_{m}} & V^{\vee}\otimes V \ar@{}[dr]|{(B)} \ar[r]^-{1_{V^{\vee}\otimes V}\otimes C_{m-1}} \ar[d]|{1_{V^{\vee}\otimes V}\otimes C_{m}} & \wedge_{1}^{0}V\otimes\wedge_{0}^{1}V\otimes\wedge_{m-1}^{m-1}V \ar@/^{0.75pc}/[dr]|{\left(f,g\right)} \ar[d]|{1_{V^{\vee}}\otimes\varphi_{0,m-1}^{1,m-1}} & \\ \wedge_{0}^{1}V\otimes\wedge_{1}^{0}V\otimes\wedge_{m}^{m}V \ar@/_{0.75pc}/[dr]|{\varphi_{0,1}^{1,0}\otimes1_{\wedge_{m}^{m}V}} \ar[r]^-{\tau_{V,V^{\vee}}\otimes1_{\wedge_{m}^{m}V}} & \wedge_{1}^{0}V\otimes\wedge_{0}^{1}V\otimes\wedge_{m}^{m}V \ar[r]^-{1_{V^{\vee}}\otimes\delta_{0,m}^{1,m}} & \wedge_{1}^{0}V\otimes\wedge_{m-1}^{m}V \ar@{}[dl]|{(A)} \ar[d]|{m\cdot\delta_{1,m-1}^{0,m}} & \wedge_{m-1}^{m-1}V\oplus\wedge_{0}^{1}V\otimes\wedge_{m-1}^{m-2}V \ar@{}[l]|{(C)} \ar@/^{0.75pc}/[dl]|{\varphi_{0,m-1}^{0,m-1}\oplus\varphi_{0,m-1}^{1,m-2}} \\ & \wedge_{1}^{1}V\otimes\wedge_{m}^{m}V \ar[r]^-{m\cdot\delta_{1,m}^{1,m}} & \wedge_{m-1}^{m-1}V\text{.} & }
\end{equation*}%
The region $\left( A\right) $ is commutative by Corollary \ref{S1 Algebras
C1} $\left( 2\right) $ with $i=l=1$, $j=k=0$ and $m=n$, the region $\left(
B\right) =1_{V^{\vee }}\otimes \left( \text{\ref{key claim D2}}\right) $ is
commutative by the commutativity of the first diagram in $\left( \text{\ref%
{key claim D2}}\right) $ and the functoriality of $\otimes $ and $\left(
C\right) $ is commutative by the second diagram of Corollary \ref%
{Alternating algebras C1} with $i=1$, $j=0$, $k=l=m-1$. Noticing that $\varphi
_{0,1}^{1,0}=1_{V\otimes V^{\vee }}$, $\varphi _{0,m-1}^{0,m-1}=1_{\wedge
_{m-1}^{m-1}V}$ and $\delta _{1,0}^{0,1}=\ev_{V}$, we deduce the equality%
\begin{equation}
m\cdot \delta _{1,m}^{1,m}\circ \left( 1_{V\otimes V^{\vee }}\otimes
C_{m}\right) =\left( 1_{\wedge _{m-1}^{m-1}V}\oplus 1_{\wedge
_{m-1}^{m-1}V}\right) \circ \left( a,b\right) =a+b\text{,}
\label{key claim D3 F}
\end{equation}%
where%
\begin{eqnarray*}
a &=&\left( \ev_{V}\otimes 1_{\wedge _{m-1}^{m-1}V}\right) \circ \left(
1_{V^{\vee }\otimes V}\otimes C_{m-1}\right) \circ \tau _{V,V^{\vee
}}=\left( \ev_{V}^{\tau }\otimes 1_{\wedge _{m-1}^{m-1}V}\right) \circ \left(
1_{\wedge _{1}^{1}V}\otimes C_{m-1}\right) \text{,} \\
b &=&\left( 1-m\right) \cdot \varphi _{0,m-1}^{1,m-2}\circ \left( 1_{\wedge
_{0}^{1}V}\otimes \delta _{1,m-1}^{0,m-1}\right) \circ \left( \tau _{\wedge
_{1}^{0}V,\wedge _{0}^{1}V}\otimes 1_{\wedge _{m-1}^{m-1}V}\right) \circ
\left( 1_{V^{\vee }\otimes V}\otimes C_{m-1}\right) \circ \tau _{V,V^{\vee }}
\\
&=&\left( 1-m\right) \cdot \varphi _{0,m-1}^{1,m-2}\circ \left( 1_{\wedge
_{0}^{1}V}\otimes \delta _{1,m-1}^{0,m-1}\right) \circ \left( 1_{V\otimes
V^{\vee }}\otimes C_{m-1}\right) \text{.}
\end{eqnarray*}%
Next we remark that, by the commutativity of $1_{V}\otimes \left( \text{\ref%
{key claim D2}}\right) $ (second diagram of $\left( \text{\ref{key claim D2}}%
\right) $\ with $m$ replaced by $m-1$), $\left( 1_{V}\otimes \delta
_{1,m-1}^{0,m-1}\right) \circ \left( 1_{V\otimes V^{\vee }}\otimes
C_{m-1}\right) =\left( 1_{V}\otimes \varphi _{1,m-2}^{0,m-2}\right) \circ
\left( 1_{V\otimes V^{\vee }}\otimes C_{m-2}\right) $ and that, by
definition of the multiplication in the mixed algebra, $\varphi
_{0,m-1}^{1,m-2}\circ \left( 1_{V}\otimes \varphi _{1,m-2}^{0,m-2}\right)
=\varphi _{1,m-2}^{1,m-2}$, so that%
\begin{equation}
b=\left( 1-m\right) \cdot \varphi _{1,m-2}^{1,m-2}\circ \left( 1_{V\otimes
V^{\vee }}\otimes C_{m-2}\right) \text{.}  \label{key claim D3 Fb}
\end{equation}%
Inserting $\left( \text{\ref{key claim D3 Fb}}\right) $ in $\left( \text{\ref%
{key claim D3 F}}\right) $ we find the claimed commutativity.

$\bigskip $

\textbf{Step 4}

We now claim that%
\begin{equation}
\xymatrix{ \mathbb{I} \ar@/^{0.75pc}/[dr]^-{\left(r-m+1\right)\cdot C_{m-1}} \ar[d]_{C_{1}\otimes C_{m}} & \\ \wedge_{1}^{1}V\otimes\wedge_{m}^{m}V \ar[r]^-{m\cdot\delta_{1,m}^{1,m}} & \wedge_{m-1}^{m-1}V\text{.} }  \label{key claim D4}
\end{equation}%
is commutative for $m\geq 1$, where $r:=\mathrm{rank}\left( V\right) $.
According to $\left( \text{\ref{key claim D3}}\right) $ we have, for $m\geq
2 $,%
\begin{eqnarray*}
m\cdot \delta _{1,m}^{1,m}\circ \left( C_{1}\otimes C_{m}\right) &=&m\cdot
\delta _{1,m}^{1,m}\circ \left( 1_{\wedge _{1}^{1}V}\otimes C_{m}\right)
\circ C_{1}= \\
&=&\left( 1_{\wedge _{m-1}^{m-1}V}\oplus 1_{\wedge _{m-1}^{m-1}V}\right)
\circ \left( a\circ C_{1},b\circ C_{1}\right)
\end{eqnarray*}%
where $a=\ev_{V}^{\tau }\otimes C_{m-1}$ and $b=\left( 1-m\right) \cdot
\varphi _{1,m-2}^{1,m-2}\circ \left( 1_{\wedge _{1}^{1}V}\otimes
C_{m-2}\right) $. We have%
\begin{equation*}
a\circ C_{1}=\left( \ev_{V}^{\tau }\otimes C_{m-1}\right) \circ
C_{1}=C_{m-1}\circ \ev_{V}^{\tau }\circ C_{1}=r\cdot C_{m-1}\text{,}
\end{equation*}%
because $r=\ev_{V}^{\tau }\circ C_{V}$. On the other hand, by Lemma \ref{key
claim L1},%
\begin{eqnarray*}
b\circ C_{1} &=&\left( 1-m\right) \cdot \varphi _{1,m-2}^{1,m-2}\circ \left(
C_{1}\otimes C_{m-2}\right) =\left( 1-m\right) \cdot \varphi
_{1,m-2}^{1,m-2}\circ \left( C_{m-2}\otimes C_{1}\right) \\
&=&\left( 1-m\right) \cdot C_{m-1}\text{.}
\end{eqnarray*}%
The claimed commutativity of $\left( \text{\ref{key claim D4}}\right) $
follows for $m\geq 2$. When $m=1$ we have, by definition, $\delta
_{1,1}^{1,1}=\left( \ev_{V}\otimes \ev_{V}^{\tau }\right) \circ \left( \tau
_{V,V^{\vee }\otimes V}\otimes 1_{V^{\vee }}\right) $, so that%
\begin{eqnarray*}
\delta _{1,1}^{1,1}\circ \left( C_{1}\otimes C_{1}\right) &=&\left(
\ev_{V}\otimes \ev_{V}^{\tau }\right) \circ \left( \tau _{V,V^{\vee }\otimes
V}\otimes 1_{V^{\vee }}\right) \circ \left( C_{V}\otimes C_{V}\right) \\
&=&\ev_{V}^{\tau }\circ \left( \ev_{V}\otimes 1_{V\otimes V^{\vee }}\right)
\circ \left( \tau _{V,V^{\vee }\otimes V}\otimes 1_{V^{\vee }}\right) \circ
\left( C_{V}\otimes 1_{V\otimes V^{\vee }}\right) \circ C_{V} \\
&=&\ev_{V}^{\tau }\circ \left( \left( \ev_{V}\otimes 1_{V}\right) \circ \tau
_{V,V^{\vee }\otimes V}\circ \left( C_{V}\otimes 1_{V}\right) \right)
\otimes 1_{V^{\vee }}\circ C_{V} \\
&=&\ev_{V}^{\tau }\circ C_{V}=r\text{,}
\end{eqnarray*}%
because $\left( \ev_{V}\otimes 1_{V}\right) \circ \tau _{V,V^{\vee }\otimes
V}\circ \left( C_{V}\otimes 1_{V}\right) =\left( 1_{V}\otimes \ev_{V}\right)
\circ \left( C_{V}\otimes 1_{V}\right) =1_{V}$ by Lemma \ref{S1 Casimir L
Properties} $\left( 2\right) $.

\bigskip

\textbf{Step 5}

We can now prove that, for $0\leq k\leq m$, we have%
\begin{equation}
\xymatrix{ \mathbb{I} \ar@/^{0.75pc}/[dr]^-{\binom{r+k-m}{k}\cdot C_{m-k}} \ar[d]_{C_{k}\otimes C_{m}} & \\ \wedge_{k}^{k}V\otimes\wedge_{m}^{m}V \ar[r]^-{\binom{m}{k}\cdot\delta_{k,m}^{k,m}} & \wedge_{m-k}^{m-k}V\text{.} }  \label{key claim D5}
\end{equation}%
When $k=0$ the claim is reduced to a triviality: we have $C_{k}\otimes
C_{m}=C_{m}$, $\binom{r+k-m}{k}C_{m-k}=C_{m}$ and $\binom{m}{k}\cdot \delta
_{k,m}^{k,m}=1_{\wedge _{m}^{m}V}$. In particular we may assume $1\leq k\leq
m$. For $k=1$ this is precisely $\left( \text{\ref{key claim D4}}\right) $,
so that we may assume that the commutativity is known for $1\leq k\leq m$
and that we would like to prove it for $2\leq k+1\leq m$. Consider the
following diagram%
\begin{equation*}
\xymatrix{ & \wedge_{1}^{1}V \ar[d]|{C_{k}\otimes1_{\wedge_{1}^{1}V}\otimes C_{m}} \ar@/^{0.75pc}/[dr]|{1_{\wedge_{1}^{1}V}\otimes C_{k}\otimes C_{m}} \ar@/^{1pc}/[drr]^{\binom{r+k-m}{k}\cdot1_{\wedge_{1}^{1}V}\otimes C_{m-k}} & & \\ \mathbb{I} \ar@/^{0.75pc}/[ur]|-{C_{1}} \ar[r]^-{C_{k}\otimes C_{1}\otimes C_{m}} \ar@/_{0.75pc}/[dr]|-{C_{k+1}\otimes C_{m}} & \wedge_{k}^{k}V\otimes\wedge_{1}^{1}V\otimes\wedge_{m}^{m}V \ar@{}[ur]|(0.4){(\tau)}  \ar@{}[ul]|(0.4){(\otimes)} \ar@{}[dl]|(0.4){(A)} \ar@{}[dr]|{(B)} \ar[d]|{\varphi_{k,1}^{k,1}\otimes1_{\wedge_{m}^{m}V}} \ar[r]^-{\tau_{\wedge_{k}^{k}V,\wedge_{1}^{1}V}\otimes1_{\wedge_{m}^{m}V}} & \wedge_{1}^{1}V\otimes\wedge_{k}^{k}V\otimes\wedge_{m}^{m}V \ar@{}[u]|(0.4){(C)} \ar[r]^-{\binom{m}{k}\cdot1_{\wedge_{1}^{1}V}\otimes\delta_{k,m}^{k,m}} & \wedge_{1}^{1}V\otimes\wedge_{m-k}^{m-k}V \ar@/^{0.75pc}/[dl]|-{\delta_{1,m-k}^{1,m-k}} \\ & \wedge_{k+1}^{k+1}V\otimes\wedge_{m}^{m}V \ar[r]^-{\binom{m}{k}\cdot\delta_{k+1,m}^{k+1,m}} & \wedge_{m-k-1}^{m-k-1}V\text{.} & }
\end{equation*}%
The region $\left( A\right) $ is commutative by Lemma \ref{key claim L1}:%
\begin{equation*}
\left( \varphi _{k,1}^{k,1}\otimes 1_{\wedge _{m}^{m}V}\right) \circ \left(
C_{k}\otimes C_{1}\otimes C_{m}\right) =\left( \varphi _{k,1}^{k,1}\circ
\left( C_{k}\otimes C_{1}\right) \right) \otimes C_{m}=C_{k+1}\otimes C_{m}%
\text{.}
\end{equation*}%
The region $\left( B\right) $ is commutative by Corollary \ref{S1 Algebras
C1} $\left( 2\right) $. Finally, the region $\left( C\right) =1_{\wedge
_{1}^{1}V}\otimes \left( \text{\ref{key claim D5}}\right) $ is commutative
by induction. We deduce%
\begin{equation}
\binom{m}{k}\cdot \delta _{k+1,m}^{k+1,m}\circ \left( C_{k+1}\otimes
C_{m}\right) =\binom{r+k-m}{k}\cdot \delta _{1,m-k}^{1,m-k}\circ \left(
C_{1}\otimes C_{m-k}\right) \text{.}  \label{key claim D5 F1}
\end{equation}%
We now note that we have $k+1\leq m$ if and only if $m-k\geq 1$, so that $%
\left( \text{\ref{key claim D4}}\right) $ with $m$ replaced by $m-k$ gives
the equality%
\begin{equation}
\left( m-k\right) \cdot \delta _{1,m-k}^{1,m-k}\circ \left( C_{1}\otimes
C_{m-k}\right) =\left( r-m+k+1\right) \cdot C_{m-k-1}\text{.}
\label{key claim D5 F2}
\end{equation}%
Noticing that $\binom{m}{k+1}=\frac{m-k}{k+1}\binom{m}{k}$ we deduce,
inserting $\left( \text{\ref{key claim D5 F2}}\right) $ in $\left( \text{\ref%
{key claim D5 F1}}\right) $, that we have%
\begin{equation*}
\binom{m}{k+1}\cdot \delta _{k+1,m}^{k+1,m}\circ \left( C_{k+1}\otimes
C_{m}\right) =\frac{1}{k+1}\binom{r+k-m}{k}\left( r-m+k+1\right) \cdot
C_{m-k-1}\text{.}
\end{equation*}%
The claim follows because $\binom{r+k+1-m}{k+1}=\frac{1}{k+1}\binom{r+k-m}{k}%
\left( r-m+k+1\right) $.

\section{A Poincar\'{e} duality isomorphism for the symmetric algebras}

In this section we suppose that $\mathcal{C}$ is rigid, $\mathbb{Q}$-linear and pseudo-abelian.
We consider an object $V\in \mathcal{C}$ and we apply the results on $\Delta $-graded algebras  with $A=\left( \vee
^{\cdot }V,\varphi _{i,j}^{V,s}\right) $ and $A^{\vee }=\left( \vee ^{\cdot
}V^{\vee },\varphi _{i,j}^{V^{\vee },s}\right) $. We will use the shorter notation
$i_{V}^{p}:=i_{V,s}^{p}$, $p_{V}^{p}:=p_{V,s}^{p}$, $e_{V}^{p}:=e_{V,s}^{p}$%
, $\varphi _{i,j}=\varphi _{i,j}^{V}:=\varphi _{i,j}^{V,s}$ and $\varphi
_{i,j}^{V^{\vee }}:=\varphi _{i,j}^{V^{\vee },s}$. The same argument employed in the alternating case shows that the internal
multiplication morphisms are given, for every $j\geq i$, by the composite%
\[
\varphi _{\iota _{i,j}}:\vee ^{i}V\otimes \vee ^{j}V^{\vee }\overset{%
i_{V}^{i}\otimes i_{V^{\vee }}^{j}}{\rightarrow }\left( \otimes ^{i}V\right)
\otimes \left( \otimes ^{j}V\right) \overset{\ev_{V}^{i,\tau }\otimes
1_{\otimes ^{j-i}V^{\vee }}}{\rightarrow }\otimes ^{j-i}V^{\vee }\overset{%
p_{V^{\vee }}^{j-i}}{\rightarrow }\vee ^{j-i}V^{\vee }\text{.}
\]%
These morphisms can then be lifted to the tensor algebras as in Lemma \ref%
{Alternating algebras L1}, the only difference being that the character $%
\varepsilon $ has to be replaced by the trivial character. The effect of
this change is that the resulting normalized family $\iota _{j}:=j\cdot
\iota _{1,j}$ is now a derivation, rather than being an anti-derivation, i.e.
it satisfies the symbolic theoretic formula%
\[
\iota _{j+l}\left( x\right) \left( \omega _{j}\vee \omega _{l}\right) =\iota
_{j+l}\left( x\right) \left( \omega _{j}\right) \wedge \omega _{l}+\omega
_{j}\wedge \iota _{j+l}\left( x\right) \left( \omega _{l}\right) \text{ for }%
x\in V\text{, }\omega _{j}\in \vee ^{j}V^{\vee }\text{ and }\omega _{l}\in
\vee ^{l}V^{\vee }\text{,}
\]%
which has a formal diagram theoretic formulation analogue to Proposition \ref%
{Alternating algebras P1}. Then the analogue of Corollary \ref{Alternating
algebras C1}, that we leave to the reader to precisely formulate, is just a
formal consequence and the proof of Lemma \ref{Alternating algebras L key},
suitable modified employing the analogue of this corollary, lead to the
following result.

\begin{lemma}
\label{Symmetric algebras L key}Let $r:=\mathrm{rank}\left( V\right) $ be
the rank of $V$, defined as the composite%
\begin{equation*}
r:\mathbb{I}\overset{C_{V}}{\rightarrow }V\otimes V^{\vee }\overset{%
\ev_{V}^{\tau }}{\rightarrow }\mathbb{I}\text{.}
\end{equation*}%
For every $g\geq i$ we have the equality%
\begin{equation*}
\binom{g}{i}^{-1}\binom{r+g-1}{i}\cdot C_{\vee ^{g-i}V}=\delta
_{i,g}^{i,g}\circ \left( C_{\vee ^{i}V}\otimes C_{\vee ^{g}V}\right) \text{,}
\end{equation*}%
where, for every $k\in \mathbb{N}_{\geq 1}$,%
\begin{equation*}
\binom{T}{k}:=\frac{1}{k!}T\left( T-1\right) ...\left( T-k+1\right) \in 
\mathbb{Q}\left[ T\right] \text{ and }\binom{T}{0}=1\text{.}
\end{equation*}
\end{lemma}

\bigskip

As in the alternating case we may define, for every $g\geq i$, the Poincar\'{e} morphisms%
\[
D^{i,g}:=D_{\iota _{i,g}}:\vee ^{i}V\rightarrow \vee ^{g-i}V^{\vee }\otimes
\vee ^{g}V^{\vee \vee }\text{ and }D^{i,g}:=D_{\iota _{i,g}^{\ast }}:\vee
^{i}V^{\vee }\rightarrow \vee ^{g-i}V\otimes \vee ^{g}V^{\vee }\text{.}
\]%
The following result is obtained from Lemma \ref{Symmetric algebras L key} in the same way as Theorem  \ref{Alternating algebras T} has been obtained from Lemma \ref{Alternating algebras L key}%
\ with .

\begin{theorem}
\label{Symmetric algebras T}The following diagrams are commutative, for
every $g\geq i\geq 0$.

\begin{enumerate}
\item[$\left( 1\right) $] 
\begin{equation*}
\xymatrix@C=60pt{ \vee^{i}V\otimes\vee^{i}V^{\vee} \ar@/^{0.75pc}/[dr]|{\binom{g}{g-i}^{-1}\binom{r+g-1}{g-i}\ev_{V,s}^{i,\tau}} \ar[d]|{D^{i,g}\otimes D_{i,g}} \\ \vee^{g-i}V^{\vee}\otimes\vee^{g}V^{\vee\vee}\otimes\vee^{g-i}V\otimes\vee^{g}V^{\vee} \ar[r]^-{\ev_{13,24}^{\phi,\phi}} & \mathbb{I}\text{.} }
\end{equation*}

\item[$\left( 2\right) $] 
\begin{equation*}
\xymatrix{ \vee^{i}V^{\vee}\otimes\vee^{g-i}V^{\vee} \ar[r]^-{1_{\vee^{i}V^{\vee}}\otimes D_{g-i,g}} \ar[d]|{D_{i,g}\otimes1_{\vee^{g-i}V^{\vee}}} & \vee^{i}V^{\vee}\otimes\vee^{i}V\otimes\vee^{g}V^{\vee} \ar[d]|{ \ev_{V,s}^{i}\otimes1_{\vee^{g}V^{\vee}}} & \vee^{i}V\otimes\vee^{g-i}V \ar[r]^-{1_{\vee^{i}V}\otimes D^{g-i,g}} \ar[d]|{D^{i,g}\otimes1_{\vee^{g-i}V}} & \vee^{i}V\otimes\vee^{i}V^{\vee}\otimes\vee^{g}V^{\vee\vee} \ar[d]|{\ev_{V,s}^{i,\tau}\otimes1_{\vee^{g}V^{\vee\vee}}} \\ \vee^{g-i}V\otimes\vee^{g}V^{\vee}\otimes\vee^{g-i}V^{\vee} \ar[r]^-{\ev_{13,\vee^{g}V^{\vee}}^{\tau}} & \vee^{g}V^{\vee}\text{,} & \vee^{g-i}V^{\vee}\otimes\vee^{g}V^{\vee\vee}\otimes\vee^{g-i}V \ar[r]^-{\ev_{13,\vee^{g}V^{\vee\vee}}^{\phi}} & \vee^{g}V^{\vee\vee}\text{.} }
\end{equation*}

\item[$\left( 3\right) $] 
\begin{equation*}
\xymatrix@C=40pt{ \vee^{i}V \ar@/^{2pc}/[rrr]^-{\binom{g}{g-i}^{-1}\binom{r+g-1}{g-i}} \ar[r]_-{D^{i,g}} & \vee^{g-i}V^{\vee}\otimes\vee^{g}V^{\vee\vee} \ar[r]_-{D_{g-i,g}\otimes1_{\vee^{g}V^{\vee\vee}}} & \vee^{i}V\otimes\vee^{g}V^{\vee}\otimes\vee^{g}V^{\vee\vee} \ar[r]_-{1_{\vee^{i}V}\otimes \ev_{V^{\vee},s}^{g,\tau}} & \vee^{i}V }
\end{equation*}%
and 
\begin{equation*}
\xymatrix@C=40pt{ \vee^{g-i}V^{\vee} \ar@/^{2pc}/[rrr]^-{\binom{g}{i}^{-1}\binom{r+g-1}{i}} \ar[r]_-{D_{g-i,g}} & \vee^{i}V\otimes\vee^{g}V^{\vee} \ar[r]_-{D^{i,g}\otimes1_{\vee^{g}V^{\vee}}} & \vee^{g-i}V^{\vee}\otimes\vee^{g}V^{\vee\vee}\otimes\vee^{g}V^{\vee} \ar[r]_-{1_{\vee^{g-i}V^{\vee}}\otimes \ev_{V^{\vee},s}^{g}} & \vee^{g-i}V^{\vee}\text{.} }
\end{equation*}

\item[$\left( 4\right) $] 
\begin{equation*}
\xymatrix{ \vee^{i}V\otimes\vee^{g-i}V \ar[r]^-{\varphi_{i,g-i}} \ar[d]|{D^{i,g}\otimes D^{g-i,g}} & \vee^{g}V \ar[d]|{\binom{g}{g-i}^{-1}\binom{r+g-1}{g-i}\cdot i_{\vee^{g}V}} & \vee^{i}V^{\vee}\otimes\vee^{g-i}V^{\vee} \ar[r]^-{\varphi_{i,g-i}} \ar[d]|{D_{i,g}\otimes D_{g-i,g}} & \vee^{g}V^{\vee} \ar[d]|{\binom{g}{g-i}^{-1}\binom{r+g-1}{g-i}} \\ \vee^{g-i}V^{\vee}\otimes\vee^{g}V^{\vee\vee}\otimes\vee^{i}V^{\vee}\otimes\vee^{g}V^{\vee\vee} \ar[r]^(0.8){\varphi_{g-i,i}^{13\rightarrow\vee^{g}V^{\vee\vee}}} & \vee^{g}V^{\vee\vee}\text{,} & \vee^{g-i}V\otimes\vee^{g}V^{\vee}\otimes\vee^{i}V\otimes\vee^{g}V^{\vee} \ar[r]^-(0.7){\varphi_{g-i,i}^{13\rightarrow\vee^{g}V^{\vee}}} & \vee^{g}V^{\vee}\text{.}}
\end{equation*}
\end{enumerate}
\end{theorem}

\bigskip

We say that $V$ has \emph{symmetric rank }$g\in \mathbb{N}_{\geq 1}$ if $\vee ^{g}V$
is an invertible object and $\binom{r+g-1}{g-i}$ and $\binom{r+g-1}{i}$ are
invertible for every $0\leq i\leq g$. For example, when $End\left( \mathbb{I}%
\right) $ is a field or $r\in \mathbb{Q}$, the second condition means that $%
r $ is not a root of the polynomials $\binom{T+g-1}{g-i}\in \mathbb{Q}\left[
T\right] $ and $\binom{T+g-1}{i}\in \mathbb{Q}\left[ T\right] $ for every $%
0\leq i\leq g$, i.e. that $r\neq 1-g,2-g,...,-i$ and $r\neq 1-g,2-g,...,i-g$
for every $1\leq i\leq g$.

We say that $V$ has \emph{strong symmetric rank }$g\in \mathbb{N}_{\geq 1}$ if $\vee
^{g}V$ is an invertible object and $r=-g$ (hence $V$ has symmetric rank $g$).
With these notations Corollary \ref{FDP C1} specializes to the following
result.

\begin{corollary}
\label{Symmetric algebras CT}If $V$ has symmetric rank $g\in \mathbb{N}$
then, for every $0\leq i\leq g$, the morphisms $D^{i,g}$, $D_{g-i,g}$, $%
D^{g-i,g}$ and $D_{i,g}$ are isomorphisms and the multiplication maps $%
\varphi _{i,g-i}^{V}$, $\varphi _{g-i,i}^{V}$, $\varphi _{i,g-i}^{V^{\vee }}$
and $\varphi _{g-i,i}^{V^{\vee }}$ are perfect pairings (meaning that the
associate $\hom $ valued morphisms are isomorphisms). Furthermore, when $V$
has strong symmetric rank $g$, we have $\binom{r+g-1}{g-i}=\left( -1\right)
^{g-i}$ and $\binom{r+g-1}{i}=\left( -1\right) ^{i}$ in the commutative
diagrams of Theorem \ref{Symmetric algebras T}.
\end{corollary}

We end this section with the analogue of Proposition \ref{Alternating
algebras P2} in this setting. This a technical result that will be crucial
for the computation of \cite{MS}. The proof is just a copy of that of
Proposition \ref{Alternating algebras P2}.

\begin{proposition}
\label{Symmetric algebras P2}The following diagrams are commutative when $%
\vee ^{g}V$ is invertible of rank $r_{\vee ^{g}V}$ (hence $r_{\vee
^{g}V}\in \left\{ \pm 1\right\} $):%
\begin{equation*}
\xymatrix{ \vee^{i}V\otimes\vee^{g-i}V\otimes V \ar[r]^-{\tau_{\vee^{i}V\otimes\vee^{g-i}V,V}} \ar[d]|{\left(1_{\vee^{i}V}\otimes\varphi_{g-i,1},\left(1_{\vee^{g-i}V}\otimes\varphi_{i,1}\right)\circ\left(\tau_{\vee^{i}V,\vee^{g-i}V}\otimes1_{V}\right)\right)} & V\otimes\vee^{i}V\otimes\vee^{g-i}V \ar[d]|{D^{1,g}\otimes\varphi_{i,g-i}} \\ \vee^{i}V\otimes\vee^{g-i+1}V\oplus\vee^{g-i}V\otimes\vee^{i+1}V \ar[d]|{D^{i,g}\otimes D^{g-i+1,g}\oplus D^{g-i,g}\otimes D^{i+1,g}} & \vee^{g-1}V^{\vee}\otimes\vee^{g}V^{\vee\vee}\otimes\vee^{g}V \ar[d]|{r_{\vee^{g}V}g\binom{g}{g-i}^{-1}\binom{r+g-1}{g-i}\cdot1_{\vee^{g-1}V^{\vee}\otimes\vee^{g}V^{\vee\vee}}\otimes i_{\vee^{g}V}} \\ M \ar[r]_-{i\cdot\varphi_{g-i,i-1}^{13}\oplus\left(g-i\right)\cdot\varphi_{i,g-i-1}^{13}} & \vee^{g-1}V^{\vee}\otimes\vee^{g}V^{\vee\vee}\otimes\vee^{g}V^{\vee\vee}}
\end{equation*}%
where
\[
M = \vee^{g-i}V^{\vee}\otimes\vee^{g}V^{\vee\vee}\otimes\vee^{i-1}V^{\vee}\otimes\vee^{g}V^{\vee\vee}\oplus\vee^{i}V^{\vee}\otimes\vee^{g}V^{\vee\vee}\otimes\vee^{g-i-1}V^{\vee}\otimes\vee^{g}V^{\vee\vee}
\]
and%
\begin{equation*}
\xymatrix{ \vee^{i}V^{\vee}\otimes\vee^{g-i}V^{\vee}\otimes V^{\vee} \ar[r]^-{\tau_{\vee^{i}V^{\vee}\otimes\vee^{g-i}V^{\vee},V^{\vee}}} \ar[d]|{\left(1_{\vee^{i}V^{\vee}}\otimes\varphi_{g-i,1},\left(1_{\vee^{g-i}V^{\vee}}\otimes\varphi_{i,1}\right)\circ\left(\tau_{\vee^{i}V^{\vee},\vee^{g-i}V^{\vee}}\otimes1_{V^{\vee}}\right)\right)} & V^{\vee}\otimes\vee^{i}V^{\vee}\otimes\vee^{g-i}V^{\vee} \ar[d]|{D_{1,g}\otimes\varphi_{i,g-i}} \\ \vee^{i}V^{\vee}\otimes\vee^{g-i+1}V^{\vee}\oplus\vee^{g-i}V^{\vee}\otimes\vee^{i+1}V^{\vee} \ar[d]|{D_{i,g}\otimes D_{g-i+1,g}\oplus D_{g-i,g}\otimes D_{i+1,g}} & \vee^{g-1}V\otimes\vee^{g}V^{\vee}\otimes\vee^{g}V^{\vee} \ar[d]|{r_{\vee^{g}V}g\binom{g}{g-i}^{-1}\binom{r+g-1}{g-i}\cdot1_{\vee^{g-1}V^{\vee}\otimes\vee^{g}V^{\vee}\otimes\vee^{g}V^{\vee}}} \\ \vee^{g-i}V\otimes\vee^{g}V^{\vee}\otimes\vee^{i-1}V\otimes\vee^{g}V^{\vee}\oplus\vee^{i}V\otimes\vee^{g}V^{\vee}\otimes\vee^{g-i-1}V\otimes\vee^{g}V^{\vee} \ar[r]_-{i\cdot\varphi_{g-i,i-1}^{13}\oplus\left(g-i\right)\cdot\varphi_{i,g-i-1}^{13}} & \vee^{g-1}V\otimes\vee^{g}V^{\vee}\otimes\vee^{g}V^{\vee}\text{.} }
\end{equation*}
\end{proposition}

\end{document}